%% file: main_arxiv.tex
\documentclass[12pt, letterpaper]{article}

\include{macros}

\usepackage{setspace}
\usepackage[authoryear]{natbib}
\RequirePackage[colorlinks,citecolor=blue,urlcolor=blue]{hyperref}

\usepackage{tikz,pgfplots}
\usetikzlibrary{arrows,decorations.pathreplacing}
\usepackage[bottom]{footmisc}
\usepackage{enumerate}

\numberwithin{equation}{section}
\theoremstyle{plain}
\newtheorem{thm}{Theorem}[section]

\def\SBM{\mbox{SBM}}

\begin{document}

\title{Entrywise Eigenvector Analysis of Random Matrices with Low Expected Rank}

\author{Emmanuel Abbe\thanks{Address: PACM and Department of EE, Princeton University, Princeton, NJ 08544, USA; E-mail: eabbe@princeton.edu. The research was supported by NSF CAREER Award CCF-1552131, ARO grant W911NF-16-1-0051, NSF CSOI CCF-0939370.}, 
	Jianqing Fan\thanks{Address: Department of ORFE, Sherrerd Hall, Princeton University, Princeton, NJ 08544, USA; E-mails: 
	\textit{\{jqfan, kaizheng, yiqiaoz\}@princeton.edu}.
	The research was supported by NSF grants DMS-1662139 and DMS-1712591, NIH grant R01-GM072611-11 and ONR grant N00014-19-1-2120.}, 
	Kaizheng Wang$^{\dag}$ and Yiqiao Zhong$^{\dag}$}
\date{}

\maketitle
\begin{abstract}
Recovering low-rank structures via eigenvector perturbation analysis is a common problem in statistical machine learning, such as in factor analysis, community detection, ranking, matrix completion, among others.  While a large variety of bounds are available for average errors between empirical and population statistics of eigenvectors, few results are tight for entrywise analyses, which are critical for a number of problems such as community detection.

This paper investigates entrywise behaviors of eigenvectors for a large class of random matrices whose expectations are low-rank, which helps settle the conjecture in \cite{abh_arxiv} that the spectral algorithm achieves exact recovery in the stochastic block model without any trimming or cleaning steps. The key is a first-order approximation of eigenvectors under the $\ell_\infty$ norm:
$$u_k \approx \frac{A u_k^*}{\lambda_k^*},$$
where $\{u_k\}$ and $\{u_k^*\}$ are eigenvectors of a random matrix $A$ and its expectation $\E A$, respectively. The fact that the approximation is both tight and linear in $A$ facilitates sharp comparisons between $u_k$ and $u_k^*$. In particular, it allows for comparing the signs of $u_k$ and $u_k^*$ even if $\| u_k - u_k^*\|_{\infty}$ is large. The results are further extended to perturbations of eigenspaces, yielding new $\ell_\infty$-type bounds for synchronization ($\mathbb{Z}_2$-spiked Wigner model) and noisy matrix completion.

\end{abstract}

\sloppy

\section{Introduction} \label{sec:intro}

Many estimation problems in statistics involve low-rank matrix estimators that are NP-hard to compute, and many of these estimators are solutions to nonconvex programs. This is partly because of the widespread use of maximum likelihood estimation (MLE) which, while enjoying good statistical properties, often poses computational challenges due to nonconvex or discrete constraints inherent in the problems. 

Fortunately, computationally efficient algorithms using eigenvectors often afford good performance. The eigenvectors either directly lead to final estimates \citep{SMa00, NJW02}, or serve as warm starts followed by further refinements \citep{KesMonOh10, JaiNetSan13, CanLiSol15}. Such algorithms mostly rely on computation of leading eigenvectors and matrix-vector multiplications, which are easily implemented.

While various heuristics abound, theoretical understanding remains scarce on the entrywise analysis, and on when refinements are needed or can be avoided. In particular, it remains open in various cases to determine whether a vanilla eigenvector-based method without preprocessing steps (e.g., trimming of outliers) or without refinement steps (e.g., cleaning with local improvements) enjoys the same optimality results as the MLE (or SDP) does. A crucial missing step is a sharp entrywise perturbation analysis of eigenvectors. This is party because the $\ell_\infty$ distance between the eigenvectors of a random matrix and their expected counterparts may not be the correct quantity to look at; errors per entry can be asymmetrically distributed, as we shall see in this paper.

This paper investigates entrywise behaviors of eigenvectors and more generally, eigenspaces, for random matrices with low expected rank using the following approach.
Let $A$ be a random matrix, $A^*=\E A$, and $E=A-A^*$ be the `error' of $A$.
In many cases, $A^*$ is a symmetric matrix with low rank determined by the structure of a statistical problem,
such as low-rank with blocks in community detection.

Consider for now the case of symmetric $A$, and let $u_k$, resp.\ $u_k^*$, be the eigenvector corresponding to the $k$-th largest eigenvalue of $A$, resp.\ $A^*$. Roughly speaking, if $E$ is moderate, our first-order approximation reads
\begin{equation*}
u_k = \frac{A u_k}{\lambda_k} \approx \frac{A u_k^*}{\lambda_k^*} = u_k^* + \frac{E u_k^*}{\lambda_k^*}.
\end{equation*}
While $u_k$ is a nonlinear function of $A$ (or equivalently $E$), the approximation is linear in $A$, which greatly facilitates the analysis. Under certain conditions, the maximum entrywise approximation error $ \| u_k - A u_k^* / \lambda_k^* \|_{\infty} $ can be much smaller than $ \| u^*_k\|_{\infty}$, allowing us to study $u_k$ through $A u_k^* / \lambda_k^*$. To obtain such results, a key part in our theory is to characterize concentration properties of $A$ and structural assumptions on its expectation $A^*$.

This perturbation analysis leads to new and sharp theoretical guarantees. In particular, we find that for the exact recovery problem in stochastic block model, the vanilla spectral algorithm (without trimming or cleaning) achieves the information-theoretic limit, and it coincides with the MLE estimator {\it whenever} the latter succeeds.
This settles in particular a conjecture left open in \cite{abh_arxiv,ABH16}.
Therefore, MLE and SDP do not have advantage over the spectral method in terms of exact recovery, if the model is correct. SDP may be preferred in some applications for its robustness and optimality certificates, but that is beyond the scope of this paper.

\subsection{A sample problem}\label{sec:sample}
Let us consider a network model that has received widespread interest in recent years: the stochastic block model (SBM). Suppose that we have a graph with vertex set $\{1,2,\cdots,n\}$, and assume for simplicity that $n$ is even. There is an unknown index set $J \in \{1,2,\cdots,n\}$ with $|J| = n/2$ such that the vertex set is partitioned into two groups $J$ and $J^c$. Within groups, there is an edge between each pair of vertices with probability $p$, and between groups, there is an edge with probability $q$. Let $x \in \R^n$ be the group membership vector with $x_i=1$ if $i\in J$ and $x_i = -1$ otherwise. The goal is to recover $x$ from the observed edges of the graph.

This random-graph-based model was first proposed for social relationship networks \citep{HLL83}, and many more realistic models have been developed based on the SBM since then. Given its fundamental importance, there are a plurality of papers addressing statistical properties and algorithmic efficiencies; 
see \cite{Abb17} for a survey.

Under the regime $p=\frac{a \log n}{n}$, $q = \frac{b \log n}{n}$ where $a > b > 0$ are constants, \cite{ABH16} and \cite{mossel_consist} proved that exact recovery is possible if and only if $\sqrt{a} - \sqrt{b} > \sqrt{2}$, and that the limit can be achieved by efficient algorithms. They used two-round procedures (with a clean-up phase) to achieve the threshold.  Semidefinite relaxations are also known to achieve the threshold \citep{ABH16,HWX16,afonson,afonso_single}, as well as spectral methods with local refinements \citep{ASa15,prout2,Gao15}. We will discuss more in Sections \ref{sec:related} and \ref{sec:sbm}.

While existing works tackle exact recovery rather successfully, some fundamental questions remain unsolved: how do the simple statistics---top eigenvectors of the adjacency matrix---behave? Are they informative enough to reveal the group structure under very challenging regimes?

\begin{figure}[h!]
	\centering
	\includegraphics[scale=0.425]{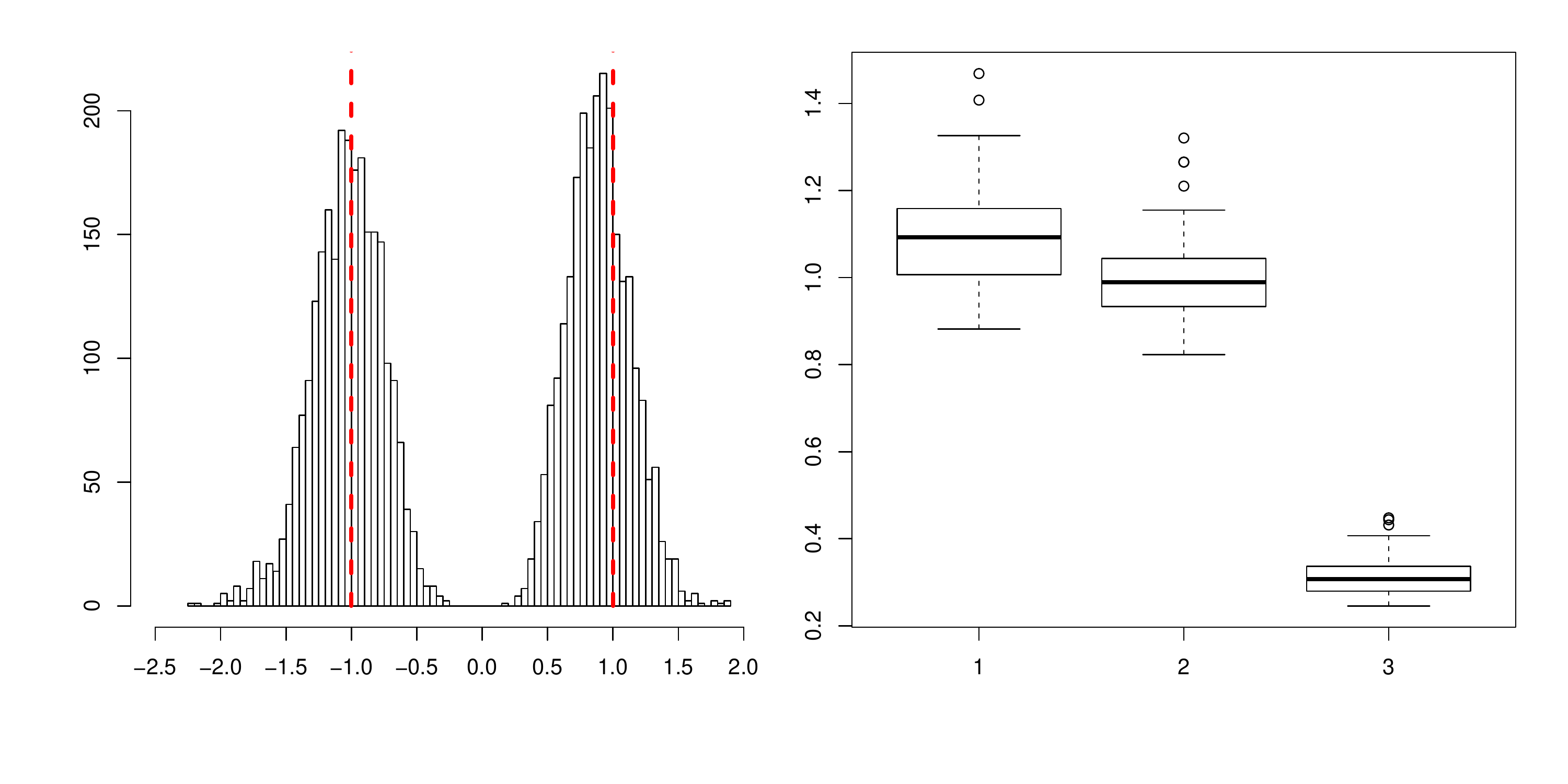}
	\caption{The second eigenvector and its first-order approximation in SBM. \textbf{Left:} The histogram of coordinates of $\sqrt{n} \,u_2$ computed from a single realization of adjacency matrix $A$, where $n$ is $5000$, $a$ is $4.5$ and $b$ is $0.25$. Exact recovery is expected as coordinates form two well-separated clusters. \textbf{Right:} boxplots showing three different distance/errors (up to sign) over $100$ realizations: (i) $ \sqrt{n}\,  \| u_2 - u_2^* \|_{\infty}$, (ii) $\sqrt{n}\,  \| Au_2^* / \lambda_2^* - u_2^*  \|_{\infty}$, (iii) $\sqrt{n}\,  \|   u_2 - Au_2^* / \lambda_2^*  \|_{\infty}$. $Au_2^* / \lambda_2^*$ is a good approximation of $u_2$ under $\ell_{\infty}$ norm even though $\| u_2 - u_2^* \|_{\infty}$ may be large.  }\label{fig:intro}
\end{figure}

To study these questions, we start with the eigenvectors of $A^* = \E A$. By definition, $A_{ij}$ is a Bernoulli random variable, and $\P(A_{ij} = 1)$ depends on whenever $i$ and $j$ are from the same groups. The expectation $\E A$ must be a block matrix of the following form:
\begin{equation*}
\E A =  \frac{ \log n}{n}
\left(
\begin{array}{cc}
a \mathbf{1}_{ \frac{n}{2} \times \frac{n}{2}} & b  \mathbf{1}_{ \frac{n}{2} \times \frac{n}{2}} \\
b \mathbf{1}_{ \frac{n}{2} \times \frac{n}{2}} & a  \mathbf{1}_{ \frac{n}{2} \times \frac{n}{2}}
\end{array}
\right),
\end{equation*}
where $\mathbf{1}_{m\times m}$ is the $m\times m$ all-one matrix. Here, for convenience, we represent $\E A$ as if $J = \{1,2,\cdots,n/2\}$. But in general $J$ is unknown, and there is a permutation of indices $\{1,\cdots,n\}$ in the matrix representation.

From the matrix representation it is clear that $\E A$ has rank $2$, with two nonzero eigenvalues $\lambda_1^* = \frac{a+b}{2}\log n$ and $\lambda_2^* = \frac{a-b}{2}\log n$. Simple calculations give the corresponding (normalized) eigenvectors: $u_1^* = \frac{1}{\sqrt{n}} \mathbf{1}_n$, and $( u_2^* )_i = 1/\sqrt{n}$ if $i \in J$ and $( u_2^*)_i = -1/\sqrt{n}$ if $i \in J^c$. Since $u_2^*$ perfectly aligns with the group assignment vector $x$, we hope to show its counterpart $u_2$, i.e., the second eigenvector of $A$, also has desirable properties.

The first reassuring fact is that, the top eigenvalues preserve proper ordering: by Weyl's inequality, the deviation of any eigenvalue $\lambda_i$ ($i \in [n]$) from $\lambda_i^*$ is bounded by $\| A - A^*\|_2$, which is $O(\sqrt{\log n})$ with high probability; see supplementary materials \citep{supp}. The Davis-Kahan $\sin\Theta$ theorem asserts that $u_1$ and $u_2$ are weakly consistent estimators for $u_1^*$ and $u_2^*$ respectively, in the sense that $|\langle u_k, u_k^*\rangle | \xrightarrow{\P } 1$ for $k=1,2$. However, this is not helpful for understanding their entrywise behaviors in the uniform sense, which is crucial for exact recovery. Nor can it explain the sharp phase transition phenomenon. This makes entrywise analysis both interesting and challenging.

This problem motivates some simulations about the coordinates of top eigenvectors of $A$. In Figure~\ref{fig:intro}, we calculate the rescaled second eigenvector $\sqrt{n} u_2$ of one typical realization $A$, and make a histogram plot of its coordinates. (Note the first eigenvector is aligned with the all-one vector $\mathbf{1}_n$, which is uninformative.) The parameters we choose are $n=5000$, $a=4.5$ and $b=0.25$, for which exact recovery is possible with high probability. Visibly, the coordinates of $\sqrt{n}u_2$ form two clusters around $\pm 1$ which, marked by red dashed lines, are coordinates of $\sqrt{n}\, u_2^*$. Intuitively, the signs of the former should suffice to reveal the group structure.

To probe into the second eigenvector $u_2$, we expand the perturbation $u_2 - u_2^*$ as follows:
\begin{equation}
u_2 - u_2^* = \left(  \frac{Au_2^*}{\lambda_2^*} - u_2^*  \right) + \left(  u_2 - \frac{Au_2^*}{\lambda_2^*}  \right). \label{all-errors}
\end{equation}
The first term is exactly $E u_2^*/\lambda_2^*$, which is linear in $E$ and can be viewed as the first-order perturbation. The second term is nonlinear in general, representing the error of higher order. Figure~\ref{fig:intro} shows boxplots of the infinity norm of rescaled perturbation errors over $100$ realizations (see (i)-(iii)), which illustrates that $\| u_2 - Au_2^* / \lambda_2^*  \|_{\infty}$ is much smaller than $\| u_2 - u_2^*  \|_{\infty}$ and $\| Au_2^* / \lambda_2^* - u_2^*  \|_{\infty}$. Indeed, we will see in Theorem~\ref{thm:simple} that
\begin{equation}\label{ineq:1}
\left\| u_2 - Au_2^* / \lambda_2^*  \right\|_{\infty} = o_{\P} \left(
\min_i |(u_2^*)_i| \right)= o_{\P} \left(1/\sqrt{n} \right).
\end{equation}
The result holds `up to sign', i.e. can choose an appropriate sign for the eigenvector $u_2$ as it is not uniquely defined; see Theorem~\ref{thm:simple} for its precise meaning. Therefore, the entrywise behavior of $u_2 - u_2^*$ is captured by its first-order term, which is much more amenable to analysis. This observation will finally lead to sharp eigenvector results in Section~\ref{sec:sbm}.

We remark that it is also possible to study the top eigenvector (denoted as $\bar{u}$) of the centered adjacency matrix $\bar{A} = A - \frac{\hat{d}}{n} \mathbf{1}_n \mathbf{1}_n^T$, where $\hat{d} = \sum_{i,j} A_{ij} / n$ is the average degree of all vertices. The top eigenvector of $\E \bar{A}$ is exactly $u_2^*$, and its empirical counterpart $\bar{u}$ is very similar to $u_2$. In fact, the same reasoning and analysis applies to $\bar{u}$, and one obtains similar plots as Figure~\ref{fig:intro} (omitted here).

\subsection{First-order approximation of eigenvectors}\label{sec:eigenvec}

Now we present a simpler version of our result that justifies the intuitions above. Consider a general symmetric random matrix (more precisely, this should be a sequence of random matrices with growing dimensions) $A \in \R^{n \times n}$
with independent entries on and above its diagonal. Suppose its expectation $A^* = \E A \in \R^{n \times n}$ is low-rank and has $r$ nonzero eigenvalues. Let us assume that
\begin{itemize}
\item[]
\textbf{(a)}\label{(a)}
 $r = O(1)$, these $r$ eigenvalues are positive and in descending order ($\lambda_1^* \ge \lambda_2^* \ge \cdots \ge \lambda_r^* > 0$), and $\lambda_1^* \asymp \lambda_r^*$.
\end{itemize}
Their corresponding eigenvectors are denoted by $u_1^*,\cdots, u_r^* \in \R^n$. In other words, we have spectral decomposition $A^* = \sum_{j=1}^r \lambda_j^* u_j^* (u_j^*)^T$.

We fix $k\in[r]$ and study the $k$-th eigenvector $u_k$. Define the eigen-gap (or spectral gap) as $\Delta^* = \min\{ \lambda_{k-1}^* - \lambda_{k}^*, \lambda_{k}^* - \lambda_{k+1}^*\}$, where we adopt the convention $\lambda_{0}^* = +\infty$ and $\lambda_{n+1}^* = -\infty$. Assume that
\begin{itemize}
	\item[]
\textbf{(b)}\label{(b)} $A$ \textit{concentrates} under the spectral norm, i.e., there is a suitable $\gamma = \gamma_n = o(1)$ such that $\| A - A^* \|_2 \le \gamma \Delta^*$ holds with probability $ 1-o(1)$.
\end{itemize}
A direct yet important implication is that, the fluctuation of $\lambda_k$ is much smaller than the gap $\Delta^*$, since Weyl's inequality forces $| \lambda_k - \lambda_k^* | \le \| A - A^* \|_2$. Thus, $\lambda_k$ is well separated from other eigenvalues, including the `bulk' $n-r$ eigenvalues whose magnitudes are at most $\| E \|_2$.

In addition, we assume that $A$ \textit{concentrates in a row-wise sense}:
\begin{itemize}
	\item[]
\textbf{(c)}\label{(c)} there exists a continuous non-decreasing function $\varphi: \R_+ \to \R_+$ that possibly depends on $n$, such that $\varphi(0)=0$, $\varphi(x)/x$ is non-increasing, and that for any $m \in [n], w \in \R^n$, with probability $1-o(n^{-1})$,
\begin{equation*}
|(A-A^*)_{m\cdot} w| \le \Delta^* \| w \|_{\infty}  \,
\varphi  \Big( \frac{\|w\|_2}{\sqrt{n} \| w \|_{\infty} } \Big) .
\end{equation*}
\end{itemize}
Here, the notation $(A-A^*)_{m\cdot}$ means the $m$-th row vector of $A-A^*$.

For the Gaussian case where $A_{ij} \sim N(A_{ij}^*,\sigma^2)$, we can simply choose a linear function $\varphi(x) = c(\Delta^*)^{-1}\sigma \sqrt{n \log n}\,x$ where $c>0$ is some proper constant. The condition then reads
\begin{equation*}
\P \left( |(A-A^*)_{m\cdot} w| \le  c \sigma \sqrt{\log n} \|w\|_2 \right) = 1 - o(n^{-1}),
\end{equation*}
which directly follows from Gaussian tail bound since $(A-A^*)_{m\cdot} w \sim N(0 , \sigma^2 \|w\|_2^2 )$. The tail of $(A-A^*)_{m\cdot} w$ is completely determined by $\|w\|_2$. For Bernoulli variables, we will use Bernstein-type inequalities to study $(A-A^*)_{m\cdot} w$, which will inevitably involve both $\|w\|_2$ and $\|w\|_{\infty}$. Hence the function $\varphi(x)$ can no longer be linear. It turns out that $\varphi(x) \propto (1 \vee \log(1/x))^{-1}$, shown in Figure~\ref{fig:phi}, is a suitable choice. More details can be found in Section~\ref{regularity-conditions} and the supplementary material \cite{supp}. In both cases we have $\varphi(1)=O(1)$ under suitable signal-to-noise conditions.

\begin{figure}[h]
\centering
\begin{tikzpicture}
      \draw[->] (0,0) -- (5,0) node[right] {$x$};
      \draw[->] (0,0) -- (0,3) node[above] {$\varphi(x)$};
      \draw[scale=5, thick, domain=0:1,smooth,variable=\x,blue] plot ({\x},{0.6*\x}) node[right, black] { \footnotesize Gaussian};
      \draw[scale=5, thick, domain=0.0001:0.3679,smooth,variable=\y,blue]  plot ({\y},{0.5/ln(1/\y)}) node[above, black] { \footnotesize Bernoulli};
      \draw[scale=5, thick, domain=0:0.0001,smooth,variable=\y,blue]  plot ({\y},{542.8681*\y});
      \draw[scale=5, thick, domain=0.3679:1,smooth,variable=\y,blue]  plot ({\y},{0.5});
    \end{tikzpicture}
\caption{Typical choices of $\varphi$ for Gaussian noise and Bernoulli noise.}\label{fig:phi}
\end{figure}
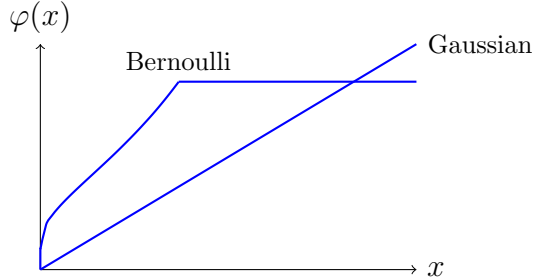


\begin{thm}[Simpler form of Theorem \ref{thm-main}] \label{thm:simple}
Let $k \in [r] = \{1,2,\cdots, r\}$ be fixed. Suppose that Assumptions (\hyperref[(a)]{a}), (\hyperref[(b)]{b}) and (\hyperref[(c)]{c}) hold, and $\| u_k^* \|_{\infty} \le \gamma$. Then, with probability $1-o(1)$,
\begin{equation}\label{ineq:2}
\min_{s \in \{ \pm 1 \}} \|u_k - s Au_k^*/\lambda_k^* \|_{\infty} = O\big( (\gamma+\varphi(\gamma)) \| u_k^*\|_{\infty}  \big) = o\big( \| u_k^*\|_{\infty}\big),
\end{equation}
where the notations $O(\cdot)$ and $o(\cdot)$ hide dependencies on $\varphi(1)$.
\end{thm}
On the left-hand side, we are allowed to choose a suitable sign $s$ as eigenvectors are not uniquely defined. The second bound is a consequence of the first one, since $\gamma = o(1)$ and $\lim_{\gamma \to 0} \varphi(\gamma) = 0$ by continuity. We hide dependency on $\varphi(1)$ in the above bound, since $\varphi(1)$ is bounded by a constant under suitable signal-to-noise ratio. More details can be found in Theorem \ref{thm-main}. Therefore, the approximation error $\| u_k - Au_k^*/\lambda_k^* \|_{\infty}$ is much smaller than $\| u_k^*\|_{\infty}$. This rigorously confirms the intuitions in Section~\ref{sec:sample}.

Here are some remarks. (1) This theorem enables us to study $u_k$ via its linearization $Au_k^*/\lambda_k^*$, since the approximation error is usually small order-wise. (2) The conditions of the theorem are fairly mild. For SBM, the theorem is applicable as long as we are in the $\frac{\log n}{n}$ regime ($p=a\frac{\log n}{n}$ and $q=b\frac{\log n}{n}$), regardless of the relative sizes of $a$ and $b$.

\subsection{MLE, spectral algorithm, and strong consistency}

Once we obtain the approximation result \eqref{ineq:2}, the analysis of entrywise behavior of eigenvectors boils down to that of $Au_k^*/\lambda_k^*$. In the SBM example, suppose we have \eqref{ineq:1} and with probability $1-o(1)$, $\sgn(Au_2^*/\lambda_2^*) = \sgn(u_2^*)$ and all the entries of $Au_2^*/\lambda_2^*$ are bounded away from zero by an order of $1/\sqrt{n}$. Then $\sgn(u_2) = \sgn(Au_2^*/\lambda_2^*)$ holds with probability $1-o(1)$. Here $\sgn(\cdot)$ denotes the entrywise sign function.
The eigenvector-based estimator $\sgn(u_2)$ for block membership can be conveniently analyzed through $Au_2^*/\lambda_2^*$, whose entries are just linear combinations of Bernoulli variables.

We remark on a subtlety of our result: our central analysis is a good control of $\|u_k - Au_k^*/\lambda_k^*\|_{\infty}$, not necessarily of $\|u_k - u_k^*\|_{\infty}$. For example, in SBM, an inequality such as $\|u_2 - u_2^*\|_{\infty} < \|u_2^*\|_{\infty}$ is not true in general. In Figure~\ref{fig:intro}, the second boxplot shows that $\sqrt{n}\, \| u_2 - u_2^* \|_{\infty}$ may well exceed $1$ even if $\sgn(u_2) = \sgn(u_2^*)$. This suggests that the distributions of the coordinates of the two clusters, though well separated, have asymmetric tails. Our Theorem~\ref{thm::sbm-lowerbound} asserts that it is in vain to seek a good bound for $\|u_2 - u_2^*\|_{\infty}$. Instead, one should resort to the central quantity $Au_2^*/\lambda_2^*$. This may partly explain why the conjecture has remained open for long.

The vector $Au_k^*/\lambda_k^*$ also plays a pivotal role in the information-theoretic lower bound for exact recovery in SBM, established in \cite{ABH16}. It is necessary to ensure $ ( Au_2^*/\lambda_2^*)_i > 0, \forall i \in J$ to hold with probability at least $1/3$. Otherwise, by symmetry and the union bound, with probability at least $1/3$ we can find some $i \in J$ and $i' \in J^c$ with $(Au_2^*/\lambda_2^*)_i < 0$ and $( Au_2^*/\lambda_2^* )_{i'} > 0$. Elementary calculation shows that in that case, a swap of group assignments of $i$ and $i'$ increases the likelihood. Thus the MLE $\hat x_{\mbox{\scriptsize MLE}}$ fails to exactly recover $J$. With a uniform prior on group assignments, the MLE is equivalent to the maximum \textit{a posteriori} estimator, which is optimal for exact recovery. Therefore, we must eliminate such local refinements to make exact recovery possible. This forms the core argument in \cite{ABH16}. The analysis above suggests an interesting property about the eigenvector-based estimator $\hat x_{eig}(A) := \sgn(u_2)$:

\begin{cor}\label{cor-intro}
Suppose we are given $a>b>0$ such that $\sqrt{a} \neq \sqrt{b} + \sqrt{2}$, i.e., we exclude the regime where $(a,b)$ is at the boundary of the phase transition. Then, whenever the MLE is successful, in the sense that $\hat x_{\mbox{\scriptsize MLE}} = x$ (up to sign) with probability $1-o(1)$, we have
\begin{equation*}
	\hat x_{eig}(A) = \hat x_{\mbox{\scriptsize MLE}}(A) = x
\end{equation*}
with probability $1-o(1)$. Here $x$ is the signed indicator of true communities.
\end{cor}
This is because the success of $\hat x_{\mbox{\scriptsize MLE}} $ hinges on $\sgn(Au_2^*/\lambda_2^*) = \sgn(u_2^*)$, which also guarantees $\hat x_{eig}$ to work. See Section~\ref{sec:sbm} for details. Such phenomenon appears in two applications considered in this paper.

\subsection{An iterative perspective: power iterations}

In the SBM, a key observation is that $ \| u_2 - A u_2^* / \lambda_2 ^* \|_{\infty} $ is small. Here we give some intuitions from an iterative (or algorithmic) perspective. For simplicity, we will focus on the top eigenvector $\bar{u}$ of the centered adjacency matrix $\bar{A} = A - \frac{\hat{d}}{n} \mathbf{1}_n \mathbf{1}_n^T$.

It is well known that the top eigenvector of a symmetric matrix can be computed via the power method. For almost any possible initialization $u^0$, the iterations $u^{t+1} = \bar{A} u^t /  \| \bar{A} u^t \|_2$ converge to $\bar{u}$. Suppose we set $u^0 = u_2^*$, the top eigenvector of $\E \bar{A}$. Although this is not a real algorithm due to the initialization, it helps us gain theoretical insights.

The first iterate after initialization is $u^1 = \bar{A} u_2^* /  \| \bar{A} u_2^* \|_2$. Standard concentration inequalities show that $\| \bar{A} u_2^* \|_2\approx \bar{\lambda}^*$, the top eigenvalue of $\E \bar{A}$. Therefore, $u^1$ is approximately $ \bar{A} u_2^*/\bar{\lambda}^*$, which coincides with our first-order approximation. If $u^t$ converges to $\bar{u}$ sufficiently fast, $u^1$ can already be good enough. This is similar to the rationale of one-step estimator \citep{Bic75}: a single, carefully designed iterate may improve the precision of a good initialization to the desired level. Figure \ref{fig:power} helps illustrate this idea.

\begin{figure}[h!]
	\centering
			\includegraphics[height=0.35\textwidth]{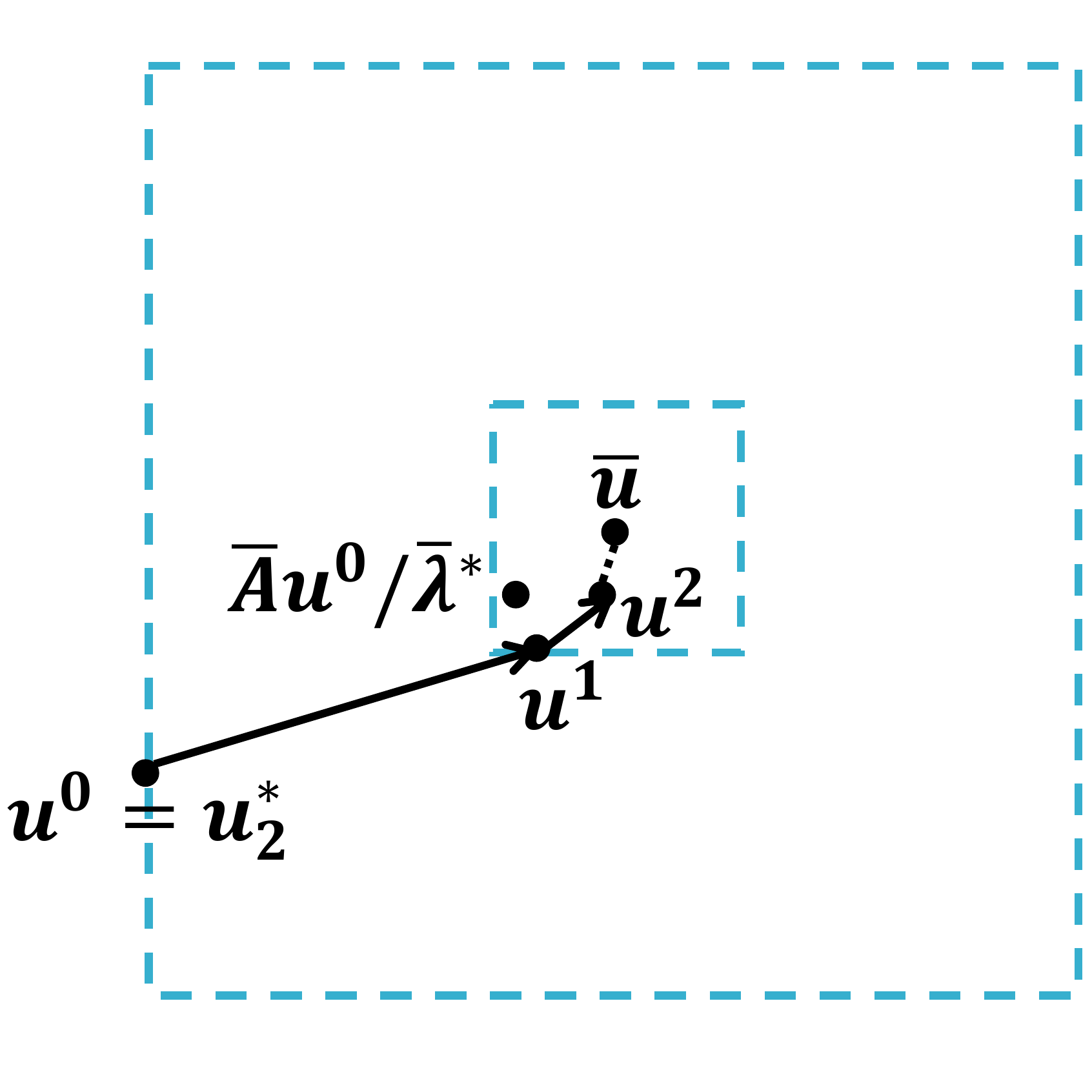}
	\caption{Error decay in power iterations. The larger and smaller squares represent $\ell_{\infty}$ balls centered at $\bar{u}$ with radii $\|u^0-\bar{u}\|_{\infty}$ and $\|u^1-\bar{u}\|_{\infty}$, respectively.
		}\label{fig:power}
\end{figure}

The iterative perspective has been explored in recent works \citep{Zho17, ZhoBou17}, where the latter studied both the eigenvector estimator and the MLE of a nonconvex problem. We are not going to show any proof with iterations or induction. Instead, we resort to the Davis-Kahan $\sin\Theta$ theorem, combined with a ``leave-one-out'' technique. Nevertheless, we believe the iterative perspective is helpful to many other (nonconvex) problems where a counterpart of Davis-Kahan theorem is absent.

\subsection{Related works}\label{sec:related}

The study of eigenvector perturbation dates back to Rayleigh \citep{rayleigh1896theory} and Schr\"{o}dinger \citep{schrodinger1926quantisierung}, in which asymptotic expansions were obtained. Later, \cite{DavKah70} developed elegant nonasymptotic perturbation bounds for eigenspaces gauged by unitary-invariant norms. 
These were extended to general rectangular matrices in \cite{Wed72}. See \cite{SSu90} for a comprehensive investigation. 
Recently, \cite{OVW18} showed significant improvements of classical, deterministic bounds when the perturbation is random.
Norms that depend on the choice of basis, such as the $\ell_\infty$ norm, are not addressed in these works but are of great interest in statistics.

There are several recent papers related to the study of entrywise perturbation. \cite{FanWanZho16} obtained $\ell_\infty$ eigenvector perturbation bounds. Their results were improved by \cite{CapTanPri17}, in which the authors focused on $2 \to \infty$ norm bounds for eigenspaces. \cite{EldBelWan17} developed an $\ell_\infty$ perturbation bound by expanding the eigenvector perturbation into infinite series.
These results are deterministic by nature, and thus yield suboptimal bounds under challenging stochastic regimes with small signal-to-noise ratio. By taking advantage of randomness, \cite{KLo16} and \cite{KolXia16} studied bilinear forms of singular vectors, leading to a sharp bound on $\ell_{\infty}$ error that was later extended to tensors \citep{XiaZho17}. \cite{Zho17} characterized entrywise behaviors of eigenvectors and explored their connections with Rayleigh-Schr\"{o}dinger perturbation theory. \cite{ZhoBou17} worked on a related but slighted more complicated problem named ``phase synchronization'', and analyzed entrywise behaviors of both the spectral estimator and MLE under a near-optimal regime. \cite{CFM17} used similar ideas to derive the optimality of both the spectral estimator and MLE in top-$K$ ranking problem.

There is a rich literature on the three applications in this paper. The synchronization problems \citep{Sin11, CucLipSin12} aim at estimating unknown signals (usually group elements) from their noisy pairwise measurements, and have attracted much attention in optimization and statistics community recently \citep{bandeira2014tightness, JMR16}. They are very relevant models for cryo-EM, robotics \citep{Sin11, Ros16} and more.

The stochastic block model has been studied extensively in the past decades, with renewed activity in the recent years \citep{Coj06,decelle,massoulie-STOC,Mossel_SBM2,KMM13,ABH16,sbm-groth,levina,ASa15,montanari_sen,bordenave,colin3cpam,banks2}, see \cite{Abb17} for further references, and in particular \cite{McS01}, \cite{Vu14}, \cite{prout}, \cite{massoulie-xu}, \cite{CRV15} and \cite{prout2}, which are closest to this paper in terms of regimes and algorithms. The matrix completion problems \citep{CRe09,CPl10,KMO101} have seen great impacts in many areas, and new insights and ideas keep flourishing in recent works \citep{Ge16, SLu16}. These lists are only a small fraction of the literature and are far from complete.


We organize our paper as follows: we present our main theorems of eigenvector and eigenspace perturbation in Section \ref{sec:main}, which are rigorous statements of the intuitions introduced in Section \ref{sec:intro}. In Section \ref{sec:app}, we apply the theorems to three problems: $\Z_2$-synchronization, SBM, and matrix completion from noisy entries. In Section \ref{sec:sim}, we present simulation results to verify our theories. Finally, we conclude and discuss future works in Section \ref{sec:discuss}.

\subsection{Notations}
We use the notation $[n]$ to refer to $\{1,2,\cdots,n\}$ for $n\in \Z_+$, and let $\R_+ = [0,+\infty)$. For any real numbers $a,b \in \R$, we denote $a \vee b = \max\{a,b\}$ and $a \wedge b = \min\{a,b\}$. For nonnegative $a_n$ and $b_n$ that depend on $n$ (e.g., problem size), we write $a_n \lesssim b_n$ to mean $a_n \le C b_n$ for some constant $C>0$. The notation $\asymp$ is similar, hiding two constants in upper and lower bounds. For any vector $x \in \R ^n$, we define $\| x\|_2 = \sqrt{ \sum_{i=1}^{n} x_i^2 }$ and $\| x \|_{\infty } = \max_i  |x_i| $.
For any matrix $M \in \R^{n \times d}$, $M_{i \cdot}$ refers to its $i$-th row, which is a row vector, and $M_{\cdot i}$ refers to its $i$-th column, which is a column vector. The matrix spectral norm is $\| M \|_2 = \max_{\|x\|_2=1} \|Mx\|_2$, the matrix max-norm is $\| M \|_{\max} = \max_{i,j} |M_{ij}|$, and the matrix $2 \to \infty$ norm is $\| M \|_{2 \to \infty} = \max_{\|x\|_2=1} \|Mx\|_{\infty} = \max_i \| M_{i \cdot} \|_2$. The set of $n\times r$ matrices with orthonormal columns is denoted by $\mathcal{O}_{n \times r}$.

\section{Main results}\label{sec:main}

\subsection{Random matrix ensembles}\label{regularity-conditions}

Suppose $A \in \R^{n\times n}$ is a symmetric random matrix and $A^* = \E A$. Denote the eigenvalues of $A$ by $\lambda_1 \geq \cdots \geq \lambda_n$, and their associated eigenvectors by $\{ u_j \}_{j=1}^n$. Analogously for $A^*$, the eigenvalues and eigenvectors are $\lambda_1^* \geq \cdots \geq \lambda_n^*$ and $\{ u^*_j \}_{j=1}^n$, respectively. We also adopt the convention $\lambda_0= \lambda^*_0 = +\infty$ and $\lambda_{n+1} =  \lambda_{n+1}^* = -\infty $. We allow some eigenvalues to be identical. Thus, some eigenvectors may be defined up to rotations.

Suppose $r$ and $s$ are two integers satisfying $1 \leq r \leq n$ and $0\leq s\leq n-r$.  Let $U = ( u_{s+ 1} , \cdots, u_{s + r}) \in \mathbb{R}^{n \times r}$, $U^* = ( u_{s + 1}^* , \cdots, u_{s + r}^* )\in  \mathbb{R}^{n \times r}$ and $\Lambda^* = \diag( \lambda_{s + 1}^* , \cdots, \lambda_{s + r}^* )\in  \mathbb{R}^{r \times r}$. We are interested in the eigenspace $\spann(U)$. To this end, we assume there is an eigen-gap $\Delta^*$ seperating $\{ \lambda_{s+j}^* \}_{j=1}^r $ from $0$ and other eigenvalues (see Figure \ref{fig:eigengap}), i.e.,
\begin{equation}\label{def:Delta}
\Delta^* = (\lambda_s^* - \lambda_{s+1}^*) \wedge (\lambda_{s+r}^* - \lambda_{s+r+1}^* ) \wedge \min_{i \in [r]} | \lambda_{s+i}^*|.
\end{equation}
Compared with the usual eigen-gap \citep{DavKah70}, our definition also takes the distances between eigenvalues and $0$ into consideration. When $A^*$ is rank-deficient, $0$ is itself an eigenvalue.

\begin{figure}[b]
\centering
\begin{tikzpicture}
[
scale = 1,
axis/.style={very thick, ->, >=stealth'},
]
\draw[axis] (-0.5,0) -- (8.5,0) node(xline)[right]{};
\draw[thick] (0,-0.1) -- (0,0.2);
\node [below] at (-0.2,-0.1) {$\lambda^*_{s+r+1}$};
\draw[thick] (0.6,-0.1) -- (0.6,0.2);
\node [below] at (0.6,-0.1) {$0$};
\draw[thick] (3,-0.1) -- (3,0.2);
\node [below] at (2.9,-0.1) {$\lambda^*_{s+r}$};
\draw[thick] (3.3,-0.1) -- (3.3,0.2);
\draw[thick] (3.6,-0.1) -- (3.6,0.2);
\draw[thick] (3.8,-0.1) -- (3.8,0.2);
\draw[thick] (4.3,-0.1) -- (4.3,0.2);
\node [below] at (4.3,-0.1) {$\lambda^*_{s+2}$};
\draw[thick] (5,-0.1) -- (5,0.2);
\node [below] at (5.2,-0.1) {$\lambda^*_{s+1}$};
\draw[thick] (8,-0.1) -- (8,0.2);
\node [below] at (8,-0.1) {$\lambda^*_{s}$};
\draw [thick,decorate,decoration={brace,amplitude=10pt}, xshift=0pt,yshift=1]
	(0.6,0.3)--(3,0.3) node[midway,xshift=2, yshift=16] {$\Delta^*$};
\end{tikzpicture}
\caption{Eigen-gap $\Delta^*$}\label{fig:eigengap}
\end{figure}
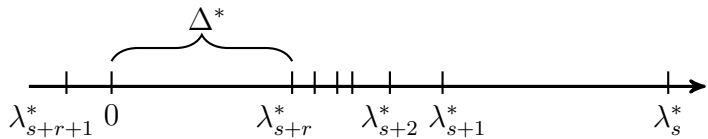
We define $\kappa: = \max_{i \in [r]} | \lambda_{s+i}^*| / \Delta^*$, which is always bounded from below by 1. In our applications, $\kappa$ is usually bounded from above by a constant, i.e., $\Delta^*$ is comparable to $\{ \lambda_{s+j}^* \}_{j=1}^r $ in terms of magnitude.

The concentration property is characterized by a parameter $\gamma \ge 0$, and a function $\varphi(x): \mathbb{R}_+ \to \mathbb{R}_+$. Roughly speaking, $\gamma^{-1}$ resembles the signal-to-noise ratio, and $\gamma$ typically vanishes as $n$ tends to infinity. $\varphi(x)$ is chosen according to the distribution of $A$, and is typically bounded by a constant for $x \in [0,1]$. In particular, we take $\varphi(x) \propto x$ for Gaussian matrices and $\varphi(x) \propto (1 \vee \log(1/x))^{-1}$ for Bernoulli matrices ---see Figure~\ref{fig:phi}. In addition, we will also make a mild structural assumption:  $\| A^* \|_{2 \to \infty} \le \gamma \Delta^*$. In many applications involving low-rank structure, the eigenvalues of interest (and thus $\Delta^*$) typically scale with $n$, whereas $ \| A^* \|_{2\rightarrow\infty }$ scales with $\sqrt{n}$.

Based on the quantities above, we make the following assumptions.
\begin{enumerate}
\item[\textbf{A1}]\label{cond-1} \textbf{(Incoherence)} $\| A^* \|_{2 \to \infty} \le \gamma \Delta^*$.
\item[\textbf{A2}]\label{cond-2} \textbf{(Row- and column-wise independence)} For any $m\in[n]$, the entries in the $m$th row and column of $A$ are independent with others, i.e. $\{ A_{ij} : i = m \text{ or } j = m \}$ are independent of $\{ A_{ij} : i \neq m,  j\neq m \}$.
\item[\textbf{A3}]\label{cond-3} \textbf{(Spectral norm concentration)}
$32 \kappa \max\{ \gamma , \varphi(\gamma) \} \le 1$ and
for some $\delta_0 \in (0,1)$,
\begin{equation}
\mathbb{P} \left( \|A-A^*\|_2\le \gamma  \Delta^* \right) \ge 1 - \delta_0.
\end{equation}
\item[\textbf{A4}]\label{cond-4} \textbf{(Row concentration)} Suppose $\varphi(x)$ is continuous and non-decreasing in $\mathbb{R}_+$ with $\varphi(0) = 0$, $\varphi(x)/x$ is non-increasing in $\mathbb{R}_+$, and $\delta_1 \in (0,1)$. For any $m\in[n]$ and $ W \in \mathbb{R}^{n\times r}$,
\begin{equation}\label{eqn-cond-4}
\P\left( \| (A-A^*)_{m\cdot} W \|_2 \le \Delta^* \| W \|_{ 2 \to \infty }  \,
\varphi  \Big( \frac{\| W \|_F }{\sqrt{n} \| W \|_{2 \to \infty} } \Big)
\right) \ge 1 - \frac{\delta_1}{n}.
\end{equation}
\end{enumerate}

Here are some remarks and intuitions. Assumption \hyperref[cond-1]{1} requires that no row of $A^*$ is dominant. To relate it to the usual concept of incoherence \citep{CRe09,CLMW11}, we consider the case $A^* = U^* \Lambda^* (U^*)^T$ and let $\mu(U^*) = \frac{n}{r} \max_{i \in [n]} \sum_{k} (U^*_{ik})^2 = \frac{n}{r} \|U^*\|_{2\to\infty}^2$. Note that
\begin{align}\label{ineq-cor-main}
\| U^* \Lambda^* (U^*)^T \|_{2\to\infty}
	\leq \| U^* \|_{2\to\infty} \| \Lambda^* (U^*)^T \|_{2}
	= \| U^* \|_{2\to\infty} \| \Lambda^*\|_2
\end{align}
and $\kappa = \|\Lambda^* \|_2 / \Delta^*$. Then Assumption \hyperref[cond-1]{1} is satisfied as long as $\mu(U^*) \le \frac{n \gamma^2}{r \kappa^2} $, which is very mild.

Assumption \hyperref[cond-2]{2} is a mild independence assumption, and it encompasses common i.i.d.\ noise assumptions.

Assumption \hyperref[cond-3]{3} requires the spectral norm of the noise matrix $A-A^*$ to be dominated by $\Delta^*$, which can be interpreted as signal strength. In our example of $\mathbb{Z}_2$ synchronization (see Section \ref{sec:Z2}), we have $\Delta^*=n$, and $A-A^*$ have i.i.d.\ $N(0,\sigma^2)$ entries above the diagonal. Since $\|A-A^*\|_2 \lesssim \sigma \sqrt{n}$ by standard concentration results, we need to require $\sigma = O(\gamma \sqrt{n})$.

Assumption \hyperref[cond-4]{4} is a generalization of the row concentration assumption in Section \ref{sec:eigenvec}, and the function $\varphi$ is problem-dependent. Here we explain the role of $\varphi$ using a special case where $r=1$ and $A\in\{0,1\}^{n\times n}$ has i.i.d. Bernoulli entries with parameter $p=p_n$ on and above its diagonal. Then $\Delta^*=np$, $\sum_{i=1}^{n}A^*_{mi} = np$. If $p$ is not too small, with high probability we have $\sum_{i=1}^{n}A_{mi} \lesssim np$ and thus
\[
| (A-A^*)_{m\cdot} W | \leq \|W\|_{\infty}  \sum_{i=1}^{n} |(A-A^*)_{mi}| \lesssim \|W\|_{\infty}np = \Delta^* \|W\|_{\infty}.
\]
If many entries in $W$ have magnitudes much less than $\|W\|_{\infty}$, there should be less fluctuation and better concentration. Indeed, Assumption \hyperref[cond-4]{4} stipulates a tighter bound by a factor of $\varphi(\frac{\| W \|_2}{\sqrt{n} \| W \|_{\infty} })$, where $\frac{\| W \|_2}{\sqrt{n} \| W \|_{\infty} }$ is typically much smaller than $1$ in this case.
This delicate concentration bound turns out to be crucial in the analysis of SBM, where $A$ is a sparse binary matrix.

\subsection{Entrywise perturbation of general eigenspaces}\label{sec:eigenspace}

In this section, we generalize Theorem~\ref{thm:simple} from individual eigenvectors to eigenspaces under milder conditions that are characterized by additional parameters. Note that neither $U$ nor $U^*$ is uniquely defined, and they can only be determined up to a rotation if the eigenvalues are identical. For this reason, our result has to involve an $r \times r$ orthogonal matrix.
Beyond asserting our result holds up to a suitable rotation, we give an explicit form of such orthogonal matrix.

Let $H = U^T U^* \in \mathbb{R}^{r \times r}$, and its singular value decomposition be $H = \bar{U} \bar{\Sigma} \bar{V}^T$, where $\bar U, \bar V \in \mathbb{R}^{r \times r}$ are orthonormal matrices, and $\bar \Sigma \in  \mathbb{R}^{r \times r}$ is a diagonal matrix. Define an orthonormal matrix $\sgn(H) \in \mathbb{R}^{r \times r}$ as
\begin{equation}\label{def:Hsvd}
\sgn(H) := \bar U \bar V^T.
\end{equation}
This orthogonal matrix is called the \textit{matrix sign function} \citep{Gro11}.
Now we are able to extend the results in Section \ref{sec:eigenvec} to general eigenspaces.
\begin{thm}\label{thm-main}
Under Assumptions \hyperref[cond-1]{A1}--\hyperref[cond-4]{A4}, with probability at least $1-\delta_0-2\delta_1$ we have
\begin{align*}
& \|U  \|_{2\to\infty} \lesssim \left( \kappa + \varphi(1) \right) \| U^* \|_{2\to\infty} + \gamma \| A^* \|_{2 \to \infty} / \Delta^*, \\
&\|U \sgn(H) - A U^*(\Lambda^*)^{-1}\|_{2\to\infty}\lesssim  \kappa
( \kappa + \varphi(1)	)  ( \gamma + \varphi(\gamma) ) \| U^*\|_{2\to\infty}
+\gamma \| A^* \|_{2 \to \infty} / \Delta^*,\\
&\|U \sgn(H) - U^*\|_{2\to\infty}\leq \|U \sgn(H) - A U^*(\Lambda^*)^{-1}\|_{2\to\infty} + \varphi(1) \| U^*\|_{2\to\infty}. 
\end{align*}
Here the notation $\lesssim$ only hides absolute constants.
\end{thm}
The third inequality is derived by simply writing $U \sgn(H) - U^*$ as a sum of the first-order error $EU^*(\Lambda^*)^{-1}$ and higher-order error $U \sgn(H) - A U^*(\Lambda^*)^{-1}$, and bounding $EU^*(\Lambda^*)^{-1}$ by the row concentration Assumption \hyperref[cond-4]{A4}. It will be useful for the noisy matrix completion problem. It is worth pointing out that Theorem \ref{thm-main} is applicable to any eigenvector of $A$ that is not necessarily the leading one. This is particularly powerful in SBM (Section \ref{sec:sbm}) where we need to analyze the second eigenvector. In addition, we do not need $A^*$ to have low rank, although the examples to be presented have such structure. For low-rank $A^*$, estimation errors of all the eigenvectors can be well controlled by the following corollary of Theorem \ref{thm-main}.
\begin{cor}\label{cor-main}
	Let Assumptions \hyperref[cond-1]{A1}--\hyperref[cond-4]{A4} hold, and suppose that $A^* = U^* \Lambda^* (U^*)^T$. With probability at least $1-\delta_0-2\delta_1$, we have
	\begin{align*}
		&\|U  \|_{2\to\infty} \lesssim \left( \kappa + \varphi(1) \right) \| U^* \|_{2\to\infty},\\
		&\|U \sgn(H) - A U^*(\Lambda^*)^{-1}\|_{2\to\infty}\lesssim  \kappa
		( \kappa + \varphi(1)	)  ( \gamma + \varphi(\gamma) ) \| U^*\|_{2\to\infty},\\
		&\|U \sgn(H) - U^*\|_{2\to\infty}\leq \|U \sgn(H) - A U^*(\Lambda^*)^{-1}\|_{2\to\infty} + \varphi(1) \| U^*\|_{2\to\infty}.
	\end{align*}
	Here the notation $\lesssim$ only hides absolute constants.
\end{cor}
Corollary \ref{cor-main} directly follows from Theorem \ref{thm-main}, inequality (\ref{ineq-cor-main}) and the fact that $\kappa \geq 1$. Below we use a simple example to illustrate the results above. Let $A^*=\lambda^* u^* (u^*)^T$ be a rank-one matrix with $\lambda^*>0$ and $\|u^*\|_2=1$. Set $r=1$ and $s=0$. This structure implies $\Delta^* = \lambda^*$ and $\kappa= 1$. Suppose $A$ has independent entries on and above the diagonal. Such $A$ is usually called a \textit{spiked Wigner matrix} in statistics and random matrix theory.

Let Assumptions \hyperref[cond-1]{A1}-\hyperref[cond-4]{A4} hold. The first two inequalities in Corollary \ref{cor-main} are simplified as
\begin{align}
&\|u \|_{\infty} \lesssim \left( 1 + \varphi(1) \right) \| u^* \|_{\infty} , \label{ineq:thm0a}\\
&\left\|u -  A u^* / \lambda^* \right\|_{\infty}\lesssim  ( \gamma + \varphi(\gamma) ) ( 1 + \varphi(1)) \| u^*\|_{\infty}\label{ineq:thm0b}.
\end{align}
In many applications, $\varphi(1)\lesssim 1$ and $\gamma=o(1)$ as $n$ goes to infinity. Then \eqref{ineq:thm0a} controls the magnitude of the empirical eigenvector $u$ by that of the true eigenvector $u^*$ in the $\ell^\infty$ sense.  Furthermore, \eqref{ineq:thm0b} has the same form as the main result in Theorem \ref{thm:simple}, stating that $A u^* / \lambda^*$ is an $\ell_{\infty}$ approximation of $u$ with error much smaller than $\| u^* \|_{\infty}$. Therefore, it is possible to study $u$ via its linearization $A u^* / \lambda^*$, which usually makes analysis much easier.

The regularity conditions in Theorem \ref{thm:simple} imply our Assumptions \hyperref[cond-1]{A1}-\hyperref[cond-4]{A4}. In particular, the condition $\| u^* \|_{\infty} \le \gamma$ there is equivalent to Assumption \hyperref[cond-1]{A1}.
As a result, Theorem \ref{thm:simple} with $r=1$ is a special case of Corollary \ref{cor-main} and hence of Theorem \ref{thm-main}. It is not hard to generalize to $r=O(1)$.

\section{Applications}\label{sec:app}

\subsection{$\Z_2$-synchronization and spiked Wigner model}\label{sec:Z2}
The problem of $\Z_2$-\textit{synchronization} is to recover $n$ unknown labels $\pm 1$ from noisy pairwise measurements. This is a prototype of more general $\mathrm{SO}(d)$-synchronization problems including phase synchronization and $\mathrm{SO}(3)$-synchronization, in which one wishes to estimate the phases of signals or rotations of cameras/molecules, etc. Such problems arise in time synchronization of distributed networks \citep{giridhar2006distributed}, calibration of cameras \citep{tron2009distributed}, and cryo-EM \citep{shkolnisky2012viewing}.

Consider an unknown signal $x \in \{ \pm 1\}^n$. Suppose we have independent measurements of the form
$Y_{ij} = x_i x_j + \sigma W_{ij}$, where $i < j$,  $W_{ij} \sim N(0,1)$ and $\sigma > 0$. We can define $W_{ii} = 0$ and $W_{ij} = W_{ji}$ for simplicity, and write our model into a matrix form as follows:
\begin{equation}\label{def:syn}
Y = xx^T + \sigma W, \qquad x \in \{ \pm 1\}^n.
\end{equation}
This is sometimes called the Gaussian $\Z_2$-\textit{synchronization} problem, in contrast to the one with $\Z_2$-noise, also known as the censored block model \citep{abbs}.
This problem can be further generalized: each entry $x_j$ is a unit-modulus complex number $e^{i \theta_j}$, if the goal is to estimate  unknown angles from pairwise measurements; or, each entry $x_j$ is an orthogonal matrix from $\mathrm{SO}(3)$,  if the goal is to estimate unknown orientations of molecules, cameras, etc. Here we focus on the simplest case $x_j \in \{\pm 1\}$.

Note that in \eqref{def:syn}, both $Y$ and $W$ are symmetric matrices in $\R^{n \times n}$, and the data matrix $Y$ has a noisy rank-one decomposition. This  falls into the spiked Wigner model. The quality of an estimator $\hat x$ is usually gauged either by its correlation with $x$, or by the proportion of labels $x_i$ it correctly recovers.
It has been shown that the information-theoretic threshold for a nontrivial correlation is $\sigma = \sqrt{n}$ \citep{JMR16,yash_sbm, lelarge2016fundamental, perry2016optimality}, and the threshold for exact recovery (i.e., $\hat x = \pm x$ with probability tending to $1$) is $\sigma = \sqrt{\frac{n}{2\log n}}$ \citep{bandeira2014tightness}.

When $\sigma \leq \sqrt{\frac{n}{(2+\varepsilon)\log n}}$ ($\varepsilon>0$ is any constant), it was proved by \cite{bandeira2014tightness} that semidefinite programming (SDP) finds the maximum likelihood estimator and achieves exact recovery. We are going to show that a very simple method, both conceptually and computationally, also achieves exact recovery. This method is outlined as follows:
\begin{enumerate}
\item Compute the leading eigenvector of $Y$, denoted by $u$;
\item Take the estimate $\hat{x} = \sgn(u)$.
\end{enumerate}

Our next theorem asserts that the eigenvector-based method above succeeds in finding $x$ consistently under $\sigma \leq \sqrt{\frac{n}{(2+\varepsilon)\log n}}$. Thus, under any regime where the MLE achieves exact recovery, our eigenvector estimator $\hat x$ equals the MLE with high probability.  This phenomenon also holds for the stochastic block model.

\begin{thm}\label{thm:syn}
Suppose $\sigma \le \sqrt{\frac{n}{(2+\veps)\log n}}$ for some $\veps > 0$. With probability $1 - o(1)$, the leading eigenvector of $Y$ with unit $\ell_2$ norm satisfies
\begin{equation*}
\sqrt{n} \,  \min_{i \in [n]} \{ s x_i u_i \} \ge 1 - \sqrt{\frac{2}{2+\varepsilon}} + \frac{C}{\sqrt{ \log n}},
\end{equation*}
for a suitable $s \in \{ \pm 1\}$, where $C>0$ is an absolute constant. As a consequence, our eigenvector-based method achieves exact recovery.
\end{thm}

Note that our approach does not utilize the structural constraints $|x_i| = 1, ~ \forall \, i \in [n]$; whereas such constraints appear in the SDP formulation \citep{bandeira2014tightness}. A natural question is an analysis of both methods with an increased noise level $\sigma$. A seminal work by \cite{JMR16} complements our story: the authors showed via non-rigorous statistical mechanics arguments that when $\sigma$ is on the order of $\sqrt{n}$, the SDP-based approach outperforms the eigenvector approach. Nevertheless, with a slightly larger signal strength, there is no such advantage of the SDP approach. 

When $\sigma \asymp \sqrt{n}$, general results for spiked Wigner models \citep{BBP05, FPe07, Ben11} imply that $\frac{1}{n} | u^T x |^2 \to 1 - \sigma^2/n$ for $\sigma/\sqrt{n}<1-\varepsilon$ with any small constant $\varepsilon>0$. 
\cite{yash_sbm} proved that non-trivial correlation with $x$ cannot be obtained by any estimator if $\sigma/\sqrt{n}>1+\varepsilon$.

\subsection{Stochastic Block Model}\label{sec:sbm}
As is briefly discussed in Section \ref{sec:intro}, we focus on the symmetric SBM with two equally-sized groups. (Though the second eigenvector of $A^*$ depends on relative sizes of the groups, our analysis only requires slight modification if groups have different sizes.) For simplicity, we allow for self-loops (i.e. edges from vertices to themselves) in the random graph, and it makes no much difference if they are excluded. In that case, the expectation of the adjacency matrix changes by a negligible quantity $O(\log n/n)$ under the spectral norm and moreover, Assumptions \hyperref[cond-1]{A1}--\hyperref[cond-4]{A4} still hold with the same parameters.
\begin{defn}\label{def-SBM}
	Let $n$ be even, $0\leq q \le p \leq 1$, and $J \subseteq [n]$ with $|J|=n/2$. $\mbox{\em SBM}(n,p,q,J)$ is the ensemble of $n\times n$ symmetric random matrices $A=(A_{ij})_{i,j\in[n]}$ where $\{A_{ij}\}_{1\leq i\leq j\leq n}$ are independent Bernoulli random variables, and
	\begin{equation}
	\P(A_{ij}=1)=\begin{cases}
	&p,\text{ if }i\in J, j \in J\text{ or }i \in J^c , j\in J^c\\
	&q,\text{otherwise}
	\end{cases}
	.
	\end{equation}
\end{defn}
The community detection problem aims at finding the bi-partition $(J,J^c)$ given only one realization of $A$. Let $z_i = 1$ if $i \in J$ and $z_i = -1$ otherwise. We want to find an estimator $\hat z$ for the unknown labels $z \in \{ \pm 1\}^n$. Intuitively, the task is more difficult when $p$ is close to $q$, and when the magnitudes of $p,q$ are small. It is impossible, for instance, to produce any meaningful estimator when $p=q$. The task is also impossible when $p$ and $q$ are as small as $o(n^{-2})$, since $A$ is a zero matrix with high probability.

As is already discussed in Section \ref{sec:intro}, under the regime $p=  \frac{a \log n}{n}, q=  \frac{b \log n}{n}$ where $a$ and $b$ are constants independent of $n$, it is information theoretically impossible to achieve exact recovery (the estimate $\hat z$ equals $z$ or $-z$ with probability tending to $1$) when $\sqrt{a} - \sqrt{b} < \sqrt{2}$. In contrast, when $\sqrt{a} - \sqrt{b}> \sqrt{2}$, the goal is efficiently achievable. Further, it is known that SDP succeeds down to the threshold. Under the regime $p=  \frac{a}{n}, q=  \frac{b}{n}$, it is impossible to obtain nontrivial correlation (i.e. the correlation between $\hat z$ and $z$ is at least some positive constant $\veps$, as a random guess gets roughly half the signs correct and almost zero correlation with $z$) between any estimator $\hat z$ and $z$ if $(a-b)^2 < 2(a+b)$, and when $(a-b)^2 > 2(a+b)$, nontrivial correlation can efficiently be obtained \citep{massoulie-STOC,Mossel_SBM2}.

Here we focus on the regime where $p = a \frac{\log n}{n}$, $q = b \frac{\log n}{n}$ and $a>b > 0$ are constants. Note that $\E A$, or equivalently $A^*$, is a rank-$2$ matrix. Its nonzero eigenvalues are $\lambda^*_1=(p+q) n/2$ and $\lambda^*_2=(p-q) n/2$, whose associated eigenvectors are $u^*_1=\frac{1}{\sqrt{n}}\mathbf{1}_n$ and $u^*_2=\frac{1}{\sqrt{n}} \mathbf{1}_{J}  - \frac{1}{\sqrt{n}} \mathbf{1}_{J^c}$. As $ u^*_2 $ is aligned with $z$ and perfectly reveals the desired partition, the following vanilla spectral method is a natural candidate:
\begin{enumerate}
	\item Compute $u_2$, the eigenvector of $A$ corresponding to its second largest eigenvalue $\lambda_2$;
	\item Set $\hat{z} = \sgn(u_2)$. 
\end{enumerate}
It has been empirically observed and conjectured that as soon as the signal strength $\sqrt{a} - \sqrt{b}$ exceed the information threshold $\sqrt{2}$,  the vanilla spectral method achieves exact recovery \citep{abh_arxiv}. Moreover, in regimes where exact recovery is impossible, \cite{zhang2016minimax} established the following minimax result. It has not been clear whether the vanilla spectral method achieves the minimax misclassification rate.

If we define the misclassification rate as
\begin{equation}\label{def:mis-sbm}
r(\hat z,z) = \min_{s \in \{ \pm 1 \} } n^{-1} \sum_{i=1}^n \mathbf{1}_{\{ \hat{z}_i \neq s z_i \}},
\end{equation}
then the results of \cite{zhang2016minimax} imply that
\begin{equation}\label{eq:miscla}
\inf_{\hat z} \sup  \E r(\hat z,z) = \exp\left(   -(1+o(1)) \cdot (\sqrt{a}-  \sqrt{b})^2 \frac{\log n}{2} \right),
\end{equation}
where the supremum is taken over approximately equal-sized SBM with 2-blocks. Note that this parameter space is slightly different from our Definition~\ref{def-SBM}, but as explained before, we can modify our proofs accordingly such that the same conclusions still hold. See the supplementary materials \citep{supp} for further explanation of \eqref{eq:miscla}.

Here we prove that the vanilla spectral method indeed succeeds in exact recovery whenever it is information-theoretic possible, which resolves the conjecture of \citep{abh_arxiv}; and if it is not, vanilla spectral method achieves the optimal misclassification rate.

\begin{thm}\label{thm-sbm}
(i) If $\sqrt{a} - \sqrt{b} > \sqrt{2}$, then there exists $\eta = \eta(a,b) > 0$ and $s \in \{ \pm 1\}$ such that with probability $1-o(1)$,
\begin{equation*}
\sqrt{n} \,  \min_{i\in [n]} s z_i(u_2)_i \ge \eta .
\end{equation*}
As a consequence, our spectral method achieves exact recovery. \\
(ii) Let the misclassification rate $r(\hat z,z)$ be defined in \eqref{def:mis-sbm}. If $\sqrt{a} - \sqrt{b} \in (0, \sqrt{2}]$, then
\begin{equation*}
\E r(\hat z, z) \le n^{- (1+o(1)) (\sqrt{a} - \sqrt{b})^2/2}.
\end{equation*}
This upper bound matches the minimax lower bound.
\end{thm}

The first part implies that, under the regime where the MLE achieves exact recovery, our eigenvector estimator is exactly the MLE with high probability. This proves Corollary~\ref{cor-intro} in the introduction. Moreover, the second part asserts that for more challenging regime where exact recovery is impossible, the eigenvector estimator has the optimal misclassification rate.

Before further explaining our results, we give a brief review of previous endeavors and an analysis of difficulties. Various papers have investigated this algorithm and its variants such as \cite{McS01}, \cite{Coj06}, \cite{RohChaYu11},
\cite{STF12}, \cite{Vu14}, \cite{massoulie-xu}, \cite{prout}, \cite{prout2}, \cite{LRi15}, \cite{Gao15}, among others. However, it is not known if the simple algorithm above achieves exact recovery down to the information-theoretic threshold, nor the optimal misclassification rate studied in \cite{zhang2016minimax} while below the threshold.
An important reason for the unsettlement of this question is that the entrywise behavior of $u_2$ is not fully understood.
In particular, people have been focusing on the $\ell_{\infty}$ error $\| u_2 - u_2^*\|_{\infty}$, which may well exceed $\|u_2^*\|_{\infty}$ (see Theorem \ref{thm::sbm-lowerbound}), suggesting that the algorithm may potentially fail by rounding on the incorrect sign. This is not necessarily the case---as errors could have larger magnitudes on the `good side' of the signal range---but $\| u_2-u_2^*\|_{\infty}$ cannot capture this. To avoid suboptimal theoretical results, multi-round algorithms are popular choices in the literature \citep{Coj06,Vu14}, which typically have a preprocessing step of trimming and/or a postprocessing step refining the initial solution. \cite{prout} and \cite{prout2} showed that such variants can achieve the exact recovery threshold. We are going to prove that the vanilla spectral algorithm alone achieves the threshold and the minimax lower-bound in one shot.

The key to proving Theorem~\ref{thm-sbm} is the following first-order approximation result for $u_2$ under the $\ell_\infty$ norm, which is a consequence of Theorem \ref{thm-main}. 
\begin{cor}\label{cor-sbm2}
If  $A \sim \mbox{\em SBM}(n , a\frac{\log n}{n}, b \frac{\log n}{n} ,J )$, then with probability $1-O(n^{-3})$ we have
\begin{equation}\label{ineq:sbm-approx}
\min_{s\in \{ \pm 1 \} }\| u_2- sA u_2^*/\lambda_2^* \|_{\infty}\leq \frac{C}{\sqrt{n}\log\log n}.
\end{equation}
where $C = C(a,b)$ is some constant depending only on $a$ and $b$.
\end{cor}
The above result holds for any constants $a$ and $b$, and does not depend on the gap $\sqrt{a} - \sqrt{b}$. This fact will be useful for analyzing the misclassification rate. By Corollary~\ref{cor-sbm2}, the $\ell_\infty$ approximation error is negligible, and thus the analysis of vanilla spectral algorithm boils down to analyzing the entries in $A u_2^*/\lambda_2^*$, which are just weighted sums of Bernoulli random variables. 

As a by-product, we can show that entrywise analysis through $\| u_2 - u_2^* \|_{\infty}$ is not a good strategy. As is mentioned earlier, our sharp result for the eigenvector estimator stems from careful analysis of the linearized version $A u^*_2 / \lambda^*_2$ of $u_2$, and the approximation error $\| u_2 - A u^*_2 / \lambda^*_2\|_{\infty}$. This is superior to direct analysis of the $\ell_{\infty}$ perturbation $\| u_2 - u_2^* \|_{\infty}$, as the next theorem implies that $\| u_2 - u_2^* \|_{\infty}>\| u_2^*\|_{\infty}$ is possible even if $\sgn(u_2) = \sgn(u_2^*)$.

\begin{thm}[Asymptotic lower bound for eigenvector perturbation]\label{thm::sbm-lowerbound}
	Let $J = [n/2]$ and $A\sim \mbox{\em SBM}(n, a\frac{\log n}{n}, b \frac{\log n}{n} , J )$, where $a>b > 0$ are constants and $n\rightarrow\infty$. For any fixed $ \eta > 1$ with $  \eta \log \eta - \eta + 1 < 2/a$, with probability $1 - o(1)$ we have
	\[
	\sqrt{n} \| u_2 - u_2^* \|_{\infty} \geq  \frac{ a (\eta - 1) }{ a-b }.
	\]
\end{thm}
Now let us consider the case in Figure~\ref{fig:intro}, where $a=4.5$ and $b = 0.25$. On the one hand, exact recovery is achievable since $\sqrt{a} - \sqrt{b} > 1.62 > \sqrt{2}$. On the other hand, by taking $\eta = 2$ we get $h(\eta) = 2\log 2 -2+1 < 4/9 = 2/a$ and $ \frac{ a (\eta - 1) }{ a-b } >1.05$. Theorem \ref{thm::sbm-lowerbound} implies
\[
\lim_{n \to \infty} \P ( \|u_2 - u_2^* \|_{\infty} > 1.05/\sqrt{n} )= 1.
\]
In words, the size of fluctuation is consistently larger than the signal strength. As a result, by merely looking at $\| u_2 - u_2^* \|_{\infty} $ we cannot expect sharp analysis of the spectral method in exact recovery.

Finally, we point out that it is not straightforward to develop a simple spectral method to achieve the information threshold for exact recovery in SBM with $K>2$ blocks. Spectral methods in this scenario \citep{RohChaYu11,LRi15} typically start with $r>1$ eigenvectors $\{ v_j \}_{j=1}^r \subseteq \R^n$ of some data matrix (e.g. the adjacency matrix or Laplacian matrix). Then, the $n$ rows of $V=(v_1,\cdots,v_r) \in \R^{n\times r}$ are treated as embeddings of the $n$ nodes into $\R^r$, from which one infers block memberships using clustering techniques. In our vanilla spectral method for 2 blocks, we only look at a single eigenvector and return the blocks based on signs of coordinates. This method always returns the same memberships (up to a global swap), even though the eigenvector is identifiable only up to a sign. When $K>2$ and $r>1$, due to possible multiplicity of eigenvalues, the embeddings of $n$ nodes may be identifiable only up to an orthonormal transform in $\R^r$. Such ambiguity causes trouble for effective clustering, although we can still study the embedding using Theorem \ref{thm-main}. Due to space constraints, we put a brief discussion in the supplementary material \cite{supp}.

\subsection{Matrix completion from noisy entries}\label{sec:NMC}
Matrix completion based on partial observations has wide applications including collaborative filtering, system identification, global positioning, remote sensing, etc., see \cite{CPl10}. A popular version is the ``Netflix problem", where one is given a incomplete table of customer ratings and wants to predict the missing entries. This could be useful for targeted recommendation in the future. Since it has been intensively studied in the past decade, our brief review below is by no means exhaustive. \cite{CRe09}, \cite{CTa10}, and \cite{Gro11} focused on exact recovery of low-rank matrices based on noiseless observations. More realistic models with noisy observations were studied in \cite{CPl10}, \cite{KMO101}, \cite{KLT11}, \cite{JaiNetSan13} and \cite{Cha15}.

As an application of Theorem \ref{thm-main}, we are going to study a model similar to the one in \cite{Cha15} where both sampling scheme and noise are random. It can be viewed as a statistical problem with missing values. Suppose we have an unknown signal matrix $M^* \in \R^{n_1\times n_2}$. For each entry of $M^*$, we have a noisy observation $M_{ij}^*+\varepsilon_{ij}$ with probability $p$, and have no observation otherwise. Let $M^{obs}\in\R ^{n_1\times n_2}$ record our observations, with missing entries treated as zeros. We consider the rescaled partial observation matrix $M=M^{obs}/p$ for simplicity. It is easy to see that $M$ is an unbiased estimator for $M^*$, and hence a popular starting point for further analysis. The definition of our model is formalized below.

\begin{defn}\label{def-MC}
	Let $M^*\in\R^{n_1\times n_2}$, $p\in (0,1]$ and $\sigma \geq 0$. We define $\mbox{\em NMC}(M^*,p ,\sigma )$ to be the ensemble of $n_1\times n_2$ random matrices  $M=(M_{ij})_{i\in[n_1],j\in[n_2]}$ with $M_{ij}=( M_{ij}^* + \varepsilon_{ij} )I_{ij}/p$, where $\{ I_{ij}, \varepsilon_{ij} \}_{ i \in[n_1], j \in [n_2] }$ are jointly independent, $\P ( I_{ij}=1 )=p=1- \P ( I_{ij} =0 )$ and $\varepsilon_{ij} \sim N(0,\sigma^2)$.
\end{defn}

Let $r =\rank(M^*)$ and $M^* = U^* \Sigma^* V^*$ be its singular value decomposition (SVD), where $U^* \in\mathcal{O}_{n_1\times r }$, $V^*\in\mathcal{O}_{n_2\times r }$, $\Sigma^* = \diag ( \sigma_1^*,\cdots,\sigma_r^* )$ is diagonal, and $\sigma_1^* \geq \cdots \geq \sigma_r^*$. We are interested in estimating $U^*$, $V^*$, and $M^*$. The rank $r $ is assumed to be known, which is usually easily estimated otherwise, see \cite{KOh09} for example. We work on a very simple spectral algorithm that often serves as an initial estimate of $M^*$ in iterative methods.
\begin{enumerate}
	\item Compute the $r$ largest singular values $\sigma_1\geq \cdots\geq \sigma_{r}$ of $M$, and their associated left and right singular vectors $\{u_j\}_{j=1}^{r }$ and $\{v_j\}_{j=1}^{r }$. Define $\Sigma=\diag(\sigma_1,\cdots,\sigma_{r })$, $U=(u_1,\cdots,u_{r })\in\mathcal{O}_{n_1\times r }$ and $V=(v_1,\cdots,v_{r } )\in\mathcal{O}_{ n_2 \times r }$.
	\item Return $U$, $V$ and $U\Sigma V^T$ as estimators for $U^*$, $V^*$, and $M^*$, respectively.
\end{enumerate}
Note that the matrices in Definition \ref{def-MC} are asymmetric in general, due to the rectangular shape and independent sampling. Hence, Theorem \ref{thm-main} is not directly applicable. Nevertheless, it could be tailored to fit into our framework by a ``symmetric dilation" trick. See the supplementary materials \citep{supp} for details. Below we present our results.
\begin{thm}\label{thm-MC-main}
	Let $M\sim \mbox{\rm NMC}(M^*,p,\sigma)$, $n=n_1+n_2$, $\kappa = \sigma_1^*/\sigma_r^*$, $H = \frac{1}{2}( U^T U^* + V^T V^* ) $, and $\eta=\left(
	\|U^*\|_{2\to\infty} \vee \|V^*\|_{2\to\infty}
	\right)$.
	There exist constants $C$ and $C'$ such that the followings hold.
	Suppose $p\geq 6 \frac{\log n}{n}$ and $\kappa \frac{n ( \|M^*\|_{\max} + \sigma )}{\sigma_r^*} \sqrt{\frac{\log n}{np}}  \leq 1/C$.
	With probability at least $1-C/n$, we have
	\begin{align*}
		&\left(
		\| U \|_{2\to\infty} \vee \| V \|_{2\to\infty}
		\right)
		\leq C' \kappa \eta,\\
		&\left(
		\|U \sgn(H) - U^*\|_{2\to\infty} \vee \| V \sgn(H) - V^*\|_{2\to\infty}
		\right) \leq C' \eta \kappa^2 \frac{ n ( \| M^* \|_{\max} + \sigma )}{\sigma_r^*}
		\sqrt{\frac{\log n}{np}},\\
		&\| U \Sigma V^T - M^* \|_{\max} \leq C'\eta ^2  \kappa^4
		( \| M^* \|_{\max} + \sigma )
		\sqrt{\frac{ n\log n}{p}}
		.
	\end{align*}
\end{thm}
To our best knowledge, the results for singular vectors are the first of this type for the spectral algorithm. Our bound on $\| U \Sigma V^T - M^* \|_{\max}$ is a by-product of that, and a similar result was derived by \cite{JNe15} using a different approach.

There are two reasons why entrywise type bounds are important. First, in applications such as recommender systems, it is often desirable to have uniform guarantees for all individuals. If we directly use existing $\ell_2$-type inequalities to control entrywise errors, the resulting bounds can be highly sub-optimal in high dimensions. Thus new results are needed. Second, in algorithms based on non-convex optimization \citep{KMO101,SLu16,JNe15}, entrywise bounds are critical for the analysis of initializations and iterations. After the first draft of this paper came out, the entrywise bounds on singular subspaces were applied by \cite{MWC17} as a guarantee for spectral intialization. The relevance of entrywise bounds goes well beyond matrix completion; see Section~\ref{sec:related}.


For the rest of this subsection, we will illustrate the results in Theorem~\ref{thm-MC-main} by comparing them with existing ones based on Frobenius norm.



Suppose $p>c\frac{\log n}{n}$ for some large constant $c>0$. By Theorems~1.1 and~1.3 of \cite{KMO101}, an upper bound for the root-mean squared error (RMSE) gives:
\begin{align}
	\frac{1}{ n } \| U \Sigma V^T - M^* \|_F \lesssim (\|M^*\|_{\max} + \sigma) \sqrt{\frac{r}{np}}. \label{eqn-NMC-RMSE}
\end{align}
This implies that the spectral algorithm is rate-optimal when $\sigma \gtrsim \| M^* \|_{\max}$, as \cite{CPl10} established a 
lower bound $\frac{1}{ n } \| \hat{M} - M^* \|_F \gtrsim \sigma \sqrt{\frac{r}{np}}$ for any estimator $\hat{M}$. On the other hand, our Theorem \ref{thm-MC-main} asserts that
\begin{align}
	\| U \Sigma V^T - M^* \|_{\max} & \lesssim_{\kappa, r,\eta} (\|M^*\|_{\max} + \sigma) \sqrt{\frac{\log n}{np}}. \notag
\end{align}
where $\lesssim_{\kappa, r,\eta}$ hides a factor $O(\kappa, r, \eta \sqrt{n/r})$ that is not large if certain matrix incoherence structure is assumed; see \cite{CRe09} for example. 
Note that our result recovers (\ref{eqn-NMC-RMSE}) up to a factor of $ \sqrt{\log n} $, since $\| X \|_F \le \sqrt{n_1 n_2} \| X \|_{\max}$ always holds for any $X$ of size $n_1 \times n_2$. 

We also compare the estimation errors of singular vectors under the Frobenius norm and the max-norm. On the one hand, the perturbation inequality in \cite{Wed72} and spectral norm concentration
yield the following.
\begin{align}
	&\max \{ \|U \sgn(H) - U^*\|_F, \|V \sgn(H) - V^*\|_F \} \lesssim
	\sqrt{r}\|M-M^* \|_2 / \sigma_r^*  \notag \\
	&\lesssim \frac{\sqrt{rn/p} (\| M^* \|_{\max}+\sigma) }{ \sigma_r^* } \lesssim
	\frac{ n \|M^*\|_{\max} }{ \sigma_r^* }
	\left(
	1 + \frac{\sigma}{\| M^* \|_{\max}}
	\right)
	\sqrt{\frac{r}{np}}
	.
	\label{eqn-NMC-RMSE-eig}
\end{align}
On the other hand, by our entry-wise bound in Theorem \ref{thm-MC-main} we have
\begin{align}
	& ~~~~ \sqrt{n}\,\max \{ \|U \sgn(H) - U^*\|_{2\to\infty}, \|V \sgn(H) - V^*\|_{2\to\infty} \}
	\notag \\&
	\lesssim_{\kappa,r,\eta}
	\frac{ n\|M^*\|_{\max} }{ \sigma_r^* }
	\left(	1 + \frac{\sigma}{\| M^* \|_{\max}}
	\right)
	\sqrt{\frac{r \log n}{np}}
	.\label{eqn-NMC-RMSE-eig-1}
\end{align}
where, as before, $\lesssim_{\kappa,r,\eta}$ hides a factor $O(\kappa, r,\eta \sqrt{n/r})$ that is usually not large. Therefore, we also recover
(\ref{eqn-NMC-RMSE-eig}) up to a factor of $\sqrt{\log n}$, since $\| X \|_F \le \sqrt{n r}\, \| X \|_{\max}$ holds for any $X$ of size $n \times r$. Note that our goal is to derive good max-norm bounds rather than improving Frobenius-norm bounds. The comparisons above demonstrate that our bounds have the `correct' order. To a certain extent, our results better portrait the behavior of spectral algorithm and provide more information than their Frobenius counterparts.

\section{Numerical experiments}\label{sec:sim}
\subsection{$\Z_2$-synchronization}\label{sec:simZ2}
We present our numerical results for the phase transition phenomenon of $\Z_2$-synchronization---see Figure \ref{fig:Z2ER}. Fix $q_1 = 500^{1/50}$ and $q_2 = 2^{1/10}$. For each $n$ in the geometric sequence $\{2, 2q_1, 2q_1^2, \cdots, 2q_1^{50} \}$ (rounded to the nearest integers), and each $\sigma$ in the geometric sequence $\{ q_2^{-32}, q_2^{-31},\cdots, q_2^{50} \}$, we compare our eigenvector-based estimator $\hat x$ with the unknown signal $x$, and report the proportion of success (namely $\hat x = \pm x$) out of $100$ independent runs in the heat map.

A theoretical curve $\sigma = \sqrt{\frac{n}{2\log n}}$ is added onto the heat map. It is clear that below the curve, the eigenvector approach almost always recovers the signal perfectly; and above the curve, it fails to recover the signal.
\begin{figure}[h!]
	\centering
	\includegraphics[scale=0.3]{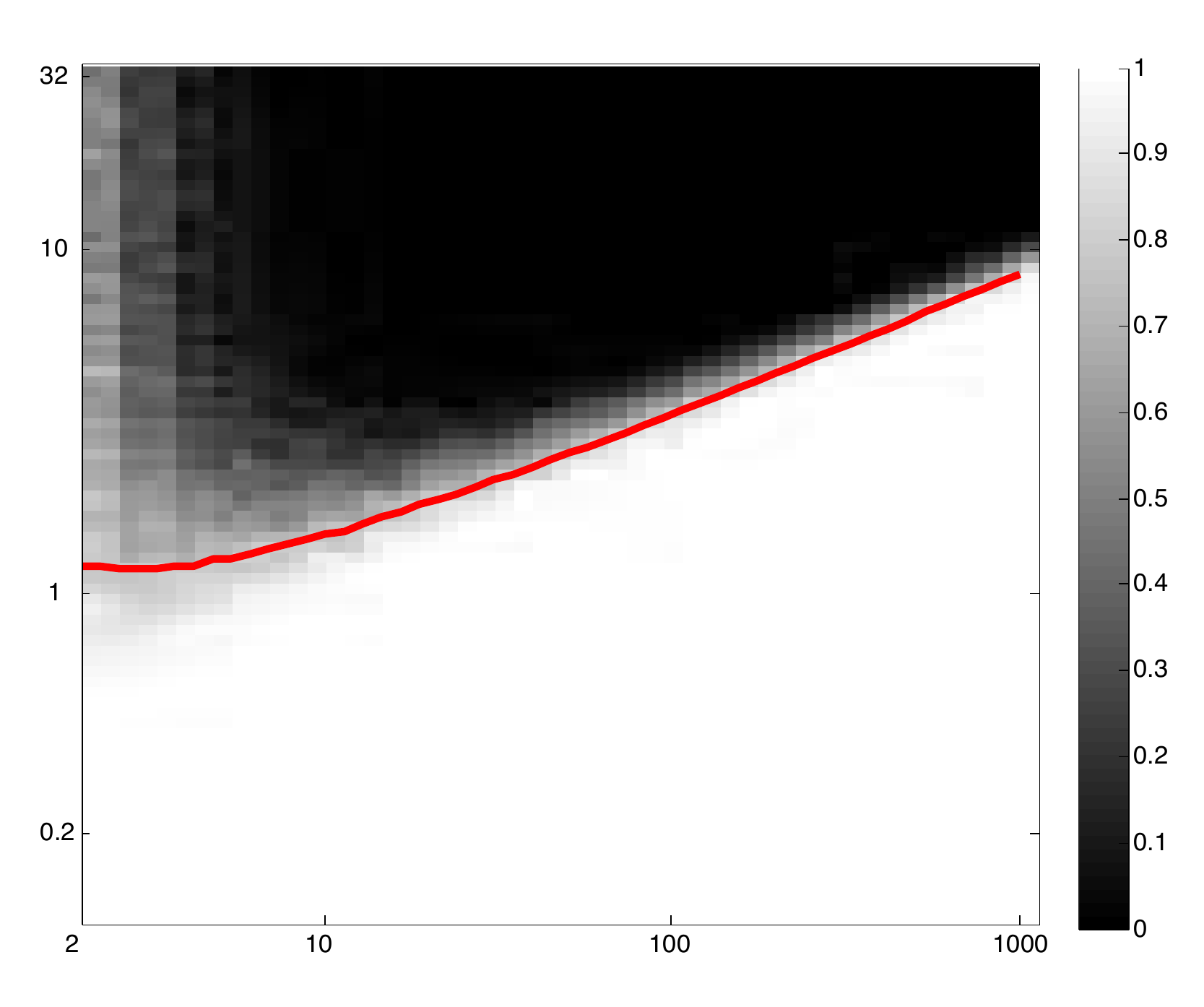}
	\caption{Phase transition of $\mathbb{Z}_2$-synchronization: the $x$-axis is the dimension $n$, and the $y$-axis is $\sigma$. 
	Lighter pixels refer to higher proportions of runs that $\hat x$ recovers $x$. The red curve shows the theoretical boundary $\sigma = \sqrt{\frac{n}{2\log n}}$.  }\label{fig:Z2ER}
\end{figure}

\subsection{Stochastic Block Model}\label{sec:simSBM}

Now we present our simulation results for exact recovery and misclassification rates of SBM. The phase transition phenomenon of SBM is exhibited on the left of Figure \ref{fig:SBM1}. In this simulation, $n$ is fixed as $300$, and parameters $a$ ($y$-axis) and $b$ ($x$-axis) vary from $0$ to $30$ and $0$ to $10$, with increments $0.3$ and $0.1$ respectively. We compare the labels returned by our eigenvector-based method with the true cluster labels, and report the the proportion of success (namely $\hat z = \pm z$) out of $100$ independent runs. As before, lighter pixels represent higher chances of success. Two theoretical curves $\sqrt{a} - \sqrt{b} = \pm \sqrt{2}$ are also added onto the heat map. Clearly, theoretical predictions match numerical results.

The right plot of Figure \ref{fig:SBM1} shows misclassification rates of our eigenvector approach with a fixed parameter $b$ and a varying parameter $a$, where $a$ is not large enough to reach the exact recovery threshold. We fix $b=2$, and increase $a$ from $2$ to $8$ by $0.2$ for three different choices of $n$ from $\{100,500,5000\}$. Then we calculate the mean misclassification rates $\E r(\hat z, z)$ averaged over $100$ independent runs, and plot $\log \E r(\hat z, z) / \log n$ against varying $b$. We also add a theoretical curve (with no markers), whose $y$-coordinates are $-(\sqrt{a} - \sqrt{b})^2/2$; see Theorem~\ref{thm-sbm} (ii). It is clear that with $n$ tending to infinity, the curves of mean misclassification rates move closer to the theoretical one.

\begin{figure}[h!]
    \centering
    \begin{subfigure}{0.44\textwidth}
        \includegraphics[width=\textwidth]{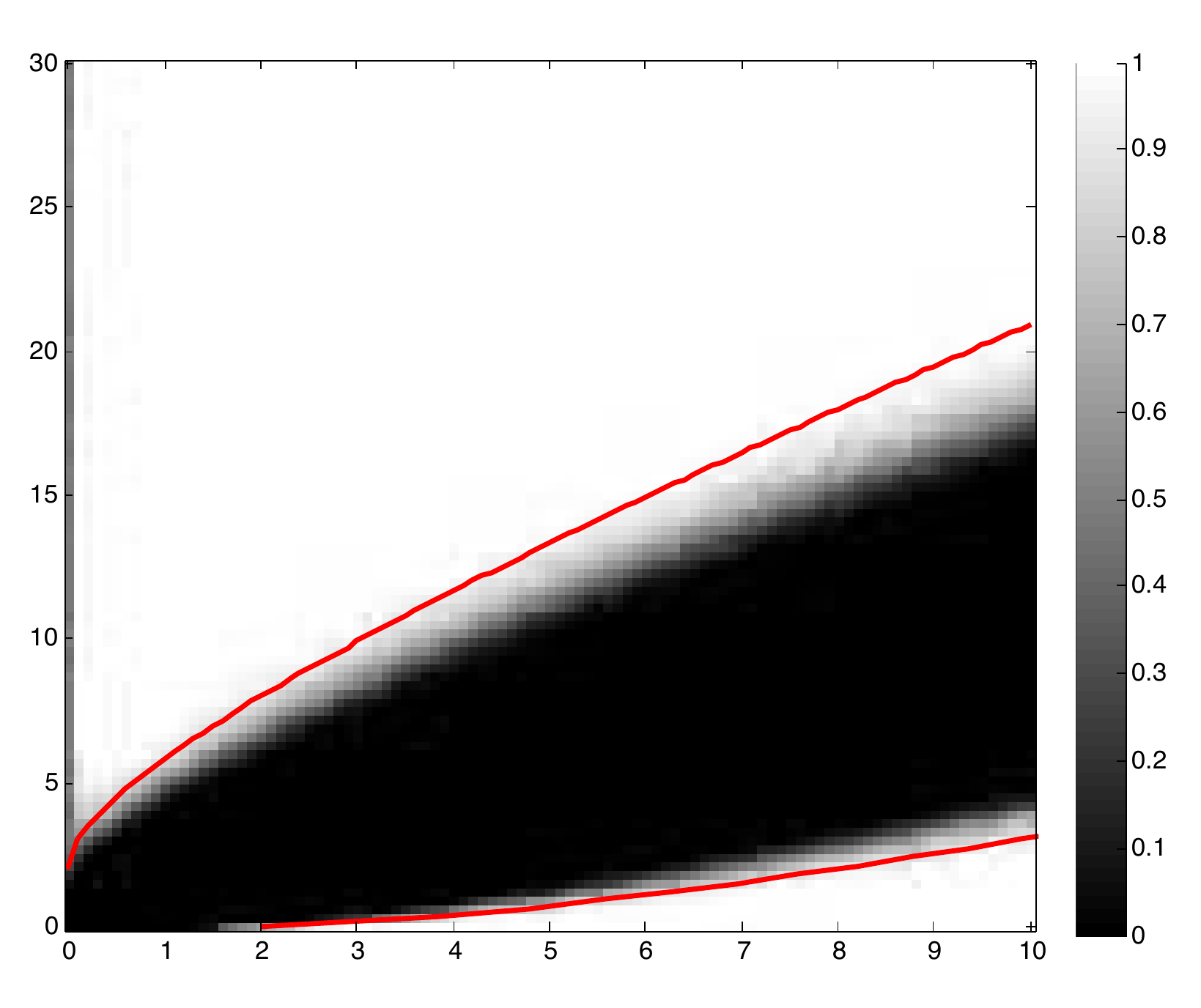}
    \end{subfigure}
    \begin{subfigure}{0.44\textwidth}
        \includegraphics[width=\textwidth]{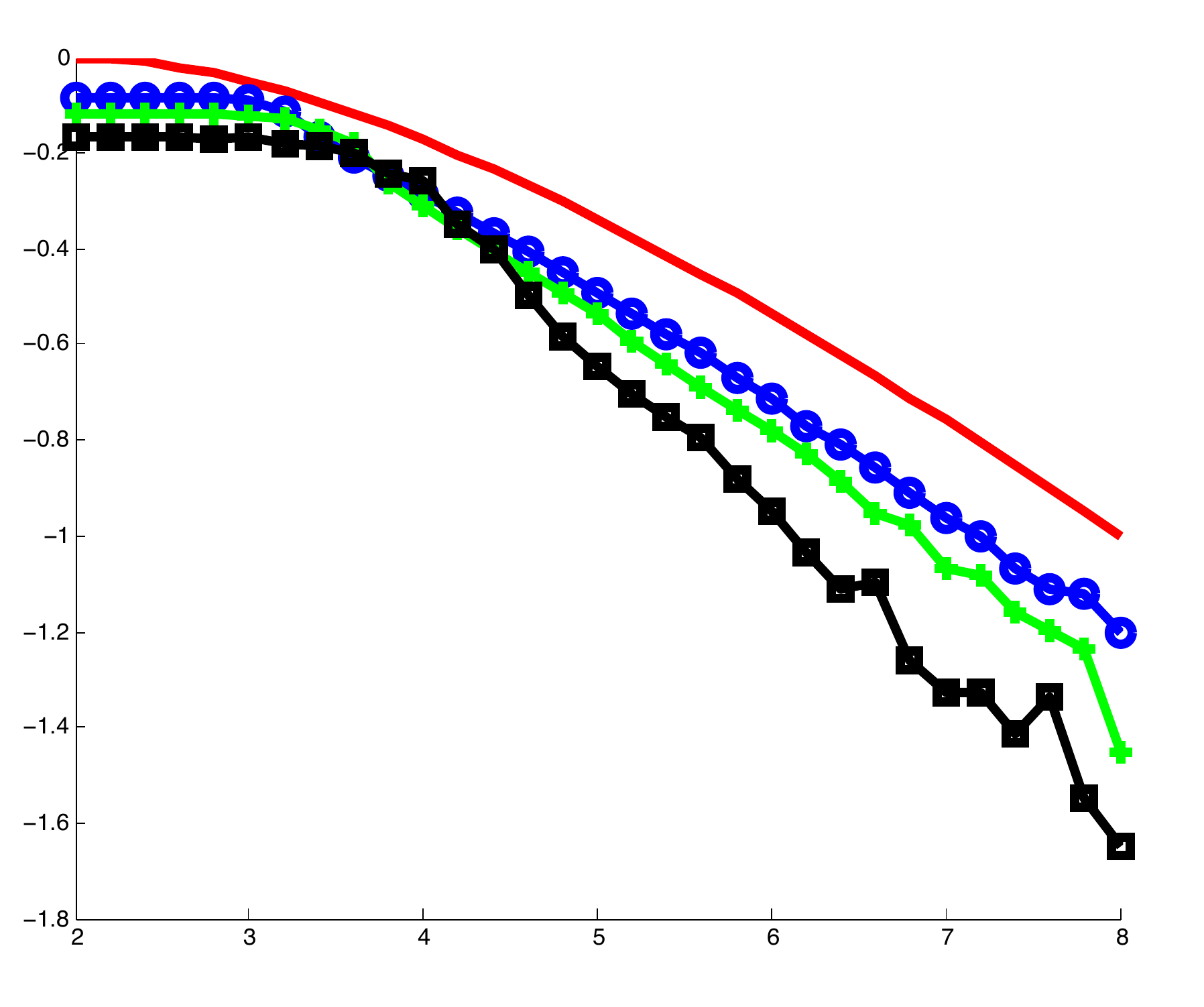}
    \end{subfigure}
    \caption{Vanilla spectral method for SBM. \textbf{Left:} phase transition of exact recovery. The $x$-axis is $b$, the $y$-axis is $a$, and lighter pixels represent higher chances of success. Two red curves $\sqrt{a} - \sqrt{b} = \pm \sqrt{2}$ represent theoretical boundaries for phase transtion, matched by numerical results. \textbf{Right:} mean misclassification rates on the logarithmic scale with $b = 2$. The $x$-axis is $a$, varying from $2$ to $8$, and the $y$-axis is $\log \E r(\hat z, z) / \log n$. No marker: theoretical curve; circles: $n=5000$; crosses: $n=500$; squares: $n=100$. 
}\label{fig:SBM1}
\end{figure}

\subsection{Matrix completion from noisy entries}\label{sec:simMC}
Finally we come to experiments of matrix completion from noisy entries. The performance of the spectral algorithm in terms of root-mean squared error (RMSE) has already been demonstrated in \cite{KMO101}, among others. In this part, we focus on the comparison between the maximum entrywise errors and RMSEs, for both the singular vectors and the matrix itself. The settings are mainly adopted from \cite{CPl10} and \cite{KMO101}. Each time we first create a rank-$r$ matrix $M^* \in \R ^{n\times n}$ using the product $M_L M_R^T$, where $M_L,M_R \in \R^{n\times r}$ have i.i.d. $N(0,20/\sqrt{n})$ entries. Then, each entry of $M^*$ is picked with probability $p$ and contaminated by random noise drawn from $N(0,\sigma^2)$, independently of others. While increasing $n$ from $500$ to $5000$ by $500$, we choose $p=\frac{10 \log n}{n}$, fix $r=5$ and $\sigma=1$. All the data presented in the plot are averaged over $100$ independent experiments.

In support of our discussions in Section \ref{sec:NMC}, Figure \ref{fig:NMC} shows that the following two ratios 
\begin{align*}
&	R_{\mathrm{mat}}=\frac{ \| U \Sigma V^T - M^* \|_{\max} }
	{ \eta^2 \sqrt{\log n} \cdot \| U \Sigma V^T - M^*\|_F },\\
&R_{\mathrm{vec}}=\frac{ \max \{ \| U\sgn(H) - U^* \|_{2\to\infty} , \|V \sgn(H) - V^* \|_{2\to\infty} \} }
	{\eta \sqrt{\log n} \cdot \max \{ \| U\sgn(H) - U^* \|_F , \|V \sgn(H) - V^* \|_F \}
	},
\end{align*}
approximately remain constant as $n$ grows. Here the RMSEs $n^{-1} \| U \Sigma V^T - M^*\|_F$ and $n^{-1/2}\max \{ \| U\sgn(H) - U^* \|_F , \|V \sgn(H) - V^* \|_F \}$ are scaled by $(\sqrt{n}\eta)^2 \sqrt{\log n}$ and $( \sqrt{n} \eta)\sqrt{\log n}$, respectively. Hence our analysis is sharp, and the perturbations are obviously delocalized among the entries.

\begin{figure}
	\centering
	\includegraphics[scale=0.35]{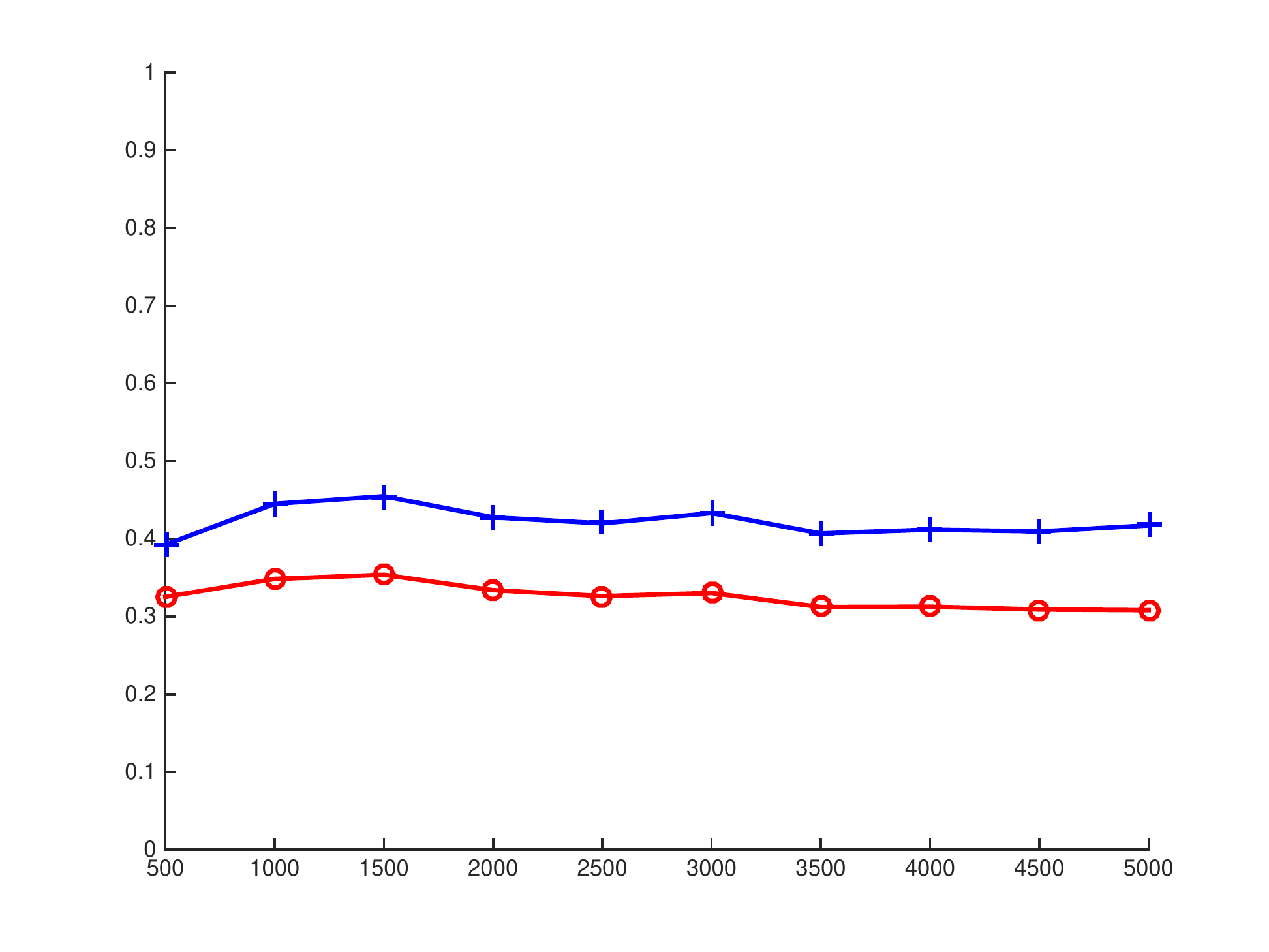}
	\caption{ $R_{\mathrm{mat}}$ and $R_{\mathrm{vec}}$ in matrix completion from noisy entries. The $x$-axis is $n$, varying from $500$ to $5000$ by $500$, and the $y$-axis is the ratio. Crosses and circles stand for $R_{\mathrm{mat}}$ and $R_{\mathrm{vec}}$, respectively.
	}\label{fig:NMC}
\end{figure}

\section{Discussions}\label{sec:discuss}

We have developed first-order approximations for eigenvectors and eigenspaces with small $\ell_\infty$ errors under random perturbations. These results lead to sharp guarantees for three statistical problems.

Several future directions deserve exploration. First, the main perturbation theorems are currently stated only for symmetric matrices. We think it may be possible to extend the current analysis to SVD of general rectangular matrices, which has broader applications such as principal component analysis. Second, there are many other graph-related matrices beyond adjacency matrices, including graph Laplacians and non-backtracking matrices, which are important both in theory and in practice. Third, we believe our assumption of row- and column-wise independence can be relaxed to block-wise independence,
which is relevant to cryo-EM and other problems. 

Finally, in our examples, the spectral algorithm is strongly consistent if and only if the MLE is, though the latter can be NP-hard to compute in general.
It would be interesting to see how general this phenomenon is, in view of better understanding the statistical and computational tradeoffs.

\section*{Acknowledgements}
The authors thank Harrison Zhou, Amit Singer, Nicolas Boumal, Yuxin Chen and Cong Ma for helpful discussions.

\newpage

\appendix

\section{Outline of proofs}\label{sec:outline}
In this section, we first present key observations leading to the inequalities \eqref{ineq:thm0a} and \eqref{ineq:thm0b} for the eigenvector case. The analysis and insights for eigenvector perturbation will be instrumental for the general eigenspace perturbation result. Then, we outline the proof ideas for Theorem \ref{thm-main}. 

\subsection{Warm-up analysis of eigenvector perturbation}\label{sec:outline1}
Let us consider the simpler setting in Section~\ref{sec:eigenspace}: we assume a rank-one structure: $A^* = \lambda^* u^* (u^*)^T$ (dropping the subscript for simplicity). By Weyl's inequality, the leading eigenvalue $\lambda$ of $A$ satisfies $| \lambda - \lambda^*| \le \| A-A^* \|_2$, and by the spectral norm concentration Assumption \hyperref[cond-3]{A3}, we obtain $ | \lambda - \lambda^*| \le \| A-A^* \|_2\le \gamma \lambda^*$. Note $\gamma < 1/2$ under Assumption \hyperref[cond-3]{A3}, so $ \lambda\ge  \lambda^*/2$. By the triangle inequality, we have
\begin{align}\label{ineq:warmup1}
\| u \|_{\infty} = \left\| \frac{Au}{\lambda} \right\|_{\infty} \le  \left\| \frac{Au^*}{\lambda} \right\|_{\infty}  + \left\| \frac{A(u - u^*)}{\lambda} \right\|_{\infty}
\le \frac{2}{\lambda^*}\left( \| Au^* \|_{\infty} + \| A(u - u^*) \|_{\infty} \right).
\end{align}
Likewise, using $u = Au / \lambda$ and $|\lambda^{-1} - (\lambda^*)^{-1} | = |\lambda - \lambda^*| / |\lambda \lambda^*| \le 2\gamma/\lambda^*$, we have
\begin{align}\label{ineq:warmup2}
\left\| u - \frac{A u^*}{\lambda^*}  \right\|_{\infty} \le \left| \frac{1}{\lambda} - \frac{1}{\lambda^*} \right| \left\|A u^* \right\|_{\infty} +  \frac{1}{\lambda} \left\|A(u -  u^*) \right\|_{\infty}   \le \frac{2}{\lambda^*}\left( \gamma \| Au^* \|_{\infty} + \| A(u - u^*) \|_{\infty} \right)
\end{align}
Note that $Au^* = \lambda^* u^* + (A-A^*)u^*$, so it is easy to bound $\| Au^* \|_{\infty}$ using the row concentration assumption: in Assumption \hyperref[cond-4]{A4}, we set $w = u^*$, and the row concentration inequality \eqref{eqn-cond-4} and the union bound imply $\| (A-A^*)u^*\|_{\infty} \le \lambda^* \varphi(1) \| u^* \|_{\infty}$ with probability $1 - \delta_1/r$ (recall $\Delta^* = \lambda^*$). Then, the goal is to obtain a good bound on
$ \| A(u - u^*) \|_{\infty} $.

However, the random quantities $A$ and $u - u^*$ are dependent, and we cannot directly use the row concentration assumption. To resolve this issue, we use a leave-one-out technique similar to the ones in \cite{bean2013optimal}, \cite{JMo15}, and \cite{ZhoBou17}. Define $n$ auxiliary matrices $A^{(1)}, A^{(2)}, \cdots, A^{(n)} \in \mathbb{R}^{n \times n}$ as follows: for any $m \in [n]$, let
\begin{equation}\label{def:aux}
[A^{(m)}]_{ij}=A_{ij}\mathbf{1}_{\{i\neq m,j\neq m\}},
\end{equation}
where $\mathbf{1}$ is the indicator function. By definition, for all $m \in [n]$, $A^{(m)}$ is a symmetric matrix, and its entries are identical to those of $A$ except that entries in its $m$th row and column are zero. The row and column-wise independence Assumption \hyperref[cond-2]{A2} then implies that $A-A^{(m)}$ and $A^{(m)}$ are independent. This simple observation is the cornerstone to decouple dependence.

For each $m \in [n]$, let $u^{(m)}$ be the leading eigenvector of $A^{(m)}$ with the appropriately chosen sign. By the triangle inequality and the definition of $\| \cdot \|_{2 \to \infty}$, we have
\begin{align}
| [A(u - u^*)]_m | &= | A_{m \cdot} (u - u^*)| \le  | A_{m \cdot} (u - u^{(m)})| + | A_{m \cdot} (u^{(m)} - u^*)|  \label{ineq:warmup3-0} \\
&\le \| A \|_{2 \to \infty} \| u - u^{(m)} \|_2 + | A_{m \cdot} (u^{(m)} - u^*)|. \label{ineq:warmup3}
\end{align}
Here recall $A_{m \cdot}$ is the $m$th row vector of $A$. The advantage of introducing $u^{(m)}$ is pronounced in the second term of \eqref{ineq:warmup3}: the $m$th row $A_{m \cdot}$  and $u^{(m)} - u^*$ are independence. This is because by definition, $u^{(m)}$ only depends on $A^{(m)}$ and thus is independent of $A_{m \cdot}$. Using the row concentration Assumption \hyperref[cond-4]{A4}, we can show that $| A_{m \cdot} (u^{(m)} - u^*)| \lesssim \lambda^*(\gamma + \varphi(\gamma)) (\|u\|_{\infty} + \| u^* \|_{\infty})$. The term $\|u\|_{\infty}$ will ultimately be absorbed into the left-hand side of \eqref{ineq:warmup1} after rearrangement.

A crucial step in the proof is a sharp bound on $\| u - u^{(m)} \|_2$. Viewing $u^{(m)}$ as a perturbed version of $u$, we use a (proper) form of Davis-Kahan's $\sin \Theta$ theorem \citep{DavKah70} to obtain
\begin{equation*}
\| u - u^{(m)} \|_2 \lesssim \frac{\| (A - A^{(m)}) u \|_2}{\lambda^*} = \frac{\| v \|_2}{\lambda^*}, \qquad \text{where } v:=  (A - A^{(m)}) u.
\end{equation*}
An important feature of this bound is that, in the numerator, $A - A^{(m)}$ only has nonzero entries in the $m$th row and $m$th column. Consider bounding $m$th entry of $v$ and its other entries separately, we have
\begin{align*}
&|v_m| = |A_{m \cdot} u| = | [A u]_{m}| = |\lambda u_m| \le |\lambda| \| u \|_{\infty}, \\
&\Big( \sum_{i \neq m} v_i^2 \Big)^{1/2} = \Big( \sum_{i \neq m} A_{mi}^2 u_m^2 \Big)^{1/2} \le \| A \|_{2 \to \infty} \| u \|_{\infty} \le 2\gamma |\lambda^*| \| u \|_{\infty},
\end{align*}
where the last inequality is due to $\| A \|_{2 \to \infty} \le \|A^*\|_{2 \to \infty} + \|A - A^*\|_2 \le 2\gamma\lambda^*$ under Assumption \hyperref[cond-1]{A1} and \hyperref[cond-3]{A3}. This will lead to a sharp bound $\| u - u^{(m)} \|_2 \lesssim \| u \|_{\infty}$, and therefore a good control of the first term of \eqref{ineq:warmup3}:
\begin{equation}\label{ineq:warmup4}
\| A \|_{2 \to \infty} \| u - u^{(m)} \|_2 \lesssim \gamma\lambda^* \| u \|_{\infty}.
\end{equation}
After rearrangement of \eqref{ineq:warmup1}, the terms involving $\| u \|_{\infty}$ will be absorbed into the left-hand side, and we will obtain the first inequality \eqref{ineq:thm0a}. Once \eqref{ineq:thm0a} is proved, we can then bound the two terms in \eqref{ineq:warmup3} in terms of $\| u^* \|_{\infty}$:
\begin{align*}
& \| A \|_{2 \to \infty} \| u - u^{(m)} \|_2 \lesssim \gamma\lambda^* (1 + \varphi(1)) \| u^* \|_{\infty}, \\
& | A_{m \cdot} (u^{(m)} - u^*)| \lesssim \lambda^*(\gamma + \varphi(\gamma)) (1 + \varphi(1)) \| u^* \|_{\infty}.
\end{align*}
Using the above two bounds, together with the bound on $\| A u^* \|_{\infty}$, we simplify \eqref{ineq:warmup2} and derive
\begin{equation*}
\left\|u -  \frac{A u^*}{\lambda^*} \right\|_{\infty} \lesssim \gamma \varphi(1) \| u^*\|_{\infty} + (\gamma + \varphi(\gamma)) (1 + \varphi(1)) \| u^* \|_{\infty},
\end{equation*}
which leads to the second inequality \eqref{ineq:thm0b}.

\subsection{Proof ideas for Theorem \ref{thm-main}}\label{sec:outline2}

Unlike the eigenvector case, the eigenspaces are up to a orthogonal matrix, so we have to study $H$ and $\sgn(H)$. Some basic properties about $H$ are stated in Lemma \ref{lem:H}. Moreover, following the decoupling idea in \ref{sec:outline1}, for each $m \in [n]$, we let $U^{(m)} = [u^{(m)}_{s+1}, \cdots, u^{(m)}_{s+r} ]\in \mathbb{R}^{n \times r}$ be the similar matrix as $U$. And in addition, we define additional $n$ auxiliary matrices $H^{(1)}, H^{(2)}, \cdots, H^{(n)} \in \mathbb{R}^{r \times r}$:
\begin{equation}\label{def:Hm}
H^{(m)} = [U^{(m)}]^T U^*.
\end{equation}

The purpose of introducing $U^{(m)}$ and $H^{(m)}$ is to ensure independence---see the comments below \eqref{ineq:warmup3}. The next lemma is a deterministic result, in parallel with the inequalities \eqref{ineq:warmup1} and \eqref{ineq:warmup2} derived for the eigenvector case.

\begin{lem}\label{lem-approx-0}
Let $\bar{\gamma} = \|E\|_2 / \Delta^*$. If $ \bar\gamma \leq 1/10$, then for any $m\in[n]$ we have
\begin{align}
	&\| (UH)_{m\cdot}  \|_{ 2} \leq \frac{2}{\Delta^*}( \|A_{m\cdot } U^*\|_2+\|A_{m\cdot}(UH- U^*)\|_2 ), \label{ineq:UHmax}\\
	&\norm{ ( UH - A U^*(\Lambda^*)^{-1} )_{m\cdot} }_2
	\le
	\frac{6\bar\gamma}{\Delta^*} \|A_{m\cdot } U^*\|_2 + \frac{2}{\Delta^*} \|A_{m\cdot} (UH- U^*)\|_2 .  \label{ineq:UHminusmax}
\end{align}
\end{lem}

The proof of Lemma \ref{lem-approx-0}, which is in the appendix, is technical by nature. This is caused by the fact that $H$ and $\Lambda$ are not commutative (whereas in the eigenvector case, $H=u^T u^* \in \R$ and $\lambda$ is commutative). By Lemma \ref{lem-approx-0} and the triangle inequality, we have the following bound for all $m$'s, which is similar to \eqref{ineq:warmup3-0}:
\begin{align}
&\| ( UH )_{m\cdot} \|_2 \leq \frac{2}{\Delta^*}( \|A_{m\cdot } U^*\|_2 +\|A_{m\cdot}(UH- U^*)\|_2 ) \notag \\
&\leq \frac{2}{\Delta^*}  \left( \|A_{m\cdot } U^*\|_2 +\|A_{m\cdot}(UH- U^{(m)} H^{(m)}) \|_2
+\|A_{m\cdot}(U^{(m)} H^{(m)}-U^*) \|_2 \right). \label{eqn-decomposition-0}
\end{align}

A deterministic argument shows that the second term in \eqref{eqn-decomposition-0}, namely $2(\Delta^*)^{-1}\|A_{m\cdot}(UH- U^{(m)} H^{(m)}) \|_2$, is a vanishing proportion of $\| (UH)_{m\cdot }\|_2$, so after rearrangement, we only need to bound the first term and the third term---see \eqref{ineq:AUUm} and \eqref{eqn-decomposition-1} in Lemma \ref{lem-approximation-error-auxiliary}. This is similar to the derivation of \eqref{ineq:warmup4}, in which we use a proper form of Davis-Kahan's theorem. The first term, by the row concentration Assumption \hyperref[cond-4]{A4}, can be controlled fairly easily---see Lemma \ref{lem:AU}. The third term, by the row concentration assumption again, can be bounded by a vanishing proportion of $\| UH  \|_{2\to\infty} + \| U^* \|_{2\to\infty}$, and finally, this leads to a bound on $\| UH \|_{2\to\infty}$ after rearranging (\ref{eqn-decomposition-0}). It is vital that the function $\varphi(x)$ in the row concentration assumption sharply captures the concentration for non-uniform weights, and this allows a good control of the third term---see Lemma \ref{lem:AV}.

All these arguments lead to a bound on $\| UH \|_{2\to\infty}$ that is roughly $O( \| U^* \|_{2\to\infty} + \| A^* \|_{2 \to \infty} / \Delta^* )$, and in many applications, this is $O( \| U^* \|_{2\to\infty})$. In other words, this says $\| UH \|_{2\to\infty }$ inflates $\| U^*\|_{2\to\infty}$ by at most a constant factor.

Similar to the eigenvector case, the bound on $\| UH \|_{2\to\infty }$ finally leads to a sharp bound on $\|UH - A U^*(\Lambda^*)^{-1}\|_{2\to\infty}$, and thus proves Theorem \ref{thm-main}. To understand this, assuming $\|A-A^*\|_2 \leq \gamma \Delta^*$, which happens with probability at least $ 1 -\delta_0$, the triangle inequality yields
\begin{align}
\|UH - A U^*(\Lambda^*)^{-1}\|_{2\to\infty} &\leq \frac{6  \gamma }{\Delta^*} \| AU^*\|_{2\to\infty} + \frac{2}{\Delta^*} \max\limits_{m\in[n]}\|A_{m\cdot}(UH- U^{(m)} H^{(m)}) \|_2  \notag \\
& + \frac{2}{\Delta^*} \max \limits_{m\in[n]}\|A_{m\cdot}(U^{(m)} H^{(m)}-U^*) \|_2. \label{eqn-decomposition-2}
\end{align}
As argued above, we can bound the first term and the second term using Lemma \ref{lem-approximation-error-auxiliary} and \ref{lem:AU}. Now that we have a bound on $\| UH \|_{2\to\infty}$, Lemma \ref{lem:AV} immediately implies that with probability $1-\delta_1$, the third term is a small proportion of that bound, or equivalently $O((\gamma+\varphi(\gamma)) (\| U^* \|_{\max} + \| A^* \|_{2 \to \infty} / \Delta^* ))$. Note that, if $\gamma = o(1)$ as $n \to \infty$, the final bound is an order smaller than $\| U^* \|_{\max} +  \| A^* \|_{2 \to \infty} / \Delta^*$.

\section{Proofs for Section~2} \label{sec:proofsFor2}

\subsection{Deterministic lemmas}
This subsection states and proves deterministic results useful for the proof. We temporarily ignore Assumptions \hyperref[cond-2]{A2}-\hyperref[cond-4]{A4} and may only assume \hyperref[cond-1]{A1}. First we present basic properties of $H = U^T U^*$, which is shown to be close to the orthonormal matrix $\sgn(H)$. Both of them play an important role in aligning $U$ with $U^*$. The techniques used in dealing with $H$ and $\sgn(H)$ are similar to the ones in \cite{FWW17}.
\begin{lem}\label{lem:H}
	$\| H \|_2 \leq 1$. When $\bar{\gamma} = \|E \|_2 / \Delta^* \leq 1$, we have
	\begin{equation}
	\| H - \sgn(H) \|_2^{1/2} \leq
	\|UU^T -U^*(U^*)^T\|_2 \leq  \frac{\|E U^* \|_2}{(1-\bar{\gamma}) \Delta^*} \leq \frac{\bar{\gamma}}{1-\bar{\gamma}},
	\end{equation}
	and $\| \Lambda H - H\Lambda\|_2\leq 2\|E\|_2$. When $\bar{\gamma} \leq 1/2$, we have $\norm{ H^{-1} }_2 \le (1- \bar{\gamma})^2 / (1-2 \bar{\gamma} )$ and
	\begin{equation}\label{ineq:key1}
	\norm{ U_{m\cdot} (H\Lambda - \Lambda H) }_{2} \le \frac{2(1-\bar{\gamma})^2 }{1-2\bar{\gamma}} \| E \|_2 \norm{ (UH)_{m\cdot} }_2,~\forall m\in[n].
	\end{equation}
\end{lem}
\begin{proof}[\bf Proof of Lemma \ref{lem:H}]
	
	First, we have $ \|H \|_2 = \| U^T U^* \|_2 \leq  \| U^T \|_2 \| U^* \|_2 = 1$. Let the SVD of $H = U^T U^*$ be $\bar U \bar \Sigma \bar V^T$, where $\bar U, \bar V \in \R^{r \times r}$ are orthogonal matrices and $\bar \Sigma = \diag\{ \bar \sigma_1,\cdots, \bar \sigma_r\}$ is a diagonal matrix. Then $\sgn(H) = \bar{U} \bar V^T$. By \cite[Chp I, Cor 5.4]{SSu90}, the singular values $1\geq \bar \sigma_1\geq \cdots\geq  \bar \sigma_r\geq 0$ are the cosines of canonical angles $0\leq \bar \theta_1 \leq \cdots \leq  \bar \theta_r \leq \pi/2$ between the column spaces of $U$ and $U^*$. We have $ \| \sgn(H) - H \|_2 = 1- \cos \bar{\theta}_r $.
	
	The Davis-Kahan $\sin \Theta$ theorem \citep{DavKah70} forces $\sin \bar \theta_r \leq \norm{U^* (U^*)^T E}_2/\delta \leq \norm{E U^*}_2/\delta$, where $\delta =(\lambda_s - \lambda_{s+1}^*)_+ \wedge (\lambda_{s+r}^* - \lambda_{s+r+1})_+$ and we define $x_+ = x \vee 0$ for $x \in \R$. On the one hand, Weyl's inequality \cite[Chp IV, Cor 4.9]{SSu90} leads to $\delta \geq \Delta^* - \|E\|_2 \geq (1-\bar{\gamma}) \Delta^*$. On the other hand, $ \sin \bar{\theta}_r = \| UU^T - U^* (U^*)^T\|_2 $ follows from \cite[Chp I, Thm 5.5]{SSu90}. Besides, $\bar \theta_r \in [0,\pi/2]$ forces $ 0 \leq \cos \bar{\theta}_r \leq 1  $, $\cos \bar{\theta}_r \geq \cos^2 \bar{\theta}_r = 1-\sin^2 \bar{\theta}_r $ and
	\begin{equation*}
		\| H - \sgn (H) \|_2^{1/2} = \sqrt{1 - \cos \bar{\theta}_r} \leq \sin \bar{\theta}_r = \|UU^T -U^*(U^*)^T\|_2\leq  \frac{ \|E U^* \|_2 } {(1-\bar{\gamma})\Delta^*} \leq \frac{\bar{\gamma}}{1-\bar{\gamma}}.
	\end{equation*}
	
	Note that $U^T A = \Lambda U^T$ and $A^* U^* = U^* \Lambda^*$. We have
	\begin{equation*}
		U^T E U^* = U^TAU^* - U^T A^* U^* = \Lambda U^T U^* - U^T U^* \Lambda^* = \Lambda H - H\Lambda^*.
	\end{equation*}	
	By the triangle inequality,
	\begin{equation*}
		\| \Lambda H - H\Lambda\|_2= \|U^T E U^* + H(\Lambda^* - \Lambda)\|_2 \leq \|U^T E U^*\|_2 + \| H(\Lambda^* -\Lambda)\|_2 \leq 2\|E\|_2,
	\end{equation*}
	where we used $\| \Lambda - \Lambda^* \|_2 \le \norm{E}_2$ and $\| H \|_2 \leq 1$.
	
	From now on we assume that $\bar{\gamma } \leq 1/2 $. $\| H - \sgn (H) \|_2 \leq [ \bar{\gamma} /(1-\bar{\gamma})] ^2 \leq 1$ leads to
	\begin{align}
		&\| H^{-1} - \sgn(H)^{-1} \|_2 \leq \frac{ \|\sgn(H)^{-1} [H - \sgn(H) ]\|_2 }{1- \| \sgn(H)^{-1} [H - \sgn(H) ]\|_2} \leq \frac{ \|H - \sgn(H) \|_2 }{1- \| H - \sgn(H) \|_2} \leq \frac{\bar{\gamma}^2}{ 1 - 2 \bar{\gamma}},\notag\\
		&\| H^{-1} \|_2 \leq \| \sgn(H)^{-1} \|_2+ \| H^{-1} - \sgn(H)^{-1} \|_2  \leq 1 + \frac{\bar{\gamma}^2}{ 1 - 2 \bar{\gamma}} = \frac{(1-\bar{\gamma})^2}{1-2\bar{\gamma}}.\notag
	\end{align}
	Finally we come to the last claim.
	\begin{align*}
		\norm{  U_{m\cdot} (H\Lambda - \Lambda H) }_2
		&= \norm{(UH)_{m\cdot} H^{-1}(H\Lambda - \Lambda H) }_2 \le \norm{ ( UH )_{m\cdot} }_2 \norm{H^{-1}}_2 \norm{H \Lambda - \Lambda H}_2 \\
		&\le \norm{ ( UH )_{m\cdot} }_2 \cdot \frac{(1-\bar{\gamma})^2 }{1-2\bar{\gamma}} \cdot 2 \norm{E}_2 .
	\end{align*}
\end{proof}

Next, we prove Lemma \ref{lem-approx-0}. This will soon leads to simplified bounds in Lemma \ref{lem-approximation-error-auxiliary}.

\begin{proof}[{\bf Proof of Lemma \ref{lem-approx-0}}]
	Define $\Lambda=\diag(\lambda_{s+1},\cdots,\lambda_{s+r}) \in \R^{r \times r}$. By Weyl's inequality, $\max_{i \in [r]} |\lambda_{s+i} - \lambda_{s+i}^* | = \| \Lambda - \Lambda^*\|_2 \le \norm{E}_2 \le \bar{\gamma} \Delta^*$, and so $\min_{i \in [r]} |\lambda_{s+i}| \ge (1-\bar{\gamma})\Delta^*$. Moreover, since  $AU=U\Lambda$, we have
	\begin{equation}\label{proof-thm-main-decomposition}
	UH \Lambda - A U^*
	= U(H\Lambda - \Lambda H) + A(UH- U^*) .
	\end{equation}
	Note that when $r=1$, $H$ and $\Lambda$ are scalars, so the term involving $H\Lambda - \Lambda H$ vanishes. By the triangle inequality and Weyl's inequality, the entries of $U$ are easy to bound. However, for a general $r$, the trouble that $H$ and $\Lambda$ do not commute requires more work.
	
	We multiply (\ref{proof-thm-main-decomposition}) by $\Lambda^{-1}$ on the right, use Lemma \ref{lem:H} and the triangle inequality to derive
	\begin{align}
		\| ( UH  - A U^*\Lambda^{-1} )_{m\cdot} \|_2
		&\leq \| U _{m\cdot} ( H \Lambda - \Lambda H ) \Lambda^{-1}  \|_2 + \| A_{m\cdot} ( UH - U^* ) \Lambda^{-1} \|_2
		\notag \\
		&\leq
		\frac{ 1 }{\min_{i \in [r]} |\lambda_{s+i}|}
		\left(
		\frac{2(1-\bar{\gamma})^2 }{1-2\bar{\gamma}}\| E \|_2 \| (UH)_{m\cdot} \|_2 + \|A_{m\cdot} (UH- U^*)\|_2
		\right)
		\notag \\
		&\le  \frac{2\bar{\gamma}}{1-2\bar{\gamma} }  \| (UH)_{m\cdot} \|_2+ \frac{1}{(1-\bar{\gamma} ) \Delta^*} \|A_{m\cdot} (UH- U^*)\|_2. \label{proof-thm-main-3}
	\end{align}
	On the other hand,
	\begin{equation}\label{ineq:UHlb}
	\| (UH  - A U^* \Lambda^{-1})_{m\cdot} \|_2 \ge \norm{ (UH)_{m\cdot} }_2 - \frac{\| A_{m\cdot} U^*\|_2 }{\min_{i \in [r]} |\lambda_{s+i}| } \ge \norm{ (UH)_{m\cdot} }_2 - \frac{\norm{ A_{m\cdot} U^*}_2 }{(1- \bar{\gamma} )\Delta^*}.
	\end{equation}
	Combining this bound with (\ref{proof-thm-main-3}),
	\begin{equation}\label{proof-thm-main-4}
	\| (UH)_{m\cdot} \|_2 \leq \frac{1-2 \bar{\gamma} }{(1-\bar{\gamma}) (1-4\bar{\gamma} )}\cdot \frac{1}{\Delta^*}\big( \|A_{m\cdot } U^*\|_2+ \|A_{m\cdot}(UH- U^*)\|_2 \big).
	\end{equation}
	Since our condition $\bar{\gamma} \le 1/10$ implies $(1- \bar{\gamma} ) (1-4\bar{\gamma}) \geq 1/2$, the first inequality of this lemma follows from \eqref{proof-thm-main-4}. Besides, (\ref{proof-thm-main-4}) leads to a further upper bound of (\ref{proof-thm-main-3}):
	\begin{align}
		\| ( UH  - A U^*\Lambda^{-1} )_{m\cdot} \|_2 &\le
		\frac{2\bar{\gamma} \|A_{m\cdot } U^*\|_2 + (1-2\bar{\gamma}) \|A_{m\cdot}(UH- U^*)\|_2 }{(1-\bar{\gamma}) (1-4\bar{\gamma}) \Delta^*} \notag \\
		& \le \frac{4\bar{\gamma}}{\Delta^*} \|A_{m\cdot } U^*\|_2 + \frac{2}{\Delta^*}  \|A_{m\cdot}(UH- U^*)\|_2 \label{ineq:UHub1}.
	\end{align}
	Similar to (\ref{ineq:UHlb}), we have another lower bound
	\begin{align}\label{ineq:UHlb2}
		\| ( UH  - A U^*\Lambda^{-1} )_{m\cdot} \|_2 &\ge \norm{ ( UH - A U^*(\Lambda^*)^{-1} )_{m\cdot} }_2 - \norm{A_{m\cdot} U^*((\Lambda^*)^{-1} - \Lambda^{-1}) }_2 \notag \\
		&\ge \norm{ ( UH - A U^*(\Lambda^*)^{-1} )_{m\cdot} }_2 -  \frac{2\bar\gamma}{  \Delta^*} \|A_{m\cdot } U^*\|_2,
	\end{align}
	where we used $|(\lambda^*_{s+i})^{-1} - \lambda_{s+i}^{-1}| \le \| E \|_2 / |\lambda^*_{s+i} \lambda_{s+i}| \le
	\frac{\|E\|_2 }{ (1-\bar{\gamma}) (\Delta^*)^2 } \leq \frac{\bar\gamma}{ (1-\bar\gamma) \Delta^*} \le \frac{2 \bar\gamma }{\Delta^*}$. Combining (\ref{ineq:UHub1}) and (\ref{ineq:UHlb2}), we obtain the second claim.
\end{proof}

\begin{lem}\label{lem-approximation-error-auxiliary}
	Let Assumption \hyperref[cond-1]{A1} hold. If $\| A - A^* \|_2 \leq \gamma \Delta^*$, then for all $m\in[n]$ we have
	\begin{align}
		&\|UU^T- U^{(m)} (U^{(m)})^T\|_2\leq 3\kappa \|(UH)_{m\cdot}\|_2, \label{ineq:distUUm}
		\\
		&\|A_{m\cdot}(UH-U^{(m)}H^{(m)} )\|_{2} \leq 3\kappa \|A\|_{2\rightarrow\infty}\|(UH)_{m\cdot}\|_2. \label{ineq:AUUm}
	\end{align}
	As a consequence, we have the following bounds under the conditions above:
	\begin{align}
		&\| (UH)_{m\cdot} \|_2 \leq \frac{4}{\Delta^*} \left( \|A_{m\cdot} U^*\|_2
		+ \|A_{m\cdot}(U^{(m)} H^{(m)}-U^*) \|_2 \right), \label{eqn-decomposition-1}\\
		&\max_{m \in[n]} \| U^{(m)} H^{(m)}-U^* \|_2\leq 6 \gamma ,  \label{ineq:Vm2} \\
		&\max_{m \in[n]} \| U^{(m)} H^{(m)}-U^* \|_{2\to\infty} \leq 4 \kappa  \|UH\|_{2\to\infty } + \|U^*\|_{2\to\infty }. \label{ineq:Vmmax}
	\end{align}
\end{lem}

\begin{proof}[\bf Proof of Lemma \ref{lem-approximation-error-auxiliary} ]
	First, we prove the ``as a consequence" part. Under Assumption \hyperref[cond-1]{A1} the condition $\| A - A^* \|_2 \le \gamma  \Delta^*$, we have $\| A \|_{2 \to \infty} \le \| A - A^*\|_2 + \|A^*\|_{2 \to \infty} \le 2\gamma  \Delta^*$. Since $4\kappa \| A \|_{2 \to \infty} \le 8\kappa  \gamma  \Delta^* \le \Delta^* / 4$, from the bound \eqref{ineq:AUUm} we have
	\begin{equation*}
		\|A_{m\cdot}(UH-U^{(m)}H^{(m)} )\|_2 \le 4\kappa \|A\|_{2\rightarrow\infty}\|(UH)_{m\cdot}\|_2 \le \frac{\Delta^*}{4} \|(UH)_{m\cdot}\|_2 , \quad \forall \, m \in [n].
	\end{equation*}
	Using this bound to simplify \eqref{eqn-decomposition-0}, we obtain the desired inequality \eqref{eqn-decomposition-1} after rearrangement.
	
	Recall the definition of $H^{(m)}$ in \eqref{def:Hm}. The fact $(U^*)^T U^* = I_r$ yields $U^{(m)} H^{(m)}-U^*=[U^{(m)} (U^{(m)})^T-U^* (U^*)^T]U^*$. We use Lemma \ref{lem:H} to derive that
	\begin{equation*}
		\|U^{(m)} H^{(m)}-U^*\|_2 \leq \|U^{(m)} (U^{(m)})^T-U^* (U^*)^T\|_2  \leq \frac{\| A^{(m)} - A^*\|_2}{ \Delta^* - \| A^{(m)} - A^*\|_2}.
	\end{equation*}
	Since $A - A^{(m)}$ can only have nonzero entries in its $m$'th row and column (recall the definition \eqref{def:aux}) we have $\|A - A^{(m)}\|_2\le 2\| A \|_{2 \to \infty}$. Thus, by the triangle inequality and the assumption $\| A -A^* \|_2 \le \gamma \Delta^*$,
	\begin{equation*}
		\| A^{(m)} - A^*\|_2 \le \| A - A^* \|_2 + \|A - A^{(m)}\|_2 \le \gamma \Delta^* + 2(\| A^* \|_{2 \to \infty} + \| A - A^* \|_{2 \to \infty}) \le  5\gamma \Delta^*.
	\end{equation*}
	This inequality and the condition $\gamma \le 1/32$ leads to a new bound:
	\begin{equation*}
		\| U^{(m)} H^{(m)}-U^* \|_2\leq  \frac{ 5 \gamma \Delta^* }{ \Delta^* - 5\gamma \Delta^* } < 6 \gamma.
	\end{equation*}
	To get (\ref{ineq:Vmmax}), we use the triangle inequality and \eqref{ineq:distUUm} to obtain that
	\begin{align}
		&\| U^{(m)} H^{(m)}-U^* \|_{2\to\infty} \leq \|U^{(m)} H^{(m)}-UH\|_2 + \|UH\|_{2\to\infty} + \|U^*\|_{2\to\infty} \notag\\
		&= \|[U^{(m)} (U^{(m)})^T-UU^T]U^*\|_2 + \|UH\|_{2\to\infty} + \|U^*\|_{2\to\infty} \notag\\
		&\leq  (3\kappa  +1 ) \|UH\|_{2\to\infty} + \|U^*\|_{2\to\infty} \leq 4 \kappa \| UH \|_{2\to\infty} + \| U^* \|_{2\to\infty}.\notag
	\end{align}
	This finishes the proof of the ``as a consequence" part, and now we return to the first part. To bound $\|UU^T- U^{(m)} (U^{(m)})^T\|_2$, we will view $A^{(m)}$ as a perturbed version of $A$ and use Lemma \ref{lem:H} to obtain a bound. Let $\Delta := (\lambda_s-\lambda_{s+1}) \wedge (\lambda_{s+r}-\lambda_{s+r+1})$ be the gap that separates $\{ \lambda_{s+j} \}_{j=1}^r$ with the other eigenvalues of $A$. By Weyl's inequality, $|\lambda_i - \lambda_i^*| \le \| A - A^* \|_2$ for any $i \in [n]$, so $\Delta\geq \Delta^* - 2 \| A-A^* \|_2 \geq (1-2\gamma )\Delta^*$. Then by Lemma \ref{lem:H},
	\begin{equation*}
		\|UU^T- U^{(m)}(U^{(m)})^T\|_2\leq \frac{2\|(A-A^{(m)})U \|_2}{\Delta} \leq \frac{2 \|(A-A^{(m)})UH\|_2\|H^{-1}\|_2}{(1-2\gamma )\Delta^*}
	\end{equation*}
	By Lemma \ref{lem:H}, $\|H^{-1}\|_2 \le  1/(1-2 \gamma  )$. The condition $\gamma \leq 1/32$ implies
	\begin{equation}\label{proof-lem-approximation-1}
	\|UU^T- U^{(m)}(U^{(m)})^T\|_2 \leq \frac{2 \| (A-A^{(m)}) UH\|_2}{(1-2\gamma )^2 \Delta^*} \le \frac{2.3 \| (A-A^{(m)}) UH\|_2}{\Delta^*} =: \frac{2.3 \| B \|_2}{ \Delta^*}.
	\end{equation}
	Note that the entries of $A-A^{(m)}$ are identical to those of $A$ in the $m$'th row and $m$'th column, and are zero elsewhere. This leads us to consider bounding $\| B \|_2^2$ by two parts:
	\begin{equation}\label{ineq:B}
	\| B \|_2^2 \le \| B \|_F^2 =  \|B_{m\cdot}\|_2^2 + \sum_{i \neq m} \|B_{i\cdot}\|_2^2
	\end{equation}
	Observe that for any $i\neq m$, we have $B_{i\cdot}=(A-A^{(m)})_{i\cdot}UH=A_{im}(UH)_{m\cdot}$, so
	\begin{equation}\label{ineq:B1}
	\sum_{i \neq m} \|B_{i\cdot}\|_2^2 \le \|(UH)_{m\cdot}\|_2^2\sum_{i\neq m}A_{im}^2 \leq \|(UH)_{m\cdot}\|_2^2  \|A\|_{2\rightarrow\infty}^2.
	\end{equation}
	We also observe that $B_{m\cdot}=A_{m\cdot}UH=U_{m\cdot}\Lambda H$, where $A_{m\cdot}U = U_{m\cdot}\Lambda$ follows from the eigenvector definition. Thus, Lemma \ref{lem:H} implies that
	\begin{align}
		\norm{ B_{m\cdot} }_2 &= \norm{ (U\Lambda H)_{m \cdot} }_2 \le \norm{ (U H \Lambda)_{m \cdot} }_2 + \norm{ [ U(\Lambda H - H\Lambda) ]_m }_2 \notag \\
		&= \norm{ (U H \Lambda)_{m \cdot} }_2 + \norm{ (UH)_{m\cdot} H^{-1} (\Lambda H - H\Lambda)}_2 \notag \\
		& \le \|(UH)_{m\cdot}\|_2\|\Lambda\|_2 + \frac{2\|E\|_2}{1-2\gamma }\|(UH)_{m\cdot}\|_2 \notag \\
		&\leq \|(UH)_{m\cdot}\|_2 ( \|\Lambda^*\|_2 + \|\Lambda - \Lambda^*\|_2 + \frac{2\|E\|_2}{1-2\gamma } ) \notag \\
		&\leq (1+ \gamma  + \frac{2\gamma }{1-2\gamma })\|(UH)_{m\cdot}\|_2 \|\Lambda^*\|_2
		\le 1.1 \|(UH)_{m\cdot}\|_2 \|\Lambda^*\|_2\label{ineq:B2}
	\end{align}
	where we used $\| H^{-1} \|_2 \leq 1/(1-2 \gamma ) $, $ \| \Lambda H - H \Lambda \|_2 \leq 2 \|E \|_2 $, $\|\Lambda - \Lambda^*\|_2 \leq \|E\|_2 \leq \gamma  \Delta^* \le \gamma  \|\Lambda^*\|_2$ and $\gamma  \le1/32$.
	Since $\|A\|_{2\rightarrow\infty} \leq 2 \gamma  \Delta^*\leq 2 \gamma  \|\Lambda^*\|_2$, we combine the bounds (\ref{ineq:B1}) and (\ref{ineq:B2}), and simplify (\ref{ineq:B}):
	\begin{equation*}
		\|B\|_2^2\le \|(UH)_{m\cdot}\|_2^2 \;
		( 1.1^2\|\Lambda^*\|_2^2+ 4\gamma ^2 \|\Lambda^*\|_2^2)
		\leq \Big( 1.2\|(UH)_{m\cdot}\|_2\|\Lambda^*\|_2 \Big)^2.
	\end{equation*}
	The first inequality follows from (\ref{proof-lem-approximation-1}) and the above bound.
	
	Now since $UH = UU^TU^*$, $U^{(m)} H^{(m)}  = U^{(m)} (U^{(m)})^TU^* $, we use $\| U^* \|_2 \le 1$ to derive
	\begin{equation*}
		\|A_{m\cdot}(UH- U^{(m)} H^{(m)}) \|_2 =\|A_{m\cdot}[UU^T- U^{(m)} (U^{(m)})^T]U^* \|_2\leq  \|A\|_{2\rightarrow\infty}\|UU^T- U^{(m)} (U^{(m)})^T\|_2.
	\end{equation*}
	The second  inequality immediately follows from \eqref{ineq:distUUm}.
\end{proof}

\subsection{Proof of Theorem \ref{thm-main}}
In this subsection, we will prove Theorem \ref{thm-main}. We will state and prove Lemma \ref{lem:AU} and Lemma \ref{lem:AV}, which provide bounds on $\| AU^*\|_{2\to\infty }$ and $\max_{m \in [n]} \|A_{m\cdot}(U^{(m)} H^{(m)} - U^* ) \|_2$ respectively. Then, we combine these two lemmas together to prove Theorem \ref{thm-main}.

\begin{lem}\label{lem:AU}
	Let Assumption \hyperref[cond-4]{A4} hold. With probability at least $1 - \delta_1$, we have
	\begin{align}\label{ineq:AU}
		& \|(A-A^*) U^*\|_{2\to\infty }\le \Delta^* \, \varphi(1) \|U^*\|_{2\to\infty },\\
		&\|AU^*\|_{2\to\infty } \le \left( \|\Lambda^*\|_2 + \Delta^* \, \varphi(1)  \right) \| U^* \|_{2\to\infty }.
	\end{align}
\end{lem}
\begin{proof}[\bf Proof of Lemma \ref{lem:AU}]
	To get the first inequality we will use the row concentration assumption \hyperref[cond-4]{A4}. With probability at least $1 - \delta_1/n$, we have $ \|(A-A^*)_{m \cdot} U^*\|_2 \le \Delta^*  \, \varphi(1) \|U^*\|_{2\to\infty }$. Here we used the monotonicity of $\varphi$ and the fact that $\| U^*\|_F \leq \sqrt{n} \| U^*\|_{2\to\infty }$. Taking a union bound over $m$, we deduce that $ \|(A-A^*) U^*\|_{2\to\infty }\le \Delta^* \, \varphi(1) \|U^*\|_{2\to\infty}$ holds with probability at least $1 - \delta_1$. By the triangle inequality,
	\begin{equation}\label{ineq:AUstar}
	\|AU^*\|_{2\to\infty } \le \|A^* U^*\|_{2\to\infty } + \|(A-A^*) U^*\|_{2\to\infty }.
	\end{equation}
	Since $A^* U^* = U^*\Lambda^*$, the first term on the right-hand side of \eqref{ineq:AUstar} is bounded by $\|\Lambda^*\|_2   \|U^*\|_{2\to\infty}$.
	The lemma is proved by combining the bounds on the two terms in \eqref{ineq:AUstar}.
\end{proof}

\begin{lem}\label{lem:AV}
	Let Assumptions \hyperref[cond-1]{A1}-\hyperref[cond-4]{A4} hold. Denote $V^{(m)}=U^{(m)} H^{(m)}-U^* \in \mathbb{R}^{n \times r}$ for $m \in [n]$. There exists an event $\mathcal{E}_1$ with probability at least $1 - \delta_1$, such that on $\{ \| A-A^* \|_2\leq \gamma \Delta^* \} \cap \mathcal{E}_1$ we have
	\begin{equation}\label{ineq:AV}
	\max_{m\in[n]}\| A_{m \cdot} V^{(m)} \|_2 \le 6\gamma \|A^* \|_{2 \to \infty} + \Delta^* \, \varphi(\gamma)\left( 4 \kappa \|UH\|_{2\to\infty } + 6\|U^*\|_{ 2\to\infty} \right).
	\end{equation}
\end{lem}
\begin{proof}[\bf Proof of Lemma \ref{lem:AV}]
	By the triangle inequality,
	\begin{equation*}
		\| A_{m \cdot} V^{(m)} \|_2 \le \| A_{m \cdot}^* V^{(m)} \|_2 + \| (A - A^*)_{m \cdot} V^{(m)} \|_2
	\end{equation*}
	When $\mathcal{E}_0 = \{ \| A -A^* \|_2 \le \gamma \Delta^* \}$ happens (which has probability at least $1-\delta_0$ by Assumption \hyperref[cond-3]{A3}), we use \eqref{ineq:Vm2} to bound the first term:
	\begin{equation}\label{ineq:emm}
	\| A_{m \cdot}^* V^{(m)} \|_2 \le \| A_{m \cdot}^* \|_2 \| V^{(m)} \|_2 \le \| A^* \|_{2 \to \infty} \| V^{(m)} \|_2 \le 6\gamma \| A^* \|_{2 \to \infty}.
	\end{equation}
	To bound the second term, we use the row concentration Assumption \hyperref[cond-4]{A4}. For any $m \in [n]$ and $W \in \R^{n\times r}$, with probability at least $1 - \delta_1 / n$,
	\begin{equation*}
		\|(A-A^*)_{m\cdot} W \|_2 \le \Delta^* \,
		\varphi  \Big( \frac{\|W\|_F}{\sqrt{n}\, \| W\|_{2\to\infty}} \Big) \, \| W \|_{2\to\infty} = \Delta^* \,
		\varphi  \Big( \frac{\|W\|_F}{\sqrt{n}\, \| W\|_{2\to\infty}}  \Big) \, \frac{\sqrt{n}\| W \|_{2 \to \infty}}{\| W \|_F}  \frac{ \| W\|_F}{\sqrt{n}}.
	\end{equation*}
	From the facts that $\varphi(x)$ is increasing and $\varphi(x)/x$ is decreasing for $x \in [0,+\infty)$ we get
	\begin{align}\label{ineq:rowconctr2}
		\|(A-A^*)_{m\cdot} W \|_2 \le \begin{cases} \Delta^* \varphi(\gamma) \| W \|_{2\to \infty}, & \text{ if } \frac{\|W\|_F}{\sqrt{n}\, \| W\|_{2\to \infty}} \le \gamma \\
			\Delta^* \frac{\varphi(\gamma)}{\gamma} \frac{ \| W\|_F}{\sqrt{n}}, & \text{ if } \frac{\|W\|_F}{\sqrt{n}\,  \| W\|_{2\to\infty}} >\gamma \end{cases}
		\le \Delta^* \varphi( \gamma ) \left(  \| W \|_{2\to\infty} \vee \frac{\|W\|_F}{\sqrt{n}\, \gamma } \right).
	\end{align}
	Thanks to our leave-one-out construction, $(A-A^*)_{m\cdot}$ and $V^{(m)}$ are independent. If we define
	\begin{equation*}
		\mathcal{E}_1 = \bigcap_{m\in[n]} \left\{ \|(A-A^*)_{m\cdot} V^{(m)} \|_2 \leq \Delta^* \varphi( \gamma ) \left(  \| V^{(m)} \|_{2\to\infty} \vee \frac{\| V^{(m)} \|_F}{\sqrt{n}\, \gamma } \right) \right\},
	\end{equation*}
	then it follows from \eqref{ineq:rowconctr2}, Assumption \hyperref[cond-4]{A4} and union bounds that $\P ( \mathcal{E}_1 ) \geq 1-\delta_1 $.
	
	Now suppose $\mathcal{E}_0 \cap \mathcal{E}_1$ happens. \eqref{ineq:Vm2} forces $ \| V^{(m)}\|_F \leq \sqrt{r} \| V^{(m)} \|_2 \leq 6 \gamma \sqrt{r} $, and we obtain that for all $m \in [n]$,
	\begin{equation*}
		\|(A-A^*)_{m\cdot} V^{(m)}\|_2
		\leq \Delta^* \varphi(\gamma) \Big( \|V^{(m)}\|_{2\to\infty } \vee ( 6\sqrt{ r/n} ) \Big).
	\end{equation*}
	Since $\| U^* \|_{2\to\infty} \ge  \sqrt{r/n}$, we use \eqref{ineq:Vmmax} to simplify the above bound:
	\begin{equation}\label{ineq:emm2}
	\|(A-A^*)_{m\cdot} V^{(m)}\|_{2}
	\leq  \Delta^* \varphi(\gamma) \left( 4\kappa  \|UH\|_{2\to\infty} + 6\|U^*\|_{2\to\infty} \right).
	\end{equation}
	The proof is completed by combining the bounds \eqref{ineq:emm} and \eqref{ineq:emm2}.
\end{proof}

Finally we are ready for Theorem \ref{thm-main}.
\begin{proof}[\bf Proof of Theorem \ref{thm-main}]
	Lemma \ref{lem:H} forces that
	\begin{align*}
		& \| U \sgn(H) - UH \|_{2\to\infty} \leq \| UH \|_{2\to\infty}\|H^{-1}\|_2\|H - \sgn(H) \|_2 \lesssim \gamma^2 \| UH \|_{2\to\infty}.
	\end{align*}
	Thanks to this observation and $\| U \sgn(H) \|_{2\to\infty} = \|U\|_{2\to\infty}$, the first two inequalities in Theorem~\ref{thm-main} are implied by
	\begin{align}
		&\|U H \|_{2\to\infty} \lesssim \left( \kappa + \varphi(1) \right) \| U^* \|_{2\to\infty} + \gamma \| A^* \|_{2 \to \infty} / \Delta^*,\label{ineq:first-new}\\
		&\|U H - A U^*(\Lambda^*)^{-1}\|_{2\to\infty}\lesssim  \kappa ( \kappa + \varphi(1)	) ( \gamma + \varphi(\gamma) )
		\| U^*\|_{2\to\infty}
		+\gamma \| A^* \|_{2 \to \infty} / \Delta^*.\label{ineq:second-new}
	\end{align}
	Below we are going to show (\ref{ineq:first-new}), (\ref{ineq:second-new}), and finally the third inequality in Theorem~\ref{thm-main}. Let $\mathcal{E}$ be an event where $\| A- A^*\|_2 \leq \gamma \Delta^*$ and the followings hold:
	\begin{align}
		&\|UH\|_{2\to\infty} \leq \frac{4}{\Delta^*}(\|A U^*\|_{2\to\infty}
		+ \max _{m \in[n]}\|A_{m\cdot}(U^{(m)} H^{(m)}-U^*) \|_{2} ), \label{eq-final-1} \\
		&\| UH  - A U^*(\Lambda^*)^{-1} \|_{2\to\infty}\leq \frac{6 \gamma }{\Delta^*} \| A U^*\|_{2\to\infty}+
		\frac{2}{\Delta^*} \|A (UH- U^*)\|_{2\to\infty},\label{eq-final-2}\\
		& \|(A-A^*) U^*\|_{2\to\infty}\le \Delta^* \, \varphi(1) \|U^*\|_{2\to\infty},\label{eq-final-6}\\
		& \|AU^*\|_{2\to\infty} \le \left( \|\Lambda^*\|_2 + \Delta^* \, \varphi(1)  \right) \| U^* \|_{2\to\infty}, \label{eq-final-3} \\
		& \|A_{m\cdot}(UH-U^{(m)}H^{(m)} )\|_{2} \leq 3\kappa \|A\|_{2\rightarrow\infty}\|(UH)_{m\cdot}\|_2
		\leq 3 \kappa \cdot 2\gamma \Delta^* \cdot \|UH\|_{2\to\infty},~\forall m,\label{eq-final-4}\\
		&\| A_{m \cdot} (U^{(m)} H^{(m)}-U^*) \|_2 \le 6\gamma \|A^* \|_{2 \to \infty} + \Delta^* \, \varphi(\gamma)\left( 4 \kappa  \|UH\|_{2\to\infty} + 6\|U^*\|_{2\to\infty} \right),~\forall m.\label{eq-final-5}
	\end{align}
	It follows from Lemmas \ref{lem-approx-0}, \ref{lem-approximation-error-auxiliary}, \ref{lem:AU} and \ref{lem:AV} that $\P(\mathcal{E}) \geq 1 - \delta_0 - 2 \delta_1$. On the event $\mathcal{E}$,
	(\ref{eq-final-1}), (\ref{eq-final-3}) and (\ref{eq-final-5}) control $\| UH \|_{2\to\infty}$ from above:
	\begin{align*}
		\|UH\|_{2\to\infty} &\leq \frac{4}{\Delta^*}( \| \Lambda^*\|_2 + \Delta^*\, \varphi(1)) \| U^* \|_{2\to\infty} + \frac{4}{\Delta^*} \cdot 6\gamma \|A^* \|_{2 \to \infty}\\
		& + \frac{4}{\Delta^*} \cdot \Delta^*\, \varphi(\gamma) \left(
		4 \kappa \|UH\|_{2\to\infty} + 6 \|U^*\|_{2\to\infty} \right).
	\end{align*}
	Since $16 \kappa \varphi(\gamma) \le 1/2$ under Assumption \hyperref[cond-1]{A1}, we rearrange the inequality to eliminate $\|UH\|_{2\to\infty}$ on the right-hand side, and obtain
	\begin{equation}\label{ineq:UHfinal}
	\|UH\|_{2\to\infty} \leq \Big( \frac{8  \| \Lambda^*\|_2 }{\Delta^*} + 8 \varphi(1) + 48 \varphi(\gamma) \Big) \| U^* \|_{2\to\infty} + \frac{48 \gamma \| A^* \|_{2 \to \infty}}{\Delta^*}.
	\end{equation}
	(\ref{ineq:first-new}) follows from (\ref{ineq:UHfinal}), $\kappa =  \| \Lambda^*\|_2/\Delta^*$ and $\varphi(\gamma) \le \varphi(1)$ (by monotonicity of $\varphi$).
	Now we move on to (\ref{ineq:second-new}). On the event $\mathcal{E}$, by \eqref{eq-final-2} and the triangle inequality,
	\begin{align*}
		\|UH - A U^*(\Lambda^*)^{-1}\|_{2\to\infty} &\leq \frac{6  \gamma }{\Delta^*} \| AU^*\|_{2\to\infty} + \frac{2}{\Delta^*} \max\limits_{m\in[n]}\|A_{m\cdot}(UH- U^{(m)} H^{(m)}) \|_2 \notag \\
		& + \frac{2}{\Delta^*} \max \limits_{m\in[n]}\|A_{m\cdot}(U^{(m)} H^{(m)}-U^*) \|_2.
	\end{align*}
	Using \eqref{eq-final-3}--\eqref{eq-final-5}, the three terms above can be bounded by $6 \gamma (\kappa + \varphi(1) ) \|U^*\|_{2\to\infty}$, $12\kappa \gamma \|UH\|_{2\to\infty}$ and $ 12\gamma \frac{\| A^*\|_{2\rightarrow\infty}}{\Delta^*} + 2 \varphi(\gamma)( 4\kappa \|UH\|_{2\to\infty} + 6 \|U^*\|_{2\to\infty} )$ , respectively. Hence
	\begin{align*}
		\|UH - A U^*(\Lambda^*)^{-1}\|_{2\to\infty} &\lesssim
		\left(\kappa + \varphi(1) \right)\left(\gamma + \varphi(\gamma) \right)
		\| U^*\|_{2\to\infty}
		+ \gamma \|A^*\|_{2\rightarrow\infty} / \Delta^* \\
		& + \kappa  ( \gamma + \varphi(\gamma) )\|UH\|_{2\to\infty}.
	\end{align*}
	Plugging (\ref{ineq:first-new}) into this estimate and using the fact $32\kappa \max\{ \gamma, \varphi(\gamma) \} \leq 1$, we derive (\ref{ineq:second-new}).
	From $A^*U^* =U^*\Lambda^*$ we deduce that
	\[
	\| A U^* (\Lambda^*)^{-1} - U^* \|_{2\to\infty} = \| (A-A^*) U^* ( \Lambda^*)^{-1}\|_{2\to\infty} \leq \| (A-A^*) U^*\|_{2\to\infty} / \Delta^*.
	\]
	The third inequality in Theorem~\ref{thm-main} follows from the second inequality in Theorem~\ref{thm-main}, \eqref{eq-final-6}, and this inequality.
\end{proof}

\section{Proofs for Section 3} \label{sec:proofsFor3}
\subsection{Proofs for $\Z_2$-synchronization}
\begin{proof}[\bf Proof of Theorem \ref{thm:syn}]
	To invoke Theorem \ref{thm:simple}, we set $A^* = xx^*$ and $A = Y$ . The rank of $A^*$ is $r=1$, and its leading value is $\lambda_1^* = \Delta^* = n$, and its leading eigenvector is $u^*_1 = \frac{1}{\sqrt{n}} x$. We will drop the subscript $1$ in the proof. Note than $| u_i^*| = 1/\sqrt{n}$ for all $i \in [n]$. We choose $\varphi(x) = x$ and $\gamma = \max\{\frac{3}{\sqrt{\log n}}, 1/\sqrt{n}\}$. It is clear that $\gamma = o(1)$ and $\|  u^* \|_{\infty} \le \gamma$. To verify that assumption that $A$ concentrates under the spectral norm, we note that a standard concentration result shows $\| A - A^* \|_2 \le 3\sigma \sqrt{n}$ with probability at least $1 - O(e^{-n/2})$; see \citep[Prop.~3.3]{bandeira2014tightness} for example. To verify the row concentration assumption, note that for each $m \in [n]$, $(A-A^*)_{m \cdot} w$ is a Gaussian variable with a variance no greater than $\sigma^2 \| w \|_2^2$, and $\Delta^* \| w \|_{\infty} \varphi( \frac{\| w\|_2}{\sqrt{n} \| w \|_{\infty}}) = \sqrt{n} \| w \|_2$. Thus,
	\begin{align*}
		&~~~~ \P\left(|(A-A^*)_{m\cdot} w| \le \Delta^* \| w \|_{\infty}  \,
		\varphi  \Big( \frac{\|w\|_2}{\sqrt{n} \| w \|_{\infty} } \Big)
		\right) \ge \P\left( \sigma \| w \|_2 |N(0,1)| \le \sqrt{n}\| w \|_2 \right)\\
		& \ge \P\left( |N(0,1)| \le  \sqrt{(2+\veps)\log n} \right) \ge 1-\frac{2}{\sqrt{2\pi(2+\veps)\log n}}\, n^{-(1+\veps/2)},
	\end{align*}
	where we used a standard Gaussian tail bound $\P(N(0,1) > t) \le \frac{1}{\sqrt{2\pi}t} e^{-t^2/2}$. Therefore, the row concentration assumption holds. Now we use Theorem \ref{thm:simple} to obtain that
	\begin{equation*}
		\min_{s \in \{ \pm 1 \}}\| su - Au^*/\lambda^* \|_{\infty} =\min_{s \in \{ \pm 1 \}} \| su - (u^* + \sigma W u^*/n) \|_{\infty} \lesssim \frac{1}{\sqrt{n \log n}}.
	\end{equation*}
	As argued before, each entry of the vector $\sigma W u^* /n$ is Gaussian, and by the union bound,
	\begin{equation*}
		\P \left( \| \sigma W u^*/n \|_{\infty} \le \sqrt{\frac{2}{(2 + \veps)n}} \right) \ge 1- n  \cdot \frac{2}{\sqrt{4\pi \log n}}\, n^{-1} = 1 - o(1).
	\end{equation*}
	It follows that with probability $1 - o(1)$, $\sqrt{n}\, \| su - u^* \|_{\infty} =  \sqrt{2/(2 + \veps)} + C (\log n)^{-1/2}$, where $C>0$ is a constant. Since $\sqrt{n} \, u^* = x$ can only be $\pm 1$, the desired inequality follows.
\end{proof}

\subsection{Proofs for Stochastic Block Model}
In this subsection, we will use $\Bern(p)$ to denote a Bernoulli random variable $\xi$ with success probability $p$, i.e., $\P(\xi = 1) = p$, $\P(\xi = 0) = 1- p$. The proofs are organized as follows. First, we present the proof of Corollary~\ref{cor-sbm2}, followed by a few associated lemmas; then, we state a tail inequality in Lemma~\ref{lem-tail}, which is useful for analyzing the entries of the linearized eigenvector; next, we prove the main result Theorem~\ref{thm-sbm}; and finally, we prove Theorem~\ref{thm::sbm-lowerbound}, which is followed by additional lemmas.

Before the proofs, we give a derivation of \eqref{eq:miscla} from \cite{zhang2016minimax}. To avoid confusion with our notations, we replace $a,b$ in Theorem~2.2 and~3.2 of \cite{zhang2016minimax} by $a',b'$.\\

\par {\bf Derivation of \eqref{eq:miscla}.}
In Theorem~2.2 and~3.2 of \cite{zhang2016minimax}, we choose $K=2$, and $a' = a\log n$ and $b' = b \log n$. Note that there is a slight difference in the definition of the parameter space from our Definition~\ref{def:mis-sbm}, as we require two communities to have exactly the same number of vertices; nevertheless, as we explained, we could adapt our proofs slighted so that vanilla spectral algorithm still matches the minimax result in their approximately equal-sized regime. 

We use Taylor expansion for $I$ (defined in (1.2) therein):
\begin{align*}
I &= -2\log \left( \frac{ab \log n}{n} + (1 - \frac{a \log n}{2n}) (1 - \frac{b \log n}{2n}) \right) \\
&= -2\log \left( \frac{ab \log n}{n} + 1 - \frac{a \log n}{2n}  - \frac{b \log n}{2n} + o\big(\frac{\log n}{n}\big) \right) \\
&= -2\left( \frac{ab \log n}{n} - \frac{a \log n}{2n}  - \frac{b \log n}{2n} \right) + o\big(\frac{\log n}{n}\big) \\
&= (\sqrt{a} - \sqrt{b})^2 \cdot \frac{\log n}{n} +  o\big(\frac{\log n}{n}\big).
\end{align*}
If $a>b>0$ are constants, it is clear that $nI \to \infty$, so the assumptions of Theorem~2.2 and~3.2 are satisfied. Thus, we can combine these two theorems to obtain \eqref{eq:miscla}. 
\hfill$\square$

\begin{proof}[\bf Proof of Corollary~\ref{cor-sbm2}]
	We will use Theorem~\ref{thm-main} to prove this result, since Theorem~\ref{thm:simple} does not give us the failure probability $O(n^{-3})$. Below we check all the required assumptions. Since we are interested in the second eigenvector, we take $s=r=1$, $\Lambda^*=\lambda_2^*$ and $U^*=u^*_2$. Recall that $A^*$ has rank $2$, with $\lambda^*_1 = \frac{(p+q)n}{2} $ and $\lambda^*_2 = \frac{(p-q)n}{2}$, so $\Delta^*=(\lambda^*_1-\lambda^*_2) \wedge (\lambda^*_2-\lambda^*_3) = (b \wedge \frac{a-b}{2})\log n$. Moreover, $\kappa =  (\frac{a-b}{2})/ (b \wedge \frac{a-b}{2})$, $\|U^*\|_{\max}=1/\sqrt{n}$ and $\| A^*\|_{2\rightarrow\infty}= \frac{\log n}{\sqrt{n}}\sqrt{\frac{a^2+ b^2}{2}} $.
	
	Let $c_1, c_2>0$ be the quantities (only depending on $a$ and $b$) defined in Lemma \ref{lem-sbm-concentration}, which is stated later (see below). We take $\gamma = [(b \wedge \frac{a-b}{2}) \sqrt{\log n}]^{-1}c_1$ and $\varphi(x)= \frac{2a+4}{b \wedge \frac{a-b}{2}}( 1 \vee \log (1/x) )^{-1} $. Assumption \hyperref[cond-1]{A1} and the first part of Assumption \hyperref[cond-3]{A3} holds when $n$ is sufficiently large, and Assumption \hyperref[cond-2]{A2} is trivially satisfied. By Lemma \ref{lem-sbm-concentration}, the second part of Assumption \hyperref[cond-3]{A3} holds with $\delta_0 = c_2 n^{-3}$. Now we show that \hyperref[cond-4]{A4} holds with $\delta_1 = 2n^{-3}$. When applied to  $A \sim \SBM(n , a\frac{\log n}{n}, b \frac{\log n}{n} ,J )$, Lemma \ref{lem-sbm-l2-linf} (stated later), with $p=a\frac{\log n}{n}$, $\alpha = 4/a$, yields
	\begin{equation*}
		\begin{split}
			& \P\Big( |(A-A^*)_{m\cdot} w| \leq
			\frac{(2a+4) \log n}{ 1 \vee \log( \frac{\sqrt{n} \|w\|_{\infty} }{\| w \|_2}  )} \|w\|_{\infty}
			\Big)
			\geq 1- 2n^{-4}.
		\end{split}
	\end{equation*}
	Hence, Assumption \hyperref[cond-4]{A4} is satisfied. Then the desired result follows from Theorem~\ref{thm-main}.
\end{proof}
In the proof above we used the two concentration inequalities, Lemma \ref{lem-sbm-concentration} and Lemma \ref{lem-sbm-l2-linf}. We omit the proof of Lemma \ref{lem-sbm-concentration}, since it directly follows from Theorem 5.2 in \cite{LRi15} or Theorem 5 in \cite{HWX16}, built upon the fundamental result in \cite{FOf05}. Lemma \ref{lem-sbm-l2-linf} is a Bernstein-type inequality and is proved using moment generating function. 
\begin{lem}\label{lem-sbm-concentration}
	Let $A \sim \SBM(n , a\frac{\log n}{n}, b \frac{\log n}{n} ,J )$. There exist $c_1,c_2>0$ determined by $a$ and $b$ such that
	\begin{equation}
	\P (\|A- A^* \|_2\geq c_1\sqrt{\log n}) \leq  c_2 n^{-3}.
	\end{equation}
\end{lem}

\begin{lem}\label{lem-sbm-l2-linf}
	Let $w\in\R^n$ be a fixed vector, $\{ X_i\}_{i=1}^n$ be independent random variables where $ X_i\sim \Bern(p_i)$. Suppose $p\geq \max\limits_{i} p_i$ and $\alpha\geq 0$. Then,
	\begin{equation}
	\P\Big( \Big| \sum_{i=1}^{n}w_i( X_i-\E X_i)\Big| \geq
	\frac{(2 + \alpha)pn}{ 1 \vee \log( \frac{\sqrt{n} \|w\|_{\infty} }{\| w \|_2}  )} \|w\|_{\infty}
	\Big)\leq 2e^{-\alpha np}.
	\end{equation}
\end{lem}
\begin{proof}[Proof of Lemma \ref{lem-sbm-l2-linf}]
	Without loss of generality we assume $\| w \|_{\infty} =1$, since rescaling does not change the event in $\P( \cdot)$.  Let $S_n=\sum_{i=1}^{n}w_i( X_i-\E X_i)$. Markov's inequality yields
	\begin{equation}\label{ineq:bmconctr}
	\P(S_n \geq t)
	=\P(e^{\lambda S_n}\geq e^{\lambda t})
	\leq e^{-\lambda t}\E e^{\lambda S_n}
	=e^{-\lambda t}\prod_{i=1}^{n}\E e^{\lambda w_i( X_i-\E X_i)}, \qquad \forall \, \lambda>0.
	\end{equation}
	We can bound the logarithm of moment generating function by
	\begin{align*}
		\log\left(\E e^{\lambda w_i( X_i-\E X_i)}\right) &=\log[(1-p_i)+p_i e^{\lambda w_i}]-\lambda w_ip_i \leq p_i(e^{\lambda w_i}-1)-\lambda w_i p_i \leq \frac{e^{\lambda \|w\|_{\infty}}}{2}\lambda^2 w_i^2p_i,
	\end{align*}
	where we applied two inequalities: $\log (1+x)\leq x$ for $x>-1$ and $e^x\leq 1+x+\frac{e^{r}}{2}x^2$ for $|x|\leq r$. We take the logarithm of both sides in \eqref{ineq:bmconctr} and use $p\geq \max\limits_i p_i$, $\| w \|_{\infty} = 1$ to obtain
	\begin{equation*}
		\log \P(S_n \geq t)\leq
		-\lambda t+\sum_{i=1}^{n}\log\left(\E e^{\lambda w_i( X_i-\E X_i)}\right)
		\leq -\lambda t + \frac{p\lambda^2}{2}e^{\lambda}\|w\|_2^2.
	\end{equation*}
	Set $\lambda = 1 \vee \log( \sqrt{n} / \| w \|_2  )$ in the above inequality. Since $\| w \|_2 \le \sqrt{n} \| w \|_{\infty} = \sqrt{n}$, we have $\log( \sqrt{n} / \| w \|_2  ) \ge 0$, and thus $\lambda \le 1 + \log( \sqrt{n} / \| w \|_2  )$. This leads to
	\begin{align*}
		\frac{p\lambda^2}{2}e^{\lambda}\|w\|_2^2 \le \frac{p\lambda^2}{2} e \sqrt{n} \| w\|_2 = \frac{epn}{2} \frac{ \|w\|_2}{\sqrt{n}} \left( 1 \vee \log \Big(\frac{\sqrt{n}}{ \|w \|_2} \Big) \right)^2 \le \frac{epn}{2},
	\end{align*}
	where we used the easily verifiable inequality: $1 \vee \log x \le \sqrt{x}$ for $x \ge 1$. Therefore, when $t = [ 1 \vee \log( \frac{\sqrt{n}}{\| w \|_2}  ) ]^{-1} (2 + \alpha)pn $, we deduce
	\begin{equation*}
		\log \P(S_n \geq t) \le -\left[ 1 \vee \log \Big(\frac{\sqrt{n}}{ \|w \|_2} \Big) \right] t + \frac{epn}{2} \le - (2 + \alpha)pn + \frac{epn}{2} \le -\alpha pn.
	\end{equation*}
	By replacing $w$ by $-w$, we get a similar bound for the lower tail. The proof is then finished by the union bound.
\end{proof}

Now we state a lemma that allows us to control the tail of the difference of Binomial variables. It generalizes a similar lemma in \cite{ABH16}.

\begin{lem}\label{lem-tail}
	Suppose $a>b$, $\{W_i\}_{i=1}^{n/2}$ are i.i.d $\Bern(\frac{a\log n}{n})$, and $\{Z_i\}_{i=1}^{n/2}$ are i.i.d. $\Bern(\frac{b\log n}{n})$, independent of $\{W_i\}_{i=1}^{n/2}$. For any $\varepsilon \in \R$, we have the following tail bound:
	\begin{equation}\label{ineqn::tail}
	\P \Big( \sum_{i=1}^{n/2} W_i - \sum_{i=1}^{n/2} Z_i \le \varepsilon \log n \Big)  \le n^{- (\sqrt{a} - \sqrt{b})^2/2 + \varepsilon \log(a/b)/2 }.
	\end{equation}
\end{lem}

\begin{proof}[Proof of Lemma~\ref{lem-tail}]
	Let $\lambda = - \log(a/b)/2 < 0$, and we apply Markov's inequality to the moment generating function,
	\begin{align*}
		\P \Big( \sum_{i=1}^{n/2} W_i - \sum_{i=1}^{n/2} Z_i \le \varepsilon \log n \Big)
		&= \P \Big( e^{ \lambda \sum_{i=1}^{n/2} (W_i - Z_i)} \ge e^{\lambda \varepsilon \log n} \Big)
		\le n^{-\lambda \varepsilon} \, \E  e^{ \lambda \sum_{i=1}^{n/2} (W_i - Z_i)}. \label{ineqn::tailMarkov}
	\end{align*}
	Observe that $e^\lambda = \sqrt{b/a}$, so
	\begin{equation*}
		\log \E e^{\lambda W_i} = \log(e^\lambda \cdot a\log n/n + 1 - a\log n/n) \le (\sqrt{ab} - a) \log n / n,
	\end{equation*}
	where we used the fact that $\log(1 + x) \le x$ for any $x > -1$. Similarly we have $\log \E e^{-\lambda Z_i}  \le (\sqrt{ab} - b) \log n / n$. By independence, this leads to
	\begin{align*}
		& ~~~~ \log \P \Big( \sum_{i=1}^{n/2} W_i - \sum_{i=1}^{n/2} Z_i \le \varepsilon \log n \Big)
		\le \varepsilon \log(a/b)/2 \cdot  \log n + \frac{n}{2} \big( \log \E e^{\lambda W_i} + \log \E e^{-\lambda Z_i} \big) \\
		&\le  \big[ \varepsilon \log(a/b)/2 - (\sqrt{a} - \sqrt{b})^2/2  \big]  \cdot \log n,
	\end{align*}
	which is exactly the desired inequality (\ref{ineqn::tail}).
\end{proof}

\begin{proof}[\bf Proof of Theorem~\ref{thm-sbm}]
	(i) Since $\sqrt{a} - \sqrt{b} > \sqrt{2}$, we can choose some $\veps = \veps(a,b)>0$ such that $ (\sqrt{a} - \sqrt{b})^2/2 - \varepsilon \log(a/b)/2 > 1$. Let $s \in \{ \pm 1\}$ be such that $\| u_2 - sAu_2^*/\lambda_2^*\|_{\infty}$ is minimized. By Corollary~\ref{cor-sbm2}, with probability $1-o(1)$,
	\begin{equation*}
		\sqrt{n} \,  \min_{i\in [n]} s z_i(u_2)_i \ge \sqrt{n} \,  \min_{i\in [n]} s^2 z_i (A u_2^*)_i/\lambda_2^* - C(\log \log n)^{-1},
	\end{equation*}
	where $C$ is defined in Corollary~\ref{cor-sbm2}. Note that $s^2=1$. Also observe that $(A u_2^*)_i /\lambda_2^* = \frac{2}{(a-b)\sqrt{n}\,\log n} \big( \sum_{j \in J} A_{ij} - \sum_{j \in J^c} A_{ij}  \big)$, and thus, for any $i \in [n]$,
	\begin{equation*}
		\sqrt{n} \,  z_i (A u_2^*)_i/\lambda_2^* = \frac{2}{(a-b)\log n} \big( \sum_{i \sim j}A_{ij} - \sum_{i \not\sim j}A_{ij} \big),
	\end{equation*}
	where $i \sim j$ means $i,j \in J$ or $i,j \in J^c$, and $i \not\sim j$ otherwise. Applying Lemma \ref{lem-tail}, we derive
	\begin{equation*}
		\P \Big( \sqrt{n} \,  z_i (A u_2^*)_i/\lambda_2^* \le \frac{2 \veps}{a-b} \Big) \le n^{-(\sqrt{a} - \sqrt{b})^2/2 + \varepsilon \log(a/b)/2} = o(n^{-1}).
	\end{equation*}
	By the union bound, we deduce that with probability $1-o(1)$, for sufficiently large $n$,
	\begin{equation*}
		\sqrt{n} \,  \min_{i\in [n]} s z_i(u_2)_i \ge \frac{2 \veps}{a-b} - o(1) \ge \frac{ \veps}{a-b}.
	\end{equation*}
	Setting $\eta = \veps(a,b) / (a-b)$, we finish the proof of part (i).
	
	(ii) Let us fix an arbitrary $\veps_0 > 0$ and denote $\eta_0 = [ (a-b) \log(a/b)/2]^{-1} \veps_0$, which is positive. Let $C(a,b)$ be the constant in Corollary~\ref{cor-sbm2}, and $B_n$ be the event that \eqref{ineq:sbm-approx} holds. Also let $s_0 \in \{ \pm 1\}$ be such that $\| u_2 - s Au_2^*/\lambda_2^*\|_{\infty}$ is minimized. When $n$ is large enough such that $C(a,b) \le \eta_0 \log \log n$, under $B_n$, we have $ \| u_2 - s_0 A u_2^*/\lambda_2^* \|_{\infty} \le \eta_0 / \sqrt{n}$. Thus, for all $i \in [n]$,
	\begin{align*}
		\{ \hat z_i \neq s_0 z_i \} &\subseteq \{ s_0 z_i  (u_2)_i  \le 0 \} \subseteq B_n^c \cup \{  s_0^2 z_i (A u_2^*)_i /\lambda_2^* \le \eta_0  / \sqrt{n} \} \\
		&= B_n^c \cup \{  z_i (A u_2^*)_i /\lambda_2^* \le \eta_0  / \sqrt{n} \}.
	\end{align*}
	As argued before, $(A u_2^*)_i /\lambda_2^* = \frac{2}{(a-b)\sqrt{n}\,\log n} \big( \sum_{j \in J} A_{ij} - \sum_{j \in J^c} A_{ij}  \big)$, so Lemma \ref{lem-tail} yields
	\begin{align*}
		\log \P \left( z_i (A u_2^*)_i /\lambda_2^* \le \eta_0  / \sqrt{n} \right) &= \log \P \Big( \frac{2}{(a-b)\log n} \big( \sum_{i \sim j}A_{ij} - \sum_{i \not\sim j}A_{ij} \big)
		\le \eta_0   \Big) \\
		& \le - \frac{ (\sqrt{a} - \sqrt{b})^2\log n}{2} + \frac{(a-b)\eta_0 \log n}{2}  \cdot \frac{\log(a/b)}{2}  \\
		& = \Big(- \frac{ (\sqrt{a} - \sqrt{b})^2}{2} + \frac{\veps}{2} \Big) \log n.
	\end{align*}
	Therefore, we can bound the expectation of misclassification rate as follows:
	\begin{align*}
		\E r(\hat z,z) &= \frac{1}{n} \sum_{i=1}^n \P \left( B_n^c \cup \{  z_i (A u_2^*)_i /\lambda_2^* \le \eta_0  / \sqrt{n} \}  \right) \\
		&\le \P(B_n^c) +  \frac{1}{n} \sum_{i=1}^n \P \left( z_i (A u_2^*)_i /\lambda_2^* \le \eta_0  / \sqrt{n}  \right) \\
		&\le \P(B_n^c) + n^{ -(\sqrt{a} - \sqrt{b})^2/2 + \veps/2 }.
	\end{align*}
	By Corollary~\ref{cor-sbm2}, $\P(B_n^c) = O(n^{-3})$, which is smaller than $n^{ -(\sqrt{a} - \sqrt{b})^2/2}$ order-wise if $\sqrt{a} - \sqrt{b} \in (0,\sqrt{2}]$. Thus, for sufficiently large $n$, we have $\E r(\hat z,z)  \le n^{ -(\sqrt{a} - \sqrt{b})^2/2 + \veps }$. This leads to the desired inequality.
\end{proof}

\begin{proof}[\bf Proof of Theorem \ref{thm::sbm-lowerbound}]
	Recall that $(u_2^*)_i = 1/\sqrt{n}$ for $i\in[n/2]$ and $-1/\sqrt{n}$ otherwise. Hence,
	\[
	(A u_2^*)_i = \frac{1}{\sqrt{n} } \sum_{j\leq n/2} A_{ij} - \frac{1}{\sqrt{n} } \sum_{j>n/2} A_{ij}
	,\qquad \forall  i\in[n/2].
	\]
	Then, Lemma \ref{lem-sbm-fluctuation} below states $\lim_{n\to \infty} \P ( \sqrt{ n } \max_{i\in[n/2]} (A u_2^*)_i \geq \frac{a \eta - b}{2} \log n ) = 1$. As a result, from $\lambda_2^* = \frac{a-b}{2} \log n$ we obtain that with probability $1-o(1)$,
	\begin{align*}
		\sqrt{n} \| A u_2^* / \lambda_2^* - u_2^* \|_{\infty} \geq \sqrt{n} \max_{i\in[n/2]} \left\{  ( A u_2^* / \lambda_2^* - u_2^* )_i \right\}
		= \sqrt{n} \max_{i\in[n/2]}  ( A u_2^* )_i/ \lambda_2^* -1 \geq
		\frac{a(\eta -1)}{a-b}.
	\end{align*}
	By Corollary~\ref{cor-sbm2}, the proof is then finished.
\end{proof}

Now we present Lemma \ref{lem-sbm-fluctuation} and its proof. Define $h( t ) = t \log t - t + 1$ for $t>0$. Recall that in the proof of Theorem~\ref{thm-sbm}, we used the union bound
\begin{lem}\label{lem-sbm-fluctuation}
	Let $A \sim SBM (n, a\frac{\log n}{n}, b \frac{\log n}{n} , [n/2] )$, where $a>0$, $b>0$ are constants and $n\to \infty$. For fixed $ \eta > 1$ with $ h(\eta)<2/a$, with probability $1-o(1)$ we have
	\begin{align*}
		\max_{i \in [n/2]} \left\{ \sum_{j\leq n/2} A_{ij} - \sum_{j>n/2} A_{ij} \right\}
		\geq
		\frac{ a \eta  -b }{2} \log n.
	\end{align*}
\end{lem}

\begin{proof}[Proof of Lemma \ref{lem-sbm-fluctuation}]
	First we make the following observation. If $\{ \sum_{j\leq n/2} A_{ij} - \sum_{j>n/2} A_{ij} \}_{i=1}^{n/2}$ were independent, then we could prove the claim by showing that
	\[
	n \cdot \P \left(
	\sum_{j\leq n/2} A_{ij} - \sum_{j>n/2} A_{ij}
	\geq
	\frac{ a\eta  -b }{2} \log n
	\right) \to \infty,\qquad \forall i\in[n/2],
	\]
	with the help of large deviation inequality (Lemma \ref{lem::largedev}). Unfortunately, these random variables $\sum_{j\leq n/2} A_{ij} - \sum_{j>n/2} A_{ij}$ are dependent across $i \in [n]$ , due to symmetry. To tackle this issue, we borrow the idea in \cite{ABH16} for proving information-theoretic lower bound. We will find some appropriate $\varepsilon \in (0,1)$ and work with $\{ \sum_{j>\varepsilon n/2 }^{n/2} A_{ij} - \sum_{j>n/2} A_{ij} \}_{i=1}^{\varepsilon n/2}$ instead of $\{ \sum_{j\leq n/2} A_{ij} - \sum_{j>n/2} A_{ij} \}_{i=1}^{n/2}$, since in the former set, the $\varepsilon n/2$ variables are independent.
	
	Now we begin our proof. Since $h(\cdot)$ is continuous, we can find $\varepsilon \in (0,1)$ such that $h( \frac{\eta}{1-2\varepsilon} )<2/a$. By letting $\zeta = \eta/(1-2\varepsilon)$ we have $(1-\varepsilon) h(\zeta) < h(\zeta) <2/a$. It suffices to show that
	\begin{align}\label{eqn-sbm-lower-bound-lem-3}
		\P \left(
		\max_{i \in [n/2]} \left\{ \sum_{j\leq n/2} A_{ij} - \sum_{j>n/2} A_{ij} \right\}
		\geq
		\frac{ (1-2\varepsilon)a\zeta  -b }{2} \log n
		\right) = 1-o(1).
	\end{align}
	
	Define $S_i = \sum_{j > \varepsilon n/2} ^{n/2} A_{ij} $ and $T_i =  \sum_{j>n/2} A_{ij} $ for $i\in[\varepsilon n/2]$. We claim that there exists $\delta>0$ such that for large $n$,
	\begin{align}
		&\P \left(
		S_i \geq (1-\varepsilon)
		\frac{a\zeta }{2}\log n
		\right) \geq  n^{-1 + \delta} , \label{eqn-sbm-lower-bound-lem-1} \\
		& \P \left(
		T_i - \frac{b}{2} \log n \leq
		\frac{a\zeta \varepsilon }{ 2 } \log n
		\right) \geq \frac{1}{2}. \label{eqn-sbm-lower-bound-lem-2}
	\end{align}
	First, let us use them to prove the desired result. Obviously, we have $\sum_{j\leq n/2} A_{ij} - \sum_{j>n/2} A_{ij} \geq S_i - T_i$ over $i\in [ \varepsilon n/2 ]$. Hence we just need to show that
	\begin{align}\label{eqn-sbm-lower-bound-lem-4}
		\P \left( \max_{i \in [\varepsilon  n/2 ]} \{ S_i - T_i \} < \frac{ (1- 2 \varepsilon ) a \zeta -b }{2} \log n \right) = o(1).
	\end{align}
	For $i\in[\varepsilon n/2]$, $S_i$ and $T_i$ are independent. Then
	\begin{align}
		&\P \left(  S_i - T_i \geq \frac{(1- 2 \varepsilon  )  a \zeta -b }{2} \log n \right)\geq
		\P \left(  S_i \geq(1-\varepsilon ) \frac{a \zeta  }{2} \log n,~~
		T_i - \frac{b}{2} \log n \leq \frac{a \zeta \varepsilon  }{2} \log n
		\right) \notag \\
		& \geq
		\P \left(  S_i \geq (1-\varepsilon ) \frac{a \zeta }{2} \log n \right) \cdot \P \left(
		T_i - \frac{b}{2} \log n \leq \frac{a \zeta \varepsilon  }{2} \log n
		\right) \geq \frac{1}{2 n ^{1-\delta}}, \notag
	\end{align}
	where the last inequality follows from (\ref{eqn-sbm-lower-bound-lem-1}) and (\ref{eqn-sbm-lower-bound-lem-2}). Note that $\{ S_i - T_i \}_{i=1}^{\varepsilon n/2}$ are i.i.d. This leads to
	\begin{align}
		&\P \left( \max_{i \in [\varepsilon  n/2 ]} \{ S_i - T_i \}  < \frac{ (1- 2 \varepsilon  ) a \zeta -b }{2} \log n \right) \leq
		\left( 1 - \frac{1}{2 n ^{1-\delta}} \right) ^{\varepsilon n/2}  \leq  e^{-\varepsilon n^{\delta} /4}. \notag
	\end{align}
	Hence (\ref{eqn-sbm-lower-bound-lem-4}) is proved, so is (\ref{eqn-sbm-lower-bound-lem-3}). Now we come to (\ref{eqn-sbm-lower-bound-lem-1}) and (\ref{eqn-sbm-lower-bound-lem-2}).
	
	Fix $i\in[\varepsilon n/2]$. By applying Lemma \ref{lem::largedev} to $\{ A_{ij} \}_{j> \varepsilon n/2}^{n/2}$, we have $N=(1-\varepsilon) n/2$, $p_N = a \frac{\log n}{n}$ and
	\[
	\liminf _{n\rightarrow\infty} \frac{ \log \P \left( S_i  \geq (1-\varepsilon) \frac{ a \zeta }{2} \log n \right) }{(1-\varepsilon) \frac{a}{2} \log n }  \geq - h(\zeta).
	\]
	Note that $(1-\varepsilon) h(\zeta) a/2 <1$. We can find $\delta>0$ such that $\log \P \left( S_i  \geq (1-\varepsilon) \frac{ a \zeta }{2} \log n \right) \geq (-1+\delta) \log n$ when $n$ is large, proving (\ref{eqn-sbm-lower-bound-lem-1}).
	Finally, the claim (\ref{eqn-sbm-lower-bound-lem-2}) follows from the fact $\mathrm{var} T_i \leq \E T_i = \frac{b}{2} \log n$ and the Markov's inequality.
	
\end{proof}

\begin{lem}[Large deviation result] \label{lem::largedev}
	Suppose $\{ \xi_i \}_{i=1}^N$ are i.i.d.\ Bernoulli random variables taking values in $\{0,1\}$ with $\E \xi_i = p_N$. Assume that $p_N \to 0$ and $Np_N \to \infty$ as $N\to \infty$. We have
	\begin{equation*}
		\liminf_{N \to \infty} \frac{1}{Np_N}\log \P \Big( \sum_{i=1}^N \xi_i \ge \eta Np_N \Big) \ge - h(\eta), \qquad \forall \, \eta > 1.
	\end{equation*}
\end{lem}
\begin{proof}[Proof of Lemma \ref{lem::largedev}]
	First we fix $\eta>1$, $ \tilde \eta > \eta$ and $\veps \in (0, (\tilde \eta - \eta)/2)$. Define $S_N = \sum_{i=1}^N \xi_i  $ and $A_N = \{ S_N \in [ (\tilde \eta - \veps) N p_N, (\tilde \eta + \veps) N p_N] \}$. Observe that for $\lambda_N \ge 0$,
	\begin{align*}
		& \P ( S_N \geq \eta Np_N ) \geq
		\P (A_N) = \P \big(e^{\lambda_N S_N } \in [ e^{\lambda_N(\tilde \eta - \veps) N p_N}, e^{\lambda_N(\tilde \eta + \veps) N p_N}] \big) \\
		&\ge e^{ -\lambda_N(\tilde \eta + \veps) N p_N }  \E \left( e^{\lambda_N S_N } \mathbf{1}_{  A_N  } \right).
	\end{align*}
	We now play a ``change of measure" trick. Define $\varphi ( t ) = \E e^{t \xi_1} = p_N e^t + (1-p_N)$, $\forall t \in \R$. In the probability space $( \Omega, \mathcal{F} , \P )$ where $\{ \xi_i \}_{i=1}^N$ live in, we define another probability measure $\mathbb{Q}$ through its Radon-Nikodym derivative $\mathrm{d} \mathbb{Q} / \mathrm{d} \P = e^{\lambda_N S_N }/ \varphi^N (\lambda_N)$. Then
	\begin{align*}
		\E _{\P} \left( e^{\lambda_N S_N } \mathbf{1}_{ A_N } \right)
		=  \varphi ^N ( \lambda_N ) \E_{\P} \left(  \frac{ \mathrm{d} \mathbb{Q} }{ \mathrm{d} \P } \mathbf{1}_{ A_N  } \right)
		=\varphi ^N ( \lambda_N ) \mathbb{Q} ( A_N ).
	\end{align*}
	Thus, we obtain the following lower bound
	\begin{align}
		\P ( S_N \geq \eta Np_N ) \ge \exp \left( -\lambda_N(\tilde \eta + \veps) N p_N + N \log \varphi(\lambda_N)  \right)  \mathbb{Q}(A_N).\label{ineq:ANbound}
	\end{align}
	Now we set the value of $\lambda_N$ by letting $\E_{ \mathbb{Q}} \xi_i =  \tilde \eta p_N$. Note that under the measure $\mathbb{Q}$, $\xi_i$ is a Bernoulli random variable with $\E_{ \mathbb{Q}} \xi_i = \mathbb{Q} ( \xi_i =1 ) = p_N e^{\lambda_N}/ \varphi(\lambda_N) $. We solve the equation and get $\lambda_N = \log\left( (1-p_N) \tilde \eta \right) - \log(1 - \tilde \eta p_N)$, which is always nonnegative as $\tilde \eta >1$.
	
	On the one hand, from $Np_N\to \infty$ and $\mathrm{var}_{\mathbb{Q}} S_N \leq \E_{\mathbb{Q}} S_N = \tilde \eta Np_N $ we get $\mathbb{Q} ( A_N ) \to 1$. On the other hand, the assumption $p_N = o(1)$ yields
	\begin{align*}
		&\lambda_N = \log \tilde \eta + p_N( \tilde \eta -1 + o(1) ) = \log \tilde \eta + o(1),\\
		&\log \varphi (\lambda_N) =  \lambda_N - \log \tilde \eta = p_N(  \tilde \eta -1 + o(1)).
	\end{align*}
	These estimates and \eqref{ineq:ANbound} lead to
	\begin{align*}
		&\liminf_{N \to \infty}
		\frac{1}{N p_N} \log \P ( S_N \geq \eta N p_N )
		\geq
		\lim_{N \to \infty} \left(	-\lambda_N(\tilde \eta + \veps)  + \frac{1}{p_N} \log \varphi(\lambda_N) + \frac{ \mathbb{Q}(A_N) }{ N p_N } \right) \\
		&=  -(\tilde \eta + \veps) \log \tilde \eta + \tilde \eta - 1
		= -h(\tilde \eta) - \veps \log \tilde \eta.
	\end{align*}
	Taking $\varepsilon \to 0$ and then $\tilde \eta \to \eta$, we finally obtain the desired result.
\end{proof}

\subsection{Proofs for matrix completion from noisy entries}

We first introduce $\mbox{\rm SNMC}(A^*,p,\sigma)$, a symmetric version of noisy matrix completion problem. Theorem \ref{thm-main} becomes applicable in this case, and we derive preliminary results in \ref{appendix-SNMC-checking}. Then we show in \ref{appendix-SNMC-reduction} how to go from here to $\mbox{\rm NMC}(M^*,p,\sigma)$, and bound reconstruction errors in \ref{appendix-SNMC-errors}.

\begin{defn}\label{def-SMC}
	Let $A^*\in\R^{n\times n}$ be symmetric, $p\in[0,1]$ and $\sigma \geq 0$. $\mbox{\rm SNMC} (A^*,p,\sigma )$ is the ensemble of $n\times n$ symmetric random matrices $A=(A_{ij})_{i,j\in[n]}$ with $A_{ij}=( A_{ij}^* + \varepsilon_{ij} )I_{ij}/p$, where $\{ I_{ij}, \varepsilon_{ij} \}_{1\leq i\leq j\leq n}$ are jointly independent, $\P ( I_{ij}=1 )=p=1- \P ( I_{ij} =0 )$ and $\varepsilon_{ij} \sim N(0,\sigma^2)$.
\end{defn}

\subsubsection{Analysis of $\mbox{\rm SNMC} (A^*,p,\sigma )$}\label{appendix-SNMC-checking}
Let us stick to the model $\mbox{\rm SNMC} (A^*,p,\sigma )$ and forget about $\mbox{\rm NMC} (M^*,p,\sigma )$ for a moment. Should there be no further specification, the quantities in this section are defined as in Section \ref{regularity-conditions}, not Section \ref{sec:NMC}.
To facilitate analysis let us define $\bar{A}=( \bar{A}_{ij})_{i,j\in[n]}$ with $\bar{A}_{ij}= A_{ij}^* I_{ij}/p$, where $I_{ij}$ is the same as in Definition \ref{def-SMC}.
\begin{defn}\label{def-SMC-parameters}
	Let $c_1$ be the constant in Lemma \ref{lem-SMC-spectral}. Define $\bar{\kappa}=n \|A^*\|_{\max} / \Delta^*$ and
	\begin{enumerate}
		\item $ \bar \gamma = \frac{ c_1 \bar{\kappa} }{\sqrt{ np}} $, $ \tilde \gamma = \frac{\sigma}{\|A^*\|_{\max} } \bar \gamma$ and $\gamma = \bar \gamma + \tilde \gamma$;
		\item $\bar \varphi(x) =  4 \bar{\kappa} \sqrt{\frac{\log n}{np}} (
		x\vee \sqrt{\frac{\log n}{np}})$, $\tilde \varphi(x) = 4 \bar{\kappa} \sqrt{\frac{ \log n}{np}} \cdot \frac{\sigma}{ \| A^* \|_{\max} }$, $\varphi(x) = \bar \varphi(x)+ \tilde \varphi(x)$, $\forall x \geq 0$;
		\item $ \delta_0 =4n^{-1}$ and $ \delta_1  =5n^{-1}$.
	\end{enumerate}
\end{defn}

\begin{lem}\label{lem-SMC-auxiliary}
	For the quantities in Definition \ref{def-SMC-parameters}, we have $\varphi(\gamma) \leq  4\gamma \sqrt{\log n} ( 1 + \gamma \sqrt{\log n} )$. Furthermore, if $p \gtrsim \frac{\log n}{n}$, then $\varphi (1) \lesssim \bar\kappa \left( 1 + \frac{\sigma}{\| A^* \|_{\max} } \right) \sqrt{\frac{\log n}{np}} \lesssim \gamma \sqrt{\log n}$.
\end{lem}
\begin{proof}[\bf Proof of Lemma \ref{lem-SMC-auxiliary}]
	The facts that $c_1\geq 1$, $\bar{\kappa} \geq 1$, $\bar \gamma \leq \gamma $ and $\bar{\kappa}\sqrt{\frac{\log n}{np}} =\bar\gamma \sqrt{\log n} /c_1 \leq  \bar{\gamma} \sqrt{\log n}$ lead to $\bar \varphi(\gamma)\leq 4\bar\gamma \sqrt{\log n} \left( \gamma \vee \bar \gamma  \sqrt{\log n}  \right)
	\leq 4\gamma^2 \log n$.
	Besides, $\tilde{ \varphi}(\gamma ) = 4\tilde{\gamma}\sqrt{\log n}/c_1 \leq 4\gamma \sqrt{\log n}$. Therefore, $\varphi(\gamma ) = \bar{\varphi}(\gamma) + \tilde{\varphi}(\gamma) \leq 4\gamma \sqrt{\log n} ( 1 + \gamma \sqrt{\log n} )$.
	
	By definition, we have $\gamma = \frac{c_1 \bar{\kappa }}{\sqrt{np}}\left(
	1 + \frac{\sigma}{\|A^*\|_{\max}}
	\right)$ and
	\[
	\varphi(1) =4\bar{\kappa} \sqrt{\frac{\log n}{np}} \left[ \frac{\sigma}{\| A^*\|_{\max}} + \left( 1 \vee \sqrt{\frac{\log n}{np}} \right) \right] \lesssim  \bar{\kappa} \left( 1+ \frac{\sigma}{\| A^*\|_{\max}}  \right) \sqrt{\frac{\log n}{np}}
	\lesssim \gamma \sqrt{\log n}.
	\]
	
\end{proof}

\begin{lem}\label{lem-SMC-sampling-rate}
	Consider $A\sim \mbox{\rm SNMC} (A^*,p,\sigma )$ and the quantities in Definition \ref{def-SMC-parameters}. Assume that $p\geq \frac{6\log n}{n}$ and $130 c_1 \bar{\kappa} \kappa \left( 1+ \frac{\sigma}{\|A^*\|_{\max}} \right)\sqrt{\frac{\log n}{np}} \leq 1$. Then Assumptions \hyperref[cond-1]{A1}-\hyperref[cond-4]{A4} hold. Besides, there exists a constant $C>0$ such that with probability at least $1-14/n$,
	\begin{align}
		&\|U  \|_{2\to\infty}/C \leq  \kappa  \| U^* \|_{2\to\infty} + \gamma \| A^* \|_{2 \to \infty} / \Delta^*,\label{ineq:lem-SMC-1}\\
		&\|U \sgn(H) - U^*\|_{2\to\infty}/C \leq
		\kappa^2 \gamma \sqrt{\log n}  \| U^*\|_{2\to\infty}
		+\gamma \| A^* \|_{2 \to \infty} / \Delta^*.\label{ineq:lem-SMC-2}
	\end{align}
\end{lem}
\begin{proof}[\bf Proof of Lemma \ref{lem-SMC-sampling-rate}]
This proof is somewhat long and technical. We are going check Assumptions \hyperref[cond-1]{A1}, \hyperref[cond-3]{A3} and \hyperref[cond-4]{A4}, respectively, and then apply Theorem \ref{thm-main}. First, \hyperref[cond-1]{A1} is justified by elementary calculation:
		\[
		\gamma \Delta^* \geq \bar \gamma \Delta^* = c_1n\|A^*\|_{\max}/\sqrt{np} \geq c_1\sqrt{n}\|A^*\|_{\max} \geq \|A^*\|_{2\rightarrow\infty}.
		\]
Now we check that $32 \kappa \max\{ \gamma , \varphi(\gamma) \} \le 1$. Note that $\gamma = \bar{\gamma} + \tilde{\gamma} = \frac{c_1\bar{\kappa}}{\sqrt{np}}( 1 + \frac{\sigma}{\|A^*\|_{\max}} )$. When $130 c_1 \bar{\kappa} \kappa \left( 1+ \frac{\sigma}{\|A^*\|_{\max}} \right)\sqrt{\frac{\log n}{np}} \leq 1$, we have $\gamma \kappa \sqrt{\log n} \leq 1/130$. By Lemma \ref{lem-SMC-auxiliary},
		\begin{align}\label{eqn-SMC-max-simple}
			\max\{ \gamma , \varphi(\gamma) \} \leq \max\{ \gamma ,4 \gamma \sqrt{\log n}( 1 + \gamma \sqrt{\log n} ) \} \leq 4 \gamma \sqrt{\log n}(1+1/130).
		\end{align}
		Hence
		\[
		32\kappa \max\{ \gamma , \varphi(\gamma) \} \leq 128(1+1/130) \kappa \gamma \sqrt{\log n} <1.
		\]
	By the following inequalities, the $ \gamma $ and $\delta_0 $ in Definition \ref{def-SMC-parameters} make the second part in \hyperref[cond-3]{A3} hold. They can be easily derived from Lemma 2 and Theorem 1.3 in \cite{KMO101} and thus we omit their proof.
	\begin{lem}\label{lem-SMC-spectral}
		Assume that $p \geq 6\frac{\log n}{n}$. There exists a constant $c_1\geq 1$ such that
		\begin{align*}
			&\P \left( \| \bar A-A^*\|_2 \leq c_1 \|A^*\|_{\max} \sqrt{n/p}  \right) \geq 1-2n^{-1},\\
			&\P\left(\| A- \bar A\|_2 \leq c_1 \sigma \sqrt{n/p} \right) \geq 1-2n^{-1}.
		\end{align*}
	\end{lem}
	It is worth noting that cited results in \cite{KMO101} are slightly different from Lemma \ref{lem-SMC-spectral}. First, they deal with independent sampling in rectangular matrices. But this is easily extended to our framework of symmetric sampling. Second, they work on a trimmed version of $A$, zeroing out rows and columns with too many revealed entries. When $p\geq 6\frac{\log n}{n}$, Chernoff bound (Lemma \ref{lem-smc-chernoff} below) guarantees that with probability at least $1-n^{-1}$, no rows of $A$ has more than $2np$ sampled entries. Thus the trimmed version of $A$ in \cite{KMO101} is equal to $A$ itself.
	\begin{lem}[Chernoff's inequality, see \cite{BLM13}]\label{lem-smc-chernoff}
		Let $\{ X_i \}_{i=1}^n$ be a sequence of independent random variables in $[0,1]$, $S_n = \sum_{i=1}^{n} X_i$ and $\mu = \mathbb{E} S_n$. Then
		\[
		\mathbb{P} ( S_n \geq (1+\varepsilon)\mu  ) \leq e^{-\varepsilon^2 \mu / (2+\varepsilon) },~\forall \varepsilon>0
		\]
	\end{lem}
	With the $\varphi(x)$ and $\delta_1$ in Definition \ref{def-SMC-parameters}, the following lemma guarantees Assumption \hyperref[cond-4]{A4} to hold.
	\begin{lem}\label{lem-SMC-row}
		For any fixed $W\in\R^{n\times r}$ and $m\in[n]$, we have
		\begin{align*}
			&\P \left ( \|(\bar A-A^*)_{m\cdot}W \|_2 \leq  \Delta^* \| W \|_{2\to\infty} \bar \varphi \left ( \frac{\|W\|_F}{\sqrt{n} \|W\|_{2\to\infty}} \right)
			\right) \geq 1-2 n^{-2},\\
			&\P\left( \| (A- \bar A)_{m\cdot} W \|_2 \leq  \Delta^* \| W \|_{2\to\infty} \tilde \varphi  \left ( \frac{\|W\|_F}{\sqrt{n} \|W\|_{2\to\infty}} \right)
			\right) \geq 1-3 n^{-2}.
		\end{align*}
	\end{lem}
	
	To prove the first inequality in Lemma \ref{lem-SMC-row}, we apply Bernstein's inequality in the following form, which is a special case of the Theorem 6.1.1 in \cite{Tro15}.
	\begin{lem}\label{lem-SMC-bernstein}
		Let $\{X_i\}_{i=1}^n$ be independent, zero-mean random vectors in $\R^{r}$, with $\|X_i\|_2\leq M$, $\forall i$ almost surely. Then, we have
		\begin{equation*}
			\P\left ( \left\|\sum_{i=1}^{n}X_i\right\|_2 \geq t \right)\leq (r+1) \exp\left(\frac{-t^2/2}{\sum_{i=1}^{n}\E \| X_i\|_2^2+Mt/3 }\right),~\forall t\geq 0.
		\end{equation*}
	\end{lem}
	For any fixed $m\in[n]$ and $W\in\R^{n\times r}$, we take $X_i = (\bar A-A^*)_{mi}W_{i\cdot }$. Note that $\|X_i\|_2\leq \|A^*\|_{\max} \| W \|_{2\to\infty}/p$ and $\E \| X_i \|_2^2 = \|W_{i\cdot}\|_2^2 \E \bar A^2_{mi} \leq \|A^*\|_{\max}^2 \|W_{i\cdot}\|_2^2/p$. Lemma \ref{lem-SMC-bernstein} yields
	\begin{align*}
		&\P( \|(\bar A-A^*)_{m\cdot}W \|_2 \geq t\|A^*\|_{\max} )\leq (r+1) \exp\left(-\frac{pt^2/2}{ \|W\|_F^2 + \|W\|_{2\to\infty} t/3}\right) \notag \\
		&\leq (r+1)\exp\left( -\frac{pt^2/4}{\|W\|_F^2 \vee  (\|W\|_{2\to\infty} t/3) }\right)
		\leq 2n \exp \left[- \left( \frac{pt^2}{4 \|W\|_F^2} \wedge \frac{3pt}{4 \|W\|_{2\to\infty} }
		\right) \right],~\forall t\geq 0.
	\end{align*}
	Let $\bar\varphi(x) = 4 \bar{\kappa} \sqrt{\frac{\log n}{np}} (
	x\vee \sqrt{\frac{\log n}{np}}
	)$ and $t = \frac{\Delta^*}{\|A^*\|_{\max}} \| W \|_{2\to\infty} \bar\varphi( \frac{\|W\|_F}{\sqrt{n} \|W\|_{2\to\infty}} )$. The relationships
	\begin{align*}
		& t  \geq \frac{\Delta^*}{ \| A^*\|_{\max} } \| W \|_{2\to\infty} \cdot \frac{4n \|A^*\|_{\max} }{ \Delta^* } \sqrt{\frac{\log n}{np}} \cdot \frac{\|W\|_F}{\sqrt{n} \|W\|_{2\to\infty}} = 4\|W\|_F \sqrt{\frac{\log n}{p}} ,\notag\\
		& t \geq \frac{\Delta^*}{ \| A^*\|_{\max} } \| W \|_{2\to\infty} \cdot \frac{4n \|A^*\|_{\max} }{ \Delta^* } \sqrt{\frac{\log n}{np}} \cdot \sqrt{\frac{\log n}{np}} = 4\|W\|_{2\to\infty} \frac{\log n}{p},\notag
	\end{align*}
	lead to $\frac{pt^2}{4 \|W\|_F^2} \geq 4\log n $ and $\frac{3pt}{4 \|W\|_{2\to\infty} } \geq 3\log n$, respectively and prove the first inequality in Lemma \ref{lem-SMC-row}.
	
	To prove the second inequality in Lemma \ref{lem-SMC-row}, we apply the following concentration bound, which is a special case of the Theorem 4.1.1 in \cite{Tro15}.
	\begin{lem}\label{lem-SMC-gaussian}
		Let $\{X_i\}_{i=1}^n$ be fixed vectors in $\R^{r}$, and $\{ \varepsilon_i \}_{i=1}^n$ be i.i.d. $N(0,1)$ random variables. We have
		\begin{equation*}
			\P\left ( \left\|\sum_{i=1}^{n}\varepsilon_i X_i\right\|_2 \geq t \right)\leq (r+1) \exp\left(\frac{-t^2/2}{\sum_{i=1}^{n}\| X_i\|_2^2}\right),~\forall t\geq 0.
		\end{equation*}
	\end{lem}
	Define $S_m = \{ i\in[n] : I_{mi}=1 \}$. Then $(A-\bar{A})_{m\cdot}W = \frac{\sigma}{p} \sum_{i\in S_m} \frac{\varepsilon_{mi}}{\sigma} W_{i\cdot}$. Given $S_m$, $\{ \varepsilon_{mi}/\sigma \}_{i\in S_m}$ are i.i.d. $N(0,1)$, and Lemma \ref{lem-SMC-gaussian} yields
	\[
	\P \left(
	\| (A-\bar A)_{m\cdot} W \|_2 \geq t |S_m
	\right)
	\leq
	(r+1)
	\exp \left(
	-\frac{p^2t^2}{ 2 \sigma^2 \sum_{i \in S_m} \| W_{i\cdot} \|_2^2 }
	\right).
	\]
	On the one hand, by taking $t=\frac{\sigma }{p}\sqrt{6 \log n} \left( \sum_{i \in S_m} \| W_{i\cdot} \|_2^2  \right)^{1/2}$ above we get
	\begin{align*}
		&\P \left( \|(A-\bar{A})_{m\cdot}W \|_2 \geq
		\frac{\sigma }{p} \sqrt{6 |S_m| \log n} \| W \|_{2\to\infty}
		\right)\\
		&\leq
		\P \left( \|(A-\bar{A})_{m\cdot}W \|_2 \geq
		\frac{\sigma }{p} \left( 6 \log n \sum_{i \in S_m} \| W_{i\cdot} \|_2^2  \right)^{1/2}
		\right)
		\leq 2n^{-2}.
	\end{align*}
	On the other hand, by taking $\varepsilon=1$ in Lemma \ref{lem-smc-chernoff} and using the assumption $p \geq 6\frac{\log n}{n}$, we obtain that $\P \left( |S_m|\geq 2np \right) \leq n^{-2}$. Hence the union bounds yield
	\begin{align*}
		&\P\left( \| (A- \bar A)_{m\cdot} W \|_2 \leq  \Delta^* \| W \|_{ 2\to\infty} \tilde \varphi \Big ( \frac{\|W\|_F}{\sqrt{n} \|W\|_{2\to\infty}} \Big)
		\right)\\
		&\geq \P \left( \| (A-\bar{A})_{m\cdot}W \|_2 \leq \sigma \|W\|_{2\to\infty} \sqrt{\frac{ 12 n\log n}{p}} \right)
		\geq 1-3n^{-2}.
	\end{align*}
	In other words, Lemma \ref{lem-SMC-row} is proved, so is Assumption \hyperref[cond-4]{A4}.
Now that all the assumptions have been checked, we apply Theorem \ref{thm-main} and derive that with probability at least $1 - \delta_0 - 2 \delta_1 = 1 - 14/n $,
	\begin{align*}
		& \|U  \|_{2\to\infty}/C' \leq \left( \kappa + \varphi(1) \right) \| U^* \|_{2\to\infty} + \gamma \| A^* \|_{2 \to \infty} / \Delta^*,\\
		&\|U \sgn(H) - U^*\|_{2\to\infty}/C' \leq  [\kappa  ( \kappa + \varphi(1)	)( \gamma + \varphi(\gamma) ) + \varphi(1) ]  \| U^*\|_{2\to\infty}
		+\gamma  \|A^*\|_{2\rightarrow\infty}/\Delta^*,
	\end{align*}
	where $C'>0$ is a constant. Since $p\geq 6\frac{\log n}{n}$ and $\bar{\kappa}\kappa \left( 1+ \frac{\sigma}{\|A^*\|_{\max}} \right)\sqrt{\frac{\log n}{np}} \lesssim 1$, Lemma \ref{lem-SMC-auxiliary} yields a crude bound $\varphi(1) \lesssim 1/\kappa \leq \kappa$. Hence
	\begin{align}
		& \|U  \|_{2\to\infty} \lesssim \kappa \| U^* \|_{2\to\infty} + \gamma \| A^* \|_{2 \to \infty} / \Delta^*,\notag\\
		&\|U \sgn(H) - U^*\|_{2\to\infty} \lesssim  [\kappa^2 ( \gamma + \varphi(\gamma) ) + \varphi(1) ]  \| U^*\|_{2\to\infty}
		+\gamma  \|A^*\|_{2\rightarrow\infty}/\Delta^*.\label{ineq:lem-SMC-2-1}
	\end{align}
	(\ref{ineq:lem-SMC-1}) has been proved. On the one hand, (\ref{eqn-SMC-max-simple}) forces $\gamma + \varphi(\gamma )\lesssim \gamma \sqrt{\log n}$. On the other hand, Lemma \ref{lem-SMC-auxiliary} implies that $\varphi(1) \lesssim \gamma \sqrt{\log n}$. These facts and (\ref{ineq:lem-SMC-2-1}) lead to (\ref{ineq:lem-SMC-2}), thus complete the proof of Lemma \ref{lem-SMC-sampling-rate}.
\end{proof}

\subsubsection{From $\mbox{\rm SNMC}(A^*, p, \sigma)$ to $\mbox{\rm NMC}(M^*,p,\sigma)$}\label{appendix-SNMC-reduction}
From now on, the quantities $M, M^*, U, U^*, V, V^*, \Sigma, \Sigma^*$ are defined as in Section \ref{sec:NMC}. We are going to present the ``symmetric dilation" trick \citep{Pau02}. Recall the SVD $M^* = U^* \Sigma^* (V^*)^T$. A key observation is that $A^*=
\begin{pmatrix}
\mathbf{0}_{n_1 \times n_1} & M^*\\
(M^*)^T & \mathbf{0}_{n_2 \times n_2}
\end{pmatrix}$ is a symmetric matrix, whose eigen-decomposition is given by
\begin{equation}\label{eq-SMC-dilation}
A^*
=\frac{1}{\sqrt{2}}
\begin{pmatrix}
U^* & U^*\\
V^* & -V^*
\end{pmatrix}
\cdot
\begin{pmatrix}
\Sigma^* & \\
& -\Sigma^*
\end{pmatrix}
\cdot
\frac{1}{\sqrt{2}}
\begin{pmatrix}
U^* & U^*\\
V^* & -V^*
\end{pmatrix}^T.
\end{equation}
The $r $ largest eigenvalues of $A^*$ are exactly the $r $ singular values of $M^*$, and their corresponding eigenvectors make up $\bar{U}^*=\frac{1}{\sqrt{2}}\begin{pmatrix}
U^*\\
V^*
\end{pmatrix}\in\mathcal{O}_{(n_1 + n_2)\times r }$. The $(r +1)$-th largest eigenvalue of $A^*$ is 0. Define $\bar{\Lambda}^*=\Sigma^*$. We have $\bar{U}^* \bar{\Lambda}^* (\bar{U}^*)^T =\frac{1}{2}\begin{pmatrix}
U^* \Sigma^* (U^*)^T & M^*\\
(M^*)^T & V^* \Sigma^* (V^*)^T
\end{pmatrix}$.

On the other hand, define $A=\begin{pmatrix}
\mathbf{0}_{n_1 \times n_1} & M\\
M^T & \mathbf{0}_{n_2 \times n_2}
\end{pmatrix}$, $\bar{U}=\frac{1}{\sqrt{2}}
\begin{pmatrix}
U\\
V
\end{pmatrix}$ and $\bar{\Lambda}=\Sigma$. Similar to \eqref{eq-SMC-dilation}, we know that $\bar{\Lambda}$ has in its diagonal the largest $r $ eigenvalues of $A$, whose corresponding eigenvectors are the columns in $\bar{U}$. Therefore, for any $W\in\R^{r  \times r }$,
\[
\|UW-U^*\|_{2\to\infty} \vee \|VW-V^*\|_{2\to\infty} = \sqrt{2} \|\bar{U} W - \bar{U}^*\|_{2\to\infty}.
\]
Thanks to these observations, the noisy matrix completion problem $\mbox{\rm NMC}(M^*,p,\sigma)$ in Definition \ref{def-MC} is reduced to the symmetric sampling problem $\mbox{\rm SNMC}(A^*, p, \sigma)$ in Definition \ref{def-SMC}, with $A^* =\begin{pmatrix}
\mathbf{0}_{n_1 \times n_1} & M^*\\
(M^*)^T & \mathbf{0}_{n_2 \times n_2}
\end{pmatrix}$. We are going to rephrase Lemma \ref{lem-SMC-sampling-rate} and get the final results.

\subsubsection{Reconstruction errors}\label{appendix-SNMC-errors}
Theorem \ref{thm-MC-main} is a corollary of the following lemma.
\begin{lem}\label{lem-SMC-sampling-rate-1}
	Let $M\sim \mbox{\rm NMC}(M^*,p,\sigma)$, and $c_1$ be the constant in Lemma \ref{lem-SMC-spectral}. Define $n=n_1+n_2$, $\kappa = \sigma_1^*/\sigma_r^*$, $H = \frac{1}{2}( U^T U^* + V^T V^* ) $, and $\eta=\left(
	\|U^*\|_{2\to\infty} \vee \|V^*\|_{2\to\infty}
	\right)$. Suppose $p\geq 6 \frac{\log n}{n}$ and $\frac{130c_1 n\sigma_1^*}{(\sigma_r^*)^2} \sqrt{\frac{\log n}{np}} ( \|M^*\|_{\max} + \sigma ) \leq 1$. Then Assumptions \hyperref[cond-1]{A1}- \hyperref[cond-4]{A4} hold.
	There exists a constant $C>0$ such that with probability at least $1-14/n$, we have
	\begin{align*}
		&\left(
		\| U \|_{2\to\infty} \vee \| V \|_{2\to\infty}
		\right)
		\leq C \kappa \eta,\\
		&\left(
		\|U \sgn(H) - U^*\|_{2\to\infty} \vee \| V \sgn(H) - V^*\|_{2\to\infty}
		\right) \leq C \eta \kappa^2 \frac{ n ( \| M^* \|_{\max} + \sigma )}{\sigma_r^*}
		\sqrt{\frac{\log n}{np}},\\
		&\| U \Sigma V^T - M^* \|_{\max} \leq C( \| M^* \|_{\max} + \sigma ) \kappa^4
		(\sqrt{n} \eta )^2
		\sqrt{\frac{\log n}{np}}
		.
	\end{align*}
\end{lem}

\begin{proof}[\bf Proof of Lemma \ref{lem-SMC-sampling-rate-1}]
	Consider the quantities $A, A^*, \bar{U}, \bar{U}^*$ in \ref{appendix-SNMC-reduction}. We have $H=\bar U ^T \bar U^*$. Let $\gamma = \frac{c_1}{\sigma_r^*}\sqrt{\frac{n}{p}}( \|M^*\|_{\max} + \sigma ) $. Lemma \ref{lem-SMC-sampling-rate} applied to $\mbox{\rm SNMC}(A^*,p,\sigma)$ yields that with probability at least $1-14/n$,
	\begin{align*}
		&\|\bar U  \|_{2\to\infty} \lesssim  \kappa
		\|\bar U^*\|_{2\to\infty}  + \gamma \| A^* \|_{2 \to \infty} / \sigma_r^*,\\
		&\|\bar U \sgn(H) - \bar U^*\|_{2\to\infty} \lesssim
		\kappa^2 \gamma \sqrt{\log n}
		\|\bar U^*\|_{2\to\infty}
		+\gamma \| A^* \|_{2 \to \infty} / \sigma_r^*.
	\end{align*}
	The decomposition in (\ref{eq-SMC-dilation}) forces that
	\[
	\|A^*\|_{2\to\infty} \leq \left \| \frac{1}{\sqrt{2}}
	\begin{pmatrix}
	U^* & U^* \\
	V^* & -V^*
	\end{pmatrix}
	\right\|_{2\to\infty} \|\Sigma^*\|_2 =
	\left(
	\|U^*\|_{2\to\infty} \vee \|V^*\|_{2\to\infty}
	\right)
	\sigma_1^*.
	\]
	Hence $\|A^*\|_{2\to\infty}/\sigma_r^* \leq \kappa \eta$ and
	\begin{align}
		&\left(
		\| U \|_{2\to\infty} \vee \| V \|_{2\to\infty}
		\right) \lesssim \|\bar U  \|_{2\to\infty}
		\lesssim \kappa \eta,
		\label{ineq-final-1}
		\\
		&\left(
		\|U \sgn(H) - U^*\|_{2\to\infty} \vee \| V \sgn(H) - V^*\|_{2\to\infty}
		\right) \lesssim
		\|\bar U \sgn(H) - \bar U^*\|_{2\to\infty} \notag \\
		&\lesssim \kappa^2 \gamma \eta \sqrt{\log n}  	
		\lesssim \kappa^2 \frac{ n ( \| M^* \|_{\max} + \sigma )}{\sigma_r^*}
		\sqrt{\frac{\log n}{np}}\eta
		\label{ineq-final-2}.
	\end{align}
	Finally we come to the entry-wise reconstruction error $\|U \Sigma V^T - M^* \|_{\max}$.  Define $\tilde{\Sigma} = \sgn(H)^T \Sigma \sgn(H)$, $\tilde{U}=U\sgn(H)$ and $\tilde{V}=V\sgn(H)$. By the fact $ \tilde{U} \tilde{\Sigma} \tilde{V}^T = U\Sigma V^T $ and H\"{o}lder's inequality, we have
	\begin{align*}
		& | ( U \Sigma V^T - M^* )_{ij} | =|\tilde{U}_{i\cdot} \tilde\Sigma \tilde V_{j\cdot}^T - U_{i\cdot}^* \Sigma^* (V_{j\cdot}^*)^T |= | \langle \tilde{\Sigma}, \tilde{U}_{i\cdot}^T \tilde{V}_{j\cdot} \rangle - \langle \Sigma^* , (U^*_{i\cdot})^T V^*_{j\cdot} \rangle | \notag \\
		& \leq | \langle \tilde{\Sigma} - \Sigma^*, \tilde{U}_{i\cdot}^T \tilde{V}_{j\cdot} \rangle | + | \langle \Sigma^* , \tilde{U}_{i\cdot}^T \tilde{V}_{j\cdot} - (U^*_{i\cdot})^T V^*_{j\cdot} \rangle| \notag\\
		& \leq \|\tilde{\Sigma} - \Sigma^* \|_2 \| \tilde U_{i\cdot}^T \tilde V_{j\cdot}\|_* + \| \Sigma^* \|_2 \| \tilde{U}_{i\cdot}^T \tilde{V}_{j\cdot} - (U^*_{i\cdot})^T V^*_{j\cdot} \|_*.
	\end{align*}
	Therefore
	\begin{align}\label{eqn-SMC-reconstruction-1}
		\|U \Sigma V^T - M^* \|_{\max}
		\leq \|\tilde{\Sigma} - \Sigma^* \|_2 \max_{i,j\in[n]}\| \tilde U_{i\cdot}^T \tilde V_{j\cdot}\|_* + \| \Sigma^* \|_2 \max_{i,j\in[n]}\| \tilde{U}_{i\cdot}^T \tilde{V}_{j\cdot} - (U^*_{i\cdot})^T V^*_{j\cdot} \|_*.
	\end{align}
	We begin to work on terms above. It is easy to see that $\| \Sigma^* \|_2 =\sigma_1^*$ and by (\ref{ineq-final-1}),
	\begin{align}\label{eqn-SMC-reconstruction-2}
		\| \tilde U_{i\cdot}^T \tilde V_{j\cdot}\|_* = \|\tilde U_{i\cdot}\|_2 \| \tilde V_{j\cdot}\|_2
		\leq  \| U\|_{2\to\infty} \| V \|_{2\to\infty} \lesssim   \kappa^2 \eta^2,~\forall i,j.
	\end{align}
	Besides,
	\begin{align}
		&\| \tilde{\Sigma} - \Sigma \|_2 = \| \sgn(H)^T [ \Sigma \sgn(H) - \sgn(H) \Sigma ] \|_2
		\leq \|  \Sigma \sgn(H) - \sgn(H) \Sigma\|_2 \notag \\
		&=  \| (\Sigma H - H \Sigma ) + \Sigma [\sgn(H) - H]+ [H-\sgn(H)] \Sigma\|_2 \notag \\
		& \leq \| \Sigma H - H \Sigma \|_2 + 2\|\Sigma\|_2 \|\sgn(H) - H \|_2.\notag
	\end{align}
	Lemmas \ref{lem-SMC-spectral} and \ref{lem:H} yield $\| \Sigma H - H \Sigma \|_2 \lesssim \|A-A^*\|_2\lesssim \gamma \sigma_r^*$ and $ \| \sgn(H) - H \|_2 \lesssim \gamma^2$. Then Assumption \hyperref[cond-1]{A1} forces $\kappa \gamma \lesssim 1$ and $\| \tilde{\Sigma} - \Sigma \|_2  \lesssim \gamma \sigma_r^* + \gamma^2 \sigma_1^*
	=\gamma \sigma_r^* ( 1 + \gamma \kappa ) \lesssim \gamma \sigma_r^*$.
	Since $\| \Sigma - \Sigma^*\|_2 \leq \| A-A^*\|_2 \lesssim \gamma \sigma_r^*$, we have
	\begin{align}\label{eqn-SMC-reconstruction-3}
		\|\tilde{\Sigma} - \Sigma^* \|_2 \leq  \| \tilde{\Sigma} - \Sigma \|_2 + \| \Sigma - \Sigma^* \|_2
		\lesssim \gamma \sigma_r^*.
	\end{align}
	We begin to work on $\max_{i,j}\| \tilde{U}_{i\cdot}^T \tilde{V}_{j\cdot} - (U^*_{i\cdot})^T V^*_{j\cdot} \|_*$. By the triangle inequality and (\ref{ineq-final-2}),
	\begin{align}\label{eqn-SMC-reconstruction-4}
		&\| \tilde{U}_{i\cdot}^T \tilde{V}_{j\cdot} - (U^*_{i\cdot})^T V^*_{j\cdot} \|_*
		\leq \| \tilde{U}_{i\cdot}^T (\tilde{V}_{j\cdot} - V^*_{j\cdot}) \|_* + \| (\tilde U_{i\cdot} - U_{i\cdot}^*)^T V^*_{j\cdot} \|_*\notag \\
		& \leq \| \tilde{U}_{i\cdot} \|_2 \|\tilde{V}_{j\cdot} - V^*_{j\cdot} \|_2 + \| \tilde U_{i\cdot} - U_{i\cdot}^* \|_2  \|V^*_{j\cdot} \|_2\notag \\
		&\lesssim
		\left(
		\| U \|_{2\to\infty} + \| V^* \|_{2\to\infty}
		\right)
		\left(
		\|U \sgn(H) - U^*\|_{2\to\infty} \vee \| V \sgn(H) - V^*\|_{2\to\infty}
		\right) \notag \\
		&\lesssim
		\kappa^3 \gamma 	\eta^2 \sqrt{\log n}
		.
	\end{align}
	By plugging \eqref{eqn-SMC-reconstruction-2}, \eqref{eqn-SMC-reconstruction-3} and \eqref{eqn-SMC-reconstruction-4} into \eqref{eqn-SMC-reconstruction-1}, we obtain that
	\begin{align*}
		&\|U \Sigma U^T - U^* \Sigma^* (U^*)^T \|_{\max}
		\lesssim  \gamma \sigma_r^*\cdot \kappa^2 \eta^2
		+ \sigma_1^*\cdot \kappa^3\gamma \eta ^2 \sqrt{\log n}
		=
		\kappa^2 \sigma_r^* \gamma	\eta ^2
		(  1
		+ \kappa^2  \sqrt{\log n}
		) \\
		&\lesssim
		\kappa^4 \sigma_r^* \gamma\eta ^2 \sqrt{\log n}
		\lesssim \kappa^4
		(\sqrt{n} \eta )^2
		\sqrt{\frac{\log n}{np}}
		( \| M^* \|_{\max} + \sigma )
		.
	\end{align*}
	Above we used $\gamma = \frac{c_1}{\sigma_r^*} \sqrt{\frac{n}{p}} ( \| M^* \|_{\max} + \sigma )$.
	
\end{proof}

\subsection{Further results of SBM: more than $2$ blocks}

We conclude this supplementary paper with further results of SBM. For illustrative purposes, we will focus on $3$ blocks.

\begin{defn}\label{def-SBM-3}
	Let $n$ be a multiple of 3, $a>b>0$ are constants, and $z \in \{ 1,2,3 \}^n$ satisfy $| \{ i:z_i=k \} |=n/3$ for $k=1,2,3$. $\mbox{\em SBM3}(n,a,b,z)$ is the ensemble of $n\times n$ symmetric random matrices $A=(A_{ij})_{i,j\in[n]}$ where $\{A_{ij}\}_{1\leq i\leq j\leq n}$ are independent Bernoulli random variables, and
	\begin{equation}
	\P(A_{ij}=1)=\begin{cases}
	&a \frac{\log n}{n},\text{ if } z_i=z_j \\
	&b \frac{\log n }{n},\text{otherwise}
	\end{cases}
	.
	\end{equation}
\end{defn}

Now $A^* = \E A$ has rank $3$. Its nonzero eigenvalues are $\lambda^*_1=(a+2b) \log n /3 $ and $\lambda^*_2=\lambda^*_3=(a-b) \log n/3$, whose associated normalized eigenvectors are $u^*_1=\frac{1}{\sqrt{n}}\mathbf{1}_n$, $u^*_2=\frac{1}{\sqrt{2n}} (2 \mathbf{1}_{J_1} - \mathbf{1}_{J_2} - \mathbf{1}_{J_3} )$ and $u^*_3= \sqrt{ \frac{ 3 }{ 2n } } ( \mathbf{1}_{J_2}  - \mathbf{1}_{J_3} )$, respectively. Here we define $J_k=\{ i:z_i=k \}$. Note that we can choose $\{ u_2^* , u_3^* \}$ to be any orthonormal basis in their linear span, due to the multiplicity of eigenvalues.

Let $U^*=(u_2^*,u_3^*)$ and $\Lambda^* = \diag( \lambda_2^*,\lambda_3^* )= \lambda_2^*  I_2$. $U$ denotes the empirical version of $U^*$ obtained from eigen-decomposition of $A$. The rows of $U$ and $U^*$ provide (empirical and population) spectral embeddings of the nodes into $\R^2$. It is easily seen that $U^*$ has only three distinct rows:
\begin{align*}
 U^*_{i\cdot} = \begin{cases}
( \sqrt{2/n} ,0 ) & \mbox{ if } i \in J_1\\
( -1/\sqrt{2n} , \sqrt{3/(2n)} ) & \mbox{ if } i \in J_2\\
( -1/\sqrt{2n} , -\sqrt{3/(2n)} ) & \mbox{ if } i \in J_3\\
\end{cases},
\end{align*}
revealing true memberships of nodes. Matrix perturbation theories \citep{DavKah70} assert the existence of some $2\times 2$ orthonormal matrix $O$ such that $UO$ is close to $U^*$. Hence we can estimate the block memberships by applying clustering algorithms for Euclidean data, such as k-means, to the rows of $U$. This forms the basis of spectral clustering \citep{RohChaYu11,LRi15} for community detection. Intuitively, such algorithms are able to return high-quality estimates if $\{ U_{i\cdot} \}_{i\in J_1}$, $\{ U_{i\cdot} \}_{i\in J_2}$ and $\{ U_{i\cdot} \}_{i\in J_3}$ are well separated. The Lemma \ref{lem-sbm-3} below characterizes such separation phenomenon. 

\begin{lem}\label{lem-sbm-3}
Consider the model in Definition \ref{def-SBM-3} with $\sqrt{a}-\sqrt{b}>\sqrt{3}$ and let $n$ go to infinity. Define $v_1=(\sqrt{2/n},0)$, $v_2=(-1/\sqrt{2n},\sqrt{3/(2n)})$ and $v_3=(-1/\sqrt{2n}, -\sqrt{3/(2n)} )$. There exists a constant $c>0$ and a sequence of $2\times 2$ orthonormal matrices $Q=Q_n$ such that with high probability,
\begin{align*}
\min_{j \neq z_i } \| U_{i\cdot} - v_j Q \|_2 \geq  \| U_{i\cdot} - v_{z_i} Q \|_{2}  + c /\sqrt{n},\qquad \forall i\in[n].
\end{align*}
\end{lem}

According to \cite{ASa15}, exact recovery is not possible when $\sqrt{a}-\sqrt{b}<\sqrt{3}$. Hence the condition $\sqrt{a}-\sqrt{b}>\sqrt{3}$ in Lemma \ref{lem-sbm-3} is necessary for exact recovery (as we have ruled out the case $a=b$ in Definition \ref{def-SBM-3}). In this scenario, Lemma \ref{lem-sbm-3} states that the distance between the embedding $U_{i\cdot}$ of node $i$ and a ``center" $v_j Q$ is minimized when $j=z_i$. If we have three estimated centers $\{ \hat{v}_j \}_{j=1}^3$ with reasonable precision, i.e. $\max_{1\leq j\leq 3} \| \hat{v}_j - v_j Q \|_2 < \delta / \sqrt{n}$ for sufficiently small $\delta$, then the estimator $\hat{z}$ given by
\begin{align*}
\hat{z}_i = \mathrm{argmin}_{1\leq j \leq 3} \| U_{i\cdot} - \hat{v}_j \|_2
\end{align*}
exactly recovers the block membership vector $z$.

Unfortunately, we have not figured out a simple estimator of the centers $\{ v_j Q \}_{j=1}^3$. The task could be possible if we allow for multi-stage procedures such as estimating the centers by applying k-means to a carefully selected subset of $\{ U_{i\cdot} \}_{i=1}^n$. But the overall algorithm becomes complicated and has no advantage over existing spectral algorithms with trimming and/or cleaning, which are known to achieve the information threshold for exact recovery. We leave this as an open question for future studies.

\begin{proof}[Proof of Lemma \ref{lem-sbm-3}]
We claim the existence of a constant $C$ and a vanishing sequence $\rho=\rho_n = o(1)$ such that with high probability, 
\begin{align}
\| U \sgn( U^T U^* ) - A U^* ( \Lambda^*)^{-1} \|_{2\to\infty} \leq \rho /\sqrt{n}. \label{ineq-lem-sbm-3-2}
\end{align}
We omit the proof of above, as it is a direct application of Corollary \ref{cor-main}, and is very similar to that of Corollary~\ref{cor-sbm2}.

Define $Q=\sgn( U^T U^* )^{T}$ and $\hat{U}=A U^* ( \Lambda^*)^{-1}$. The triangle's inequality yields
\begin{align*}
&\| U_{i\cdot} - v_{j} Q \|_2 
=  \| U_{i\cdot} Q^T - v_j \|_2 \geq -\| U_{i\cdot} Q^T - \hat{U}_{i\cdot} \|_2 +  \| \hat{U}_{i\cdot} - v_j \|_2, \\
&\| U_{i\cdot} - v_{z_i} Q \|_2 
=  \| U_{i\cdot} Q^T - v_{z_i} \|_2 \leq \| U_{i\cdot} Q^T - \hat{U}_{i\cdot} \|_2 +  \| \hat{U}_{i\cdot} - v_{z_i} \|_2.
\end{align*}
As a result,
\begin{align*}
\| U_{i\cdot} - v_{j} Q \|_2 - \|U_{i\cdot} - v_{z_i} Q \|_2 \geq  \| \hat{U}_{i\cdot} - v_j \|_2 -  \| \hat{U}_{i\cdot} - v_{z_i} \|_2 -2 \| U_{i\cdot} Q^T - \hat{U}_{i\cdot} \|_2.
\end{align*}
Since (\ref{ineq-lem-sbm-3-2}) forces
\begin{align*}
&\max_{i\in[n]} \| U_{i\cdot} Q^T - \hat{U}_{i\cdot} \|_2 = \| U \sgn( U^T U^* ) - A U^* ( \Lambda^*)^{-1} \|_{2\to\infty} \leq \rho/\sqrt{n}
\end{align*}
to hold with high probability, it suffices to show the existence of some constant $c'>0$ such that with high probability,
\begin{align*}
\min_{j\neq z_i} \| \hat{U}_{i\cdot} - v_j \|_2 -  \| \hat{U}_{i\cdot} - v_{z_i} \|_2 \geq  c'/\sqrt{n}, \qquad \forall i\in[n].
\end{align*}

Without loss of geneality, we assume from now on that $J_1=\{ 1,\cdots,n/3 \}$, $J_2= \{ n/3+1,\cdots,2n/3 \}$ and $J_3=\{ 2n/3+1,\cdots,n \}$. We are going to show that for some positive constants $c_1$ and $c_2$,
\begin{align}
\P (  \| \hat{U}_{1 \cdot} - v_2 \|_2 -  \| \hat{U}_{1\cdot} - v_{1} \|_2 \geq  c_1/\sqrt{n} ) \geq 1- n^{-1-c_2}.
\label{ineq-lem-sbm-3-3}
\end{align}
If this is true, then the desired result follows from symmetry and union bounds.

Note that $\| v_1\|_2=\|v_2\|_2$,
\begin{align*}
& 2 \hat{U}_{1\cdot} ( v_1-v_2 )^T =  \| \hat{U}_{1 \cdot} - v_2 \|_2^2 -  \| \hat{U}_{1\cdot} - v_{1} \|_2^2 \\
&=  ( \| \hat{U}_{1 \cdot} - v_2 \|_2 -  \| \hat{U}_{1\cdot} - v_{1} \|_2 ) ( \| \hat{U}_{1 \cdot} - v_2 \|_2 +  \| \hat{U}_{1\cdot} - v_{1} \|_2 ) \\
&\leq ( \| \hat{U}_{1 \cdot} - v_2 \|_2 -  \| \hat{U}_{1\cdot} - v_{1} \|_2 )( 2 \| \hat{U}_{1 \cdot} \|_2 + 2\| v_1 \|_2 ).
\end{align*}
and as a result,
\begin{align}
\| \hat{U}_{1 \cdot} - v_2 \|_2 -  \| \hat{U}_{1\cdot} - v_{1} \|_2  
\geq  \hat{U}_{1\cdot} ( v_1-v_2 )^T  / ( \| \hat{U}_{1 \cdot} \|_2 +  \| v_{1} \|_2 ).
\label{ineq-lem-sbm-3-4}
\end{align}

Now we work on the right hand side of (\ref{ineq-lem-sbm-3-4}). Since $\hat{U}_{1\cdot} = A_{1\cdot} U^* (\Lambda^*)^{-1} $ and $\Lambda^* = \lambda_2^* I_2 = \frac{(a-b) \log n }{ 3 } I_2$, we have
\begin{align*}
&\hat{U}_{11} = A_{1\cdot} U^*_{\cdot 1} /\lambda_2^* = \frac{ 3 }{ (a-b) \sqrt{2n} \log n } \left( 2 \sum_{i=1}^{n/3} A_{1i}
- \sum_{i=1+n/3}^{2n/3} A_{1i} - \sum_{i=1+2n/3}^{n} A_{1i} \right) ,\\
&\hat{U}_{12} = A_{1\cdot} U^*_{\cdot 2} /\lambda_2^* = \frac{ 3\sqrt{3} }{ (a-b) \sqrt{2n} \log n } \left( \sum_{i=1+n/3}^{2n/3} A_{1i} - \sum_{i=1+2n/3}^{n} A_{1i} \right) , \\
& \hat{U}_{1\cdot} ( v_1-v_2 )^T = \frac{3}{\sqrt{2n}} \hat{U}_{11} - \frac{\sqrt{3}}{\sqrt{2n}} \hat{U}_{12}=\frac{9}{(a-b)n\log n} \left(  \sum_{i=1}^{n/3} A_{1i}
- \sum_{i=1+n/3}^{2n/3} A_{1i}  \right).
\end{align*}

On the one hand, it is to show using Bernstein inequality that $\P ( \sum_{i=1}^{n}A_{1i} \geq c_3 \log n ) \leq n^{-1-c_4}$ holds for positive constants $c_3,c_4$. Hence
\begin{align}
\P ( \| \hat{U}_{1 \cdot} \|_2 +  \| v_{1} \|_2 \leq c_3'/\sqrt{n} ) \geq 1-n^{-1-c_4'}
\label{ineq-lem-sbm-3-5}
\end{align}
for some constants $c_3'$ and $c_4'$.

On the other hand, we define $m=2n/3$ and observe that $\{ A_{1i} \}_{i=1}^{m}$ are independent Bernoulli random variables with parameters
\begin{align*}
\P ( A_{1i}=1 ) = \begin{cases}
a \frac{\log n}{n} = \frac{2a}{3}(1+o(1)) \frac{\log m}{m}, & \mbox{ for } 1\leq i \leq m/2\\
b \frac{\log n}{n} = \frac{2b}{3}(1+o(1)) \frac{\log m}{m}, & \mbox{ for } 1+m/2\leq i \leq m
\end{cases}.
\end{align*}
Moreover, $\sqrt{a}-\sqrt{b}>\sqrt{3}$ gives $\sqrt{ 2a/3 } - \sqrt{ 2b/3 } > \sqrt{2}$. Based on all these, we can use Lemma \ref{lem-tail} to conclude that
\begin{align*}
\P \left( \sum_{i=1}^{m/2} A_{1i} - \sum_{i=1+m/2}^{m} A_{1i} \geq c_5 \log n
\right) \geq 1 - n ^{-1-c_6}
\end{align*}
holds for positive constants $c_5$ and $c_6$. Therefore,
\begin{align}
\P (  \hat{U}_{1\cdot} ( v_1-v_2 )^T \geq 9c_5 / [ (a-b) n ] ) \geq 1 - n ^{-1-c_6}.
\label{ineq-lem-sbm-3-6}
\end{align}

Finally, we obtain (\ref{ineq-lem-sbm-3-3}) from (\ref{ineq-lem-sbm-3-4})-(\ref{ineq-lem-sbm-3-6}) and complete the proof.

\end{proof}

\bibliographystyle{ims}
\bibliography{bib}

\end{document}

%% file: macros.tex
\usepackage{amsmath,amssymb,amsfonts,amsthm,bbm}
\usepackage{verbatim}
\usepackage{subcaption}
\usepackage{color}

\newcommand{\diag}{\mathrm{diag}}
\newcommand{\rank}{\mathrm{rank}}
\newcommand{\spann}{\mathrm{span}}
\newcommand{\sgn}{\mathrm{sgn}}
\newcommand{\Bern}{\mathrm{Bernoulli}}
\newcommand{\norm}[1]{\left\|{#1}\right\|}
\newcommand{\R}{\mathbb{R}}

\newcommand{\E}{\mathbb{E}}
\newcommand{\Z}{\mathbb{Z}}

\renewcommand{\P}{\mathbb{P}}
\newcommand{\veps}{\varepsilon}

\theoremstyle{plain}

\newtheorem{defn}{Definition}[section]
\newtheorem{lem}{Lemma}
\newtheorem{cor}{Corollary}[section]

\theoremstyle{definition}